\documentclass[12pt,final]{article}
\usepackage[english]{babel}
\usepackage[utf8]{inputenc}
\usepackage{subcaption}
\usepackage[T1]{fontenc}
\usepackage{lmodern}
\usepackage{amssymb,amsmath,amsfonts}
\usepackage{xspace}
\usepackage{graphicx}
\usepackage{hyperref}
\usepackage{url}
\usepackage{stmaryrd}
\usepackage[babel=true]{csquotes}
\usepackage{amsthm}
\usepackage{mathrsfs}
\usepackage{float}
\usepackage[margin=3.5cm,lmargin=2.5cm, rmargin=2.5cm]{geometry}
\usepackage{enumitem}
\usepackage{verbatim}
\usepackage[dvipsnames]{xcolor}

\usepackage{tikz}
\usetikzlibrary{calc,cd,patterns,decorations.pathreplacing}
\usepackage{pgfplots}

\theoremstyle{plain}
\newtheorem{thm}{Theorem}[section]
\newtheorem{lm}[thm]{Lemma}
\newtheorem{prop}[thm]{Proposition}
\newtheorem{cor}[thm]{Corollary}

\theoremstyle{definition}
\newtheorem{dfn}[thm]{Definition}
\newtheorem{rmk}[thm]{Remark}
\newtheorem{example}[thm]{Example}

\def\dpar#1#2{\frac{\partial #1}{\partial #2}}

\def\Z{\mathbb{Z}\xspace}
\def\R{\mathbb{R}\xspace}
\def\C{\mathbb{C}\xspace}

\def\P{\mathbb{P}\xspace}
\def\S{\mathbb{S}\xspace}
\def\T{\mathbb{T}\xspace}
\def\De{\Delta}

\newcommand{\om}{\omega}

\def\CP{\mathbb{CP}\xspace}

\def\blowup#1#2{\text{Bl}_{#1}^{#2}}

\def\family{fixed-$\S^1$ family\xspace}

\def\families{fixed-$\S^1$ families\xspace}
\def\Families{Fixed-$\S^1$ families\xspace}
\def\Hirzscaled#1{W_{#1}(\alpha,\beta)\xspace}
\def\omHirz#1{\omega_{\Hirzscaled{#1}}}

\def\polygonnzero{\De^{n,0}(\alpha,\beta)}
\def\polygonnone{\De^{n,1}(\alpha,\beta)}

\def\poly#1{\Delta_{#1}\xspace}

\hypersetup{pdfborder={0 0 0}, colorlinks=true, linkcolor={blue!80!black}, citecolor={red!85!black},urlcolor={green!50!black}}

\makeatletter
\newcommand{\myitem}[1]{%
\item[#1]\protected@edef\@currentlabel{#1}%
}
\makeatother

\newcommand\blfootnote[1]{%
  \begingroup
  \renewcommand\thefootnote{}\footnote{#1}%
  \addtocounter{footnote}{-1}%
  \endgroup
}

\pgfplotsset{compat=1.13} 
\begin{document}

\title{\vspace{-1cm}Semitoric families}

\author{Yohann Le Floch \qquad Joseph Palmer}

\maketitle

\vspace{-.5cm}

\begin{abstract}
Semitoric systems are a type of four-dimensional integrable system for which one of the integrals generates a global
$\mathbb{S}^1$-action; these systems were classified by Pelayo and V\~{u} Ng\d{o}c in terms of five symplectic invariants.
We introduce and study semitoric families, which are one-parameter families of integrable systems with a fixed $\mathbb{S}^1$-action that are semitoric for all but finitely many values of the parameter, with the goal of developing a strategy to find a semitoric system associated to a given partial list of semitoric invariants. We also enumerate the possible behaviors of such families at the parameter values for which they are not semitoric, providing examples illustrating nearly all possible behaviors, which describes the possible limits of semitoric systems with a fixed $\mathbb{S}^1$-action. Furthermore, we introduce natural notions of blowup and blowdown in this context, investigate how semitoric families behave under these operations, and use this to prove that each Hirzebruch surface admits a semitoric family with certain desirable invariants; these families are related to the semitoric minimal model program. Finally, we give several explicit semitoric families on the first and second Hirzebruch surfaces showcasing various possible behaviors of such families which include new semitoric systems.
\end{abstract}


\newpage

\noindent
{\bf Yohann Le Floch} \\
Institut de Recherche Math\'ematique avanc\'ee,\\
UMR 7501, Universit\'e de Strasbourg et CNRS,\\
7 rue Ren\'e Descartes,\\
67000 Strasbourg, France.\\
{\em E-mail:} \texttt{ylefloch@unistra.fr}\\


\noindent
{\bf Joseph Palmer}\\
University of Antwerp\\
Department of Mathematics and Computer Science\\
Middelheimlaan 1\\
B-2020 Antwerpen, Belgium\\
\indent \emph{and}\\
University of Illinois at Urbana-Champaign\\
Department of Mathematics\\
1409 W Green St\\
Urbana, IL 61801 USA\\
{\em E\--mail}: \texttt{jpalmer5@illinois.edu}

\newpage

\setcounter{tocdepth}{2}
\tableofcontents

\newpage

\section{Introduction}

\subsection{Motivation}

Semitoric systems form a special class of integrable systems in dimension four for which one of the integrals generates a global $\S^1$-action. They were introduced by V\~u Ng\d{o}c~\cite{VNpoly} following the work of Symington \cite{Sym} and constitute a generalization of four-dimensional toric systems. More precisely, a semitoric system $(M,\omega,F=(J,H))$ is a Liouville integrable system $(J,H)$ on a four-dimensional symplectic manifold $(M,\omega)$ such that $J$ is proper and is the momentum map for an effective Hamiltonian $\S^1$-action on $(M,\omega)$, and the momentum map $F$ only possesses mild singularities (all non-degenerate and not of hyperbolic type). One of the primary differences between toric and semitoric systems is that the latter admit so-called focus-focus singular points that do not appear in toric systems and which introduce monodromy in the natural affine structure of the image of the momentum map, as introduced by Duistermaat~\cite{Dui} and studied by many authors since. The symplectic classification of semitoric systems with generic critical fibers was obtained by Pelayo and V\~u Ng\d{o}c \cite{PVNinventiones,PVNacta}, and these systems have generated a lot of interest during the last decade, see for instance \cite{HSS, dullin-pelayo, HohPal, HSS-vertical,  LFP,  AloDulHoh, AHS, MirPreSol,  HohMeu, AloHoh21}; for recent nice surveys of semitoric systems and additional references see~\cite{VNSepe,AloHoh19}.

For compact toric systems, the image of the momentum map is a convex polytope \cite{Ati, GuiSte}. A celebrated theorem of Delzant \cite{Del} states that this polytope determines the toric system up to equivariant symplectomorphism, and that any Delzant (i.e.~rational, smooth, and convex) polytope is the image of the momentum map of some toric system. A remarkable feature of \cite{Del} is that it also provides an explicit procedure to construct the toric system associated with a given Delzant polytope via symplectic reduction of $\C^d$ by the action of a torus, where $d$ and the action are straightforward to determine given the polytope. The construction of a semitoric system from its five symplectic invariants in \cite{PVNacta}, however, is much more complicated and involves the delicate operation of symplectically gluing different local (or semi-local) normal forms; 
of course, this extra difficulty reflects the richer nature of these invariants and indeed the additional complexity inherent in semitoric systems.

Nevertheless, a very natural question is the following: can we find a natural and simple procedure to construct \emph{some} semitoric system given a subset of its five invariants, and placing no constraints on the remaining invariants? Specifically, this paper is motivated by the task of constructing explicit semitoric systems given two of the five symplectic invariants, namely the semitoric polygon and the number of focus-focus points. We package these two invariants together with the so-called height invariant into a single object which we call a marked semitoric polygon, see Section \ref{sec:semitoric}. 
One reason that we choose the polygon invariant to play a special role is that it is the only invariant which is a direct analogue of the complete invariant of toric integrable systems.
Additionally, such a construction could help to find explicit systems associated to a given \emph{semitoric helix}, an invariant which can be recovered from the semitoric polygon and was introduced in~\cite{KPP_min} and used in the classification of minimal semitoric systems therein.
This problem, of searching for a technique to easily construct a semitoric system from its invariants, is
the inverse of Problem 2.35 from~\cite{PelVNsteps}.

There are only a few fully explicit examples of semitoric systems on closed symplectic four-dimensional
manifolds in the literature: two on $\S^2\times\S^2$~\cite{SZ,LFP,AHS,HohPal}, one on $\C\P^2$~\cite{CDEW},
and one on $\S^2\times\S^2$ blown up four times~\cite{HohMeu} (this last example was obtained using
the techniques developed in the present paper). The examples on $\S^2\times\S^2$ are particularly relevant for motivating this work.
The first one is obtained by coupling two angular momenta in a non-trivial way, and has been introduced in \cite{SZ} and studied in \cite{LFP} and in \cite{AHS}; it is of the form $F_t=(J,H_t)\colon\S^2 \times \S^2\to\R^2$ where $H_t$ is a one-parameter family (in fact, a convex combination) of Hamiltonians, and displays either zero or one focus-focus singularities depending on the value of $t$. This system is semitoric except for two values of $t$ (when a singular point transitions between being of focus-focus type and elliptic-elliptic type). The second one was introduced in \cite{HohPal} and is a generalization of the former; it is again an integrable system on $\S^2 \times \S^2$, but this time the momentum map is of the form $F_{s_1,s_2}=(J,H_{s_1,s_2})$ where $H_{s_1,s_2}$ is a two-parameter family and for almost every choice of parameters the system is semitoric with either zero, one, or two focus-focus singularities. 
Both of these examples are briefly reviewed in Section~\ref{sec:knownex}.

There are several interesting things to note about these two examples. First, in each of them the manifold, the symplectic form, and the component of the momentum map generating the $\S^1$-action are fixed; only the second component of the momentum map varies. Second, they undergo Hamiltonian-Hopf bifurcations where a singular point transitions between being of elliptic-elliptic and focus-focus type, which
are well studied (see for instance~\cite{Montaldi_notes, vdM, COR03, COR03_corr, EfsCusSad}). Third, the limiting systems (i.e.~when the parameters $(s_1, s_2)$ or $t$ are in the boundary of the parameter space) are related to one of the semitoric invariants from the Pelayo-V\~{u} Ng\d{o}c classification.
More precisely, semitoric systems have associated to them a family of polygons~\cite{VNpoly} constructed from the affine
structure on the image of the momentum map, and in these examples there is a relationship between the polygons associated
to the limiting systems and the polygons associated to the system for intermediate values of the parameters (for which it has more focus-focus points).
This relationship is explained in detail in Section~\ref{sec:semitorictransfam}.
Note that there exists a similar example of such transition in the non-compact setting, obtained by coupling a spin and an harmonic oscillator, see~\cite[Section 6.2]{VNpoly}. Finally, in Dullin-Pelayo~\cite{dullin-pelayo}, the authors start with a semitoric system $(M,\omega,(J,H))$ and perturb only $H$ to produce hyperbolic singularities. Thus, for all these reasons, it seems very natural to study such families.

The idea of the present paper is to introduce a class of systems describing this transition, and to try to understand to what extent such a construction can be generalized. We introduce \emph{\families}, which are one-parameter families of integrable systems with momentum map of the form $F_t=(J,H_t)$, $0\leq t\leq 1$, where $J$ generates an $\S^1$\--action and $(t,p)\mapsto H_t(p)$ is smooth, and \emph{semitoric families}, which are \families that are semitoric for all but finitely many values of $t$, and we study some basic properties of both kinds of families, such as the possible behaviors of the system at the times when it is not semitoric. Furthermore, we introduce toric type blowups at completely elliptic points and the associated blowdowns and study how \families interact with these operations. As an application of these properties, we prove, starting from the coupled angular momenta system and using blowups and blowdowns, that a semitoric family with a specified marked semitoric polygon exists on the $n$-th Hirzebruch surface for each $n\in\Z_{\geq 0}$.
Finally, we give several \emph{explicit} semitoric families on the first and second Hirzebruch surfaces, which display some of the various possible behaviors of semitoric families.

In fact, there is a reason to be particularly interested in the specific semitoric polygons of these systems on Hirzebruch surfaces
and the blowup and blowdown operations on semitoric families.
In~\cite{KPP_min}, it is shown that every semitoric system can be produced by performing a sequence of toric type blowups on a semitoric system whose associated semitoric helix is one of seven types, types (1)-(7) in Theorem $1.3$ of the aforementioned paper, where some of the types are a single helix and some are a family depending on parameters. Note that the blowups introduced in~\cite{KPP_min} are defined as corner chops of the semitoric polygons, so they are only defined on semitoric systems and thus are less general than the ones introduced here.
The present paper is part of a program whose goal is to understand, as explicitly as possible, the systems with these minimal helices and the operation of a toric type blowup, and to use this knowledge to better understand all semitoric systems.
More specifically, one could hope that nearly every semitoric system can be obtained explicitly in a semitoric transition family with a relatively simple form (as considered in this paper); this is optimistic, but even if this construction does not work in general it is of interest to be able to produce a wide variety of examples.

\begin{figure} 
\begin{center} 
\def\dashlength{0.08}
\begin{subfigure}[b]{.3\linewidth} 
\centering 
\begin{tikzpicture} 
\draw (0,-2)--(0,2); 
\draw (-2,0)--(2,0); 
\draw [thick,->] (0,0)--(0,1); 
\draw [thick,->] (0,0)--(-1,-2); 
\draw (0,1) node [above right] {$v_0$}; 
\draw (-1,-2) node [left] {$v_1$}; 
\draw (-1,-\dashlength) -- (-1,\dashlength);
\draw (-\dashlength,1) -- (\dashlength,1);
\draw (1,-\dashlength) -- (1,\dashlength);
\draw (-\dashlength,-1) -- (\dashlength,-1);
\begin{scope}[yshift = -4.5cm] 
\filldraw[draw=black, fill=gray!60] (-2,0) -- (0,1) -- (2,0) -- cycle; 
\draw (0,0.4) node {$\times$}; 
\draw [dashed] (0,0.4) -- (0,1); 
\draw [transparent] (0,-.3) circle (1/16);
\end{scope} 
\end{tikzpicture} 
\caption{Minimal helix of type $(1)$.} 
\end{subfigure} 
\hspace{5pt} 
\begin{subfigure}[b]{.3\linewidth} 
\centering 
\def\const{2.5} 
\begin{tikzpicture} 
\draw (0,-2)--(0,2); 
\draw (-2,0)--(2,0); 
\draw [thick,->] (0,0)--(0,1); 
\draw [thick,->] (0,0)--(-1,-1); 
\draw (0,1) node [above right] {$v_0$}; 
\draw (-1,-1) node [left] {$v_1$}; 
\draw (-1,-\dashlength) -- (-1,\dashlength);
\draw (-\dashlength,1) -- (\dashlength,1);
\draw (1,-\dashlength) -- (1,\dashlength);
\draw (-\dashlength,-1) -- (\dashlength,-1);
\begin{scope}[yshift = -4.5cm] 
\filldraw[draw=black, fill=gray!60] (-1.5,0) -- (-0.5,1) -- (0.5,1) -- (1.5,0) -- cycle; 
\draw (-0.5,0.3) node {$\times$}; 
\draw [dashed] (-0.5,0.3) -- (-0.5,1); 
\draw (0.5,0.6) node {$\times$}; 
\draw [dashed] (0.5,0.6) -- (0.5,1); 
\draw [transparent] (0,-.3) circle (1/16);
\end{scope} 
\end{tikzpicture} 
\caption{Minimal helix of type $(2)$.} 
\end{subfigure} 
\hspace{5pt} 
\begin{subfigure}[b]{.3\linewidth} 
\centering 
\def\const{2.5} 
\begin{tikzpicture} 
\draw (0,-2)--(0,2); 
\draw (-2,0)--(2,0); 
\draw [thick,->] (0,0)--(0,1); 
\draw [thick,->] (0,0)--(-1,-3); 
\draw [thick,->] (0,0)--(0,-1); 
\draw (0,1) node [above right] {$v_0$}; 
\draw (0,-1) node [below right] {$v_2$}; 
\draw (-1,-3) node [above right] {$v_1$}; 
\draw[decoration={brace,raise=5pt, amplitude = 5pt},decorate] (-1,-3) -- (-1,-0.05);
\draw (-1.3,-1.5) node [left] {$k$}; 
\draw (-1,-\dashlength) -- (-1,\dashlength);
\draw (-\dashlength,1) -- (\dashlength,1);
\draw (1,-\dashlength) -- (1,\dashlength);
\draw (-\dashlength,-1) -- (\dashlength,-1);
\begin{scope}[yshift = -4.5cm] 
\filldraw[draw=black, fill=gray!60] (-2.5,0) -- (-1.5,1) -- (-0.5,1) -- (2.5,0) -- cycle; 
\draw (-1.5,0.5) node {$\times$}; 
\draw [dashed] (-1.5,0.5) -- (-1.5,1); 
\draw [transparent] (0,-.3) circle (1/16);
\end{scope} 
\end{tikzpicture} 
\caption{Minimal helix of type $(3)$.} 
\end{subfigure} 
\end{center} 
\caption{A visual representation of the minimal semitoric helices of types (1), (2), and (3) from~\cite{KPP_min} and one example of a semitoric polygon having each helix. See~\cite[Theorem 2.4]{KPP_min} for notation and the exact values of the integral vectors $v_i$.}
\label{fig:minimal} 
\end{figure}
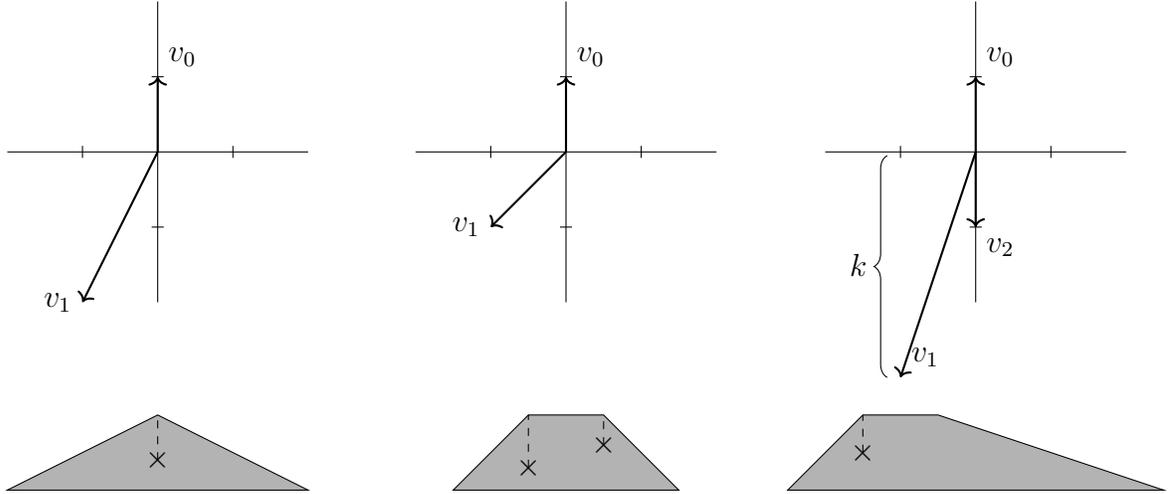

The minimal helices of types (1), (2), and (3) are shown in Figure~\ref{fig:minimal}, along with related marked semitoric polygons.
The minimal helices of type (3) depend on a single parameter $k\in\Z\setminus\{\pm 2\}$, and the coupled angular momenta system is minimal of type (3) with $k=1$. The new systems discussed in this paper on the $n$-th Hirzebruch surface have helix which is minimal of type (3) with parameter $k=-n+1$, hence we have every helix of type (3) with $k\leq 1$. In fact, one can construct the systems with $k>1$ from these, so this completes the search for all minimal helices of type (3).
There is only one minimal helix of type (2), and
the example from~\cite{HohPal} has this as its helix. 
Note that there is also a unique helix of type (1); an example with this helix on $\mathbb{CP}^2$ can be found in \cite{CDEW}. Thus, an example has been obtained for all minimal semitoric helices of type (1), (2), and (3).
Additionally, many semitoric polygons are associated to each helix, and we use scaled Hirzebruch surfaces to obtain infinitely many different semitoric 
polygons associated to the minimal helices of type (3).

\begin{rmk} Semitoric systems have a close relationship with several other interesting objects. First, semitoric systems are a particular class of the almost toric systems investigated in~\cite{Sym, LeuSym}. Moreover, since the first component $J$ of the momentum map of a semitoric system generates an effective Hamiltonian $\S^1$-action, by forgetting the second component $H$, we can see a semitoric system as a particular instance of a four-dimensional Hamiltonian $\S^1$-space, as studied in \cite{karshon}; the relation between these Hamiltonian $\S^1$-spaces and semitoric systems was studied in~\cite{HSS}. Natural higher dimensional analogues of Hamiltonian $\S^1$-spaces are complexity one spaces, which come with an effective Hamiltonian action of a torus of dimension $n-1$ on a symplectic manifold of dimension $2n$, and have been studied in~\cite{KarTol01, KarTol03, GodSou, GodSou_errata, KarTol14, KarTol20, SabSep} for instance. Note that several generalizations of semitoric systems have been studied: in higher dimension~\cite{Wac, wacheux-local-preprint}, when $J$ is not necessarily proper (but $F$ is)~\cite{PelRatVN}, when the restriction on singularities is relaxed to include some hyperbolic singularities and simple degenerate singularities~\cite{HohPal21} and other natural generalizations~\cite{PelRatVN_connectivity,HSS-vertical,RatWacZun}. \end{rmk}

\begin{rmk}
As already mentioned, in this paper we are only interested in integrable systems defined on closed symplectic manifolds. But the theory of semitoric systems includes the non-compact case, which is relevant from the physical point of view; for instance the celebrated Jaynes-Cummings model~\cite{jaynes-cummings,cummings1965stimulated,shore-knight,babelon-spin,BabDou} from quantum optics, in its simplest form, is a semitoric system defined on $\S^2 \times \R^2$. The spherical pendulum~\cite{CusBat, Dui}, which possesses one focus-focus singularity, fails to be a semitoric system only because the component of its momentum map generating an $\S^1$-action is not proper; however its momentum map itself is proper, so it belongs to the class of systems investigated in~\cite{PelRatVN}. The Champagne bottle~\cite{child}, with phase space $T^* \R^2$, is another instance of a system with one focus-focus singularity which is not strictly semitoric for the same reason
as the spherical pendulum. In fact, the introduction of semitoric systems in \cite{VNpoly} was motivated by such physical examples, in which quantum monodromy~\cite{cushman-duist-pendulum} was observed as a consequence of classical monodromy. The Arnold-Liouville-Mineur theorem~\cite{Min47, Dui, BatSni, HofZeh, GuiSte_book, Zung18} (see also \cite{Zung96} for a generalization of this theorem for singularities) states that if $(M^{2n},\omega,F)$ is an integrable system on a compact connected symplectic manifold, near any regular fiber of $F$, the system is symplectically equivalent to the standard system with linear dynamics on $T^* \mathbb{T}^n$. This induces local coordinates on $M$ which are called action-angle variables, and the action variables are the components of a local toric momentum map. Monodromy is one of the topological obstructions to the construction of global action variables~\cite{Dui}, and is induced for instance by the presence of focus-focus singularities, as shown in~\cite{Zou, Mat96, Zung97}. More precisely, the action variables induce a natural
 integral affine structure on the set of regular values of the momentum map, $B_{\mathrm{reg}}$, and the monodromy measures
 how the affine structure changes when transported around a non-trivial loop in the fundamental group of $B_{\mathrm{reg}}$.
 Besides the references already mentioned above, monodromy (and its quantum counterpart) in specific physical systems and in general contexts has been studied by many authors, such as~\cite{CusKno, joyeux-sado-tenny, Zung03, EJS, DulGiaCus, WaaDulRic, DDB07, DDSZ, Tar, BKK, EfsMar17, MBE20}, see also~\cite{MarBroEfs} for a recent survey on this topic and additional references.
\end{rmk}

\subsection{\texorpdfstring{\Families}{Fixed-S1 families} and semitoric families}

We will be interested in families of semitoric systems, hence we start by recalling the definition of these systems.
\begin{dfn}
 A four-dimensional integrable system $(M,\om,F=(J,H))$ is \emph{semitoric} if
\begin{enumerate}[noitemsep]
\item $J$ is proper,
\item $J$ is the momentum map for an effective Hamiltonian $\S^1$-action,
\item all singular points of $F$ are non-degenerate and have no components of hyperbolic type.
\end{enumerate}
\end{dfn}
This definition will become clearer once the singularities of four-dimensional integrable systems have been discussed, see Section \ref{subsect:sing_dim4} and Definition \ref{def:semitoric}.

Now we introduce the main objects of study in this paper:
\begin{dfn}\label{def:family}
A \emph{\family} is a family of integrable systems $(M,\om,F_t)$, $0\leq t\leq 1$,
on a four-dimensional manifold $M$ such that $F_t = (J,H_t)$ where $H_t = H(t,\cdot)$ and $H: [0,1] \times M \to \R$ is smooth.
If additionally there exist $k \in \Z_{\geq 0}$ and $t_1, \ldots, t_k \in [0,1]$ such that $(M,\om,F_t)$ is semitoric if and only if $t \notin \{t_1, \ldots, t_k\}$,
then we call $(M,\om,F_t)$, $0\leq t\leq 1$, a \emph{semitoric family} with \emph{degenerate times} $t_1, \ldots, t_k$. Here we adopt the convention that $k=0$ means that there are no degenerate times (hence the system is semitoric for every $t \in [0,1]$).
\end{dfn}

Here, by a slight abuse of language, we use \emph{time} to refer to the deformation parameter $t$,
even though it is not associated with any dynamics.

\begin{rmk}\label{rmk:HSS}
 Notice that in a \family the underlying Hamiltonian $\S^1$\--manifold coming
 from the $\S^1$\--action generated by $J$ does not depend on $t$, so a semitoric family
 may be thought of as a smooth family of smooth functions $H_t$ on a fixed Hamiltonian $\S^1$\--manifold (note that not every such family $H_t$, $0 \leq t \leq 1$, will work).
 Compact semitoric systems are viewed as Hamiltonian $\S^1$\--manifolds by forgetting
 the other Hamiltonian in~\cite{HSS}. 
 The moduli space of semitoric
 systems, $\mathcal{M}_{\mathrm{ST}}$, was
 studied in~\cite{Pal17} and in particular
 endowed with a natural topology. 
 A semitoric family can almost be thought of as a path
 in $\mathcal{M}_{\mathrm{ST}}$, except that the system is not semitoric at the degenerate times. Therefore, one can think of a semitoric family instead as a path in the closure
 of $\mathcal{M}_{\mathrm{ST}}$, which passes
 through the walls separating different connected components
 of $\mathcal{M}_{\mathrm{ST}}$.
\end{rmk}

One of the topics we study in this paper is listing the possible behaviors of such families at degenerate times.
Among these possibilities, we are particularly interested in the scenario in which a single elliptic-elliptic point undergoes a Hamiltonian-Hopf bifurcation and becomes focus-focus after a
degenerate time (case \ref{item:HHbif} in Section~\ref{sec:semitoricfam-degen}). This  motivates the following definition, where the notion of non-degenerate singular points and their classification is as reviewed in Section~\ref{sec:singularpoints}.

\begin{dfn}
\label{dfn:transition_fam}
 A \emph{semitoric transition family} with \emph{transition point} $p \in M$ and \emph{transition times} $t^-,t^+\in (0,1)$, $t^-<t^+$, is a semitoric family $(M,\om,(J,H_t))$, $0\leq t \leq 1$, 
 with degenerate times $t^-$ and $t^+$, such that
\begin{itemize}[itemsep=0pt]
 \item for $t<t^-$ and $t>t^+$ the point $p$ is singular of elliptic-elliptic type,
 \item for $t^-<t<t^+$, the point $p$ is singular of focus-focus type,
 \item for $t=t^-$ and $t=t^+$ there are no degenerate singular points in $M\setminus\{p\}$,
 \item if $p$ is a maximum (respectively minimum) of $(H_0)_{|J^{-1}(J(p))}$ then $p$ is a minimum (respectively maximum)
  of $(H_1)_{|J^{-1}(J(p))}$.
\end{itemize}
If the number of focus-focus points of $(J,H_t)$ for $t \in (t^-,t^+)$ equals $k \geq 1$, we call it a \emph{semitoric $k$-transition family}.
\end{dfn}

Note that in such a family $p$ must be degenerate at $t=t^-$ and $t=t^+$ (see Lemma~\ref{lem:HPdegen} and~\cite[Proposition 2.8]{HohPal}); moreover, by the results contained in Section \ref{sec:firstprop}, the number of rank zero points of the system is the same for every value of $t$, and a point in $M\setminus\{p\}$ which is of focus-focus type for some value of the parameter stays focus-focus for all $t$. Definition \ref{dfn:transition_fam} is illustrated in Figure~\ref{fig:1transition}.

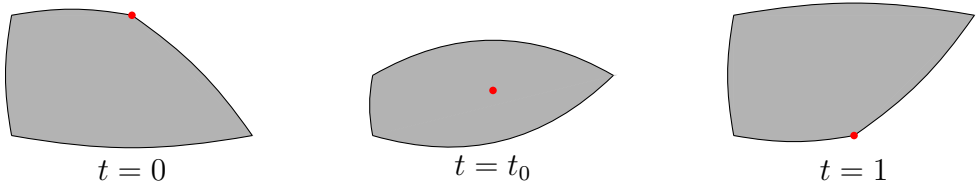
\begin{figure}
\begin{center}
\begin{tikzpicture}[scale=.80]
\filldraw[draw=gray!60,fill=gray!60] (0,0) node[anchor=north,color=black]{}
  -- (0,2) node[anchor=south,color=black]{}
  -- (2,2) node[anchor=south,color=black]{}
  -- (4,0) node[anchor=north,color=black]{}
  -- cycle;

\node[label={$t=0$}] at (2,-1.1){};

\draw[bend left=10, fill = gray!60] (0,0) to (0,2);
\draw[bend left=10, fill = gray!60] (0,2) to (2,2);
\draw[bend left=10, fill = gray!60] (2,2) to (4,0);
\draw[bend left=10, fill = gray!60] (4,0) to (0,0);

\fill[red] (2,2) circle (1/16);

\filldraw[draw=gray!60,fill=gray!60] (6,0) node[anchor=north,color=black]{}
  -- (6,1) node[anchor=south,color=black]{}
  -- (10,1) node[anchor=south,color=black]{}
  -- cycle;  

\draw[bend left=10, fill = gray!60] (6,0) to (6,1);
\draw[bend left, fill = gray!60] (6,1) to (10,1);
\draw[bend left, fill = gray!60] (10,1) to (6,0);

\fill[red] (8,0.75) circle (1/16); 

\node[label={$t=t_0$}] at (8,-1.1){};

\filldraw[draw=gray!60,fill=gray!60] (12,0) node[anchor=north,color=black]{}
  -- (12,2) node[anchor=south,color=black]{}
  -- (16,2) node[anchor=south,color=black]{}
  -- (14,0) node[anchor=north,color=black]{}
  -- cycle;  

\draw[bend left=10, fill = gray!60] (12,0) to (12,2);
\draw[bend left=10, fill = gray!60] (12,2) to (16,2);
\draw[bend left=10, fill = gray!60] (16,2) to (14,0);
\draw[bend left=10, fill = gray!60] (14,0) to (12,0);

\fill[red] (14,0) circle (1/16); 

\node[label={$t=1$}] at (14,-1.1){};

\end{tikzpicture}
\end{center}
\caption{The momentum map image of a semitoric 1-transition family with the image of the transition
 point indicated in red, where $t^-<t_0<t^+$.
 As $t$ increases from $0$ to $1$ the system starts with zero focus-focus points,
and then one elliptic-elliptic point transitions into being focus-focus (becoming degenerate for $t=t^-$),
and finally that point transitions back into being elliptic-elliptic (again becoming degenerate for $t=t^+$).}
\label{fig:1transition}
\end{figure}

\subsection{Main results}

The first goal of this paper is to prove various foundational facts about \families, semitoric families, and semitoric transition families related to the possible $\mathbb{S}^1$-actions, the possible systems at degenerate times, and the behavior of these systems as the
parameter varies. 
These results can be found in Section~\ref{sec:firstprop}.
For instance, we prove some restrictions on how the number, and position, of rank zero singular points can change with the parameter $t$
and we prove that, even at degenerate times, semitoric families do not admit singular points of hyperbolic type (which implies that at each degenerate time, the system has at least one degenerate singular point). Note, however, that a \family which is not a semitoric family may display hyperbolic singularities; we give an example of such a family in Section \ref{subsect:triangle}. 
In Section~\ref{sec:semitoricfam-degen}, we describe the possible scenarios that can occur for a semitoric family in a time interval containing a degenerate time, and give explicit examples of many systems which exhibit these behaviors, some of which are introduced in this paper.
Taking the special case in which the only degenerate time is $t_1=1$, Section~\ref{sec:semitoricfam-degen} describes the elements of the closure of the set of semitoric systems inside the set of all integrable systems $(M,\om,F)$ on a fixed symplectic manifold which have the momentum map of a fixed Hamiltonian $\S^1$-action as the first component of $F$; this is closely related to Problem 2.45 from~\cite{PelVNsteps}, which asks for the closure of the set of semitoric systems in the set of all smooth functions $F\colon M\to \R^2$.
We also show that some of the invariants of a semitoric system (the number of focus-focus points and unmarked semitoric polygon invariant) can only change as the system passes through a degenerate time.
Finally, in some cases we can see how these invariants change during the degenerate times;
indeed, by Lemma~\ref{lem:polygons_transition}, the family of semitoric polygons of 
a semitoric transition family with $t\in (t^-, t^+)$ is, roughly speaking, the union of the semitoric
polygons of the systems for $t=0$ and $t=1$.

After studying the general theory of such families, the remainder of the paper focuses on proving the existence of a semitoric transition family with prescribed semitoric polygons on each Hirzebruch surface on the one hand, and on the other hand on describing explicit examples of semitoric families on the first and second Hirzeburch surfaces. In the course of attaining the first goal, we define the blowup of a semitoric family at an elliptic-elliptic point and prove that it is still a semitoric family.

Near any of its completely elliptic points, an integrable system can be locally modeled as an $n$-torus acting on $\C^n$, and we use this local model to define a natural notion of blowup at such points for any integrable system. This is related to the $\S^1$-equivariant blowup on symplectic 4-manifolds from \cite[Section 6]{karshon}, but even in the case of a semitoric system our construction is more rigid since we require these blowups to be compatible with the local $\T^2$-action induced by the momentum map near the elliptic-elliptic point, and thus both components of the momentum map have to be taken into account. Furthermore, in the case that the system is semitoric we discuss the effect of such a blowup on the associated semitoric polygon\footnote{on the way showing that our definition agrees with the less general definition used in~\cite{KPP_min}, which only applies to semitoric systems.}, and how to perform a blowup on a \family to obtain another \family. These constructions present some difficulties, for instance because of the choices involved in this local normal form and because the blowups of elements in a \family are naturally defined in possibly different (albeit symplectomorphic) manifolds. We also define the notion of toric type blowdown, which is the inverse operation. The following is a collection of the results in Section~\ref{sec:blowups}; for more detailed statements see Proposition~\ref{prop:blowupdefined}, Lemma~\ref{lem:blowups_semitoric_chops}, Theorem~\ref{thm:blowups}, Proposition~\ref{prop:blowdowns}, and Corollary~\ref{cor:blowup_down_family}.

\begin{thm}\label{thm:blowupintro}
 The toric type blowup of an integrable system at a completely elliptic point and toric type blowdown as described in Section~\ref{sec:blowups} are well-defined.
 Moreover,
 \begin{itemize}[noitemsep]
  \item a toric type blowup (respectively blowdown) of a simple semitoric system corresponds to a corner chop (respectively unchop) of the semitoric polygon,
  \item if each member of a \family admits a toric type blowup (respectively blowdown) of the same size at the same point (respectively surface), then the new collection of systems formed by performing these operations can be naturally identified with a \family,
  \item under this identification, the toric type blowup or blowdown of a semitoric family is also a semitoric family with the same degenerate times.
 \end{itemize}
\end{thm}

In order to produce explicit semitoric systems, we use the fact that Lemma~\ref{lem:polygons_transition} describes the relationship between the semitoric polygons of systems in a semitoric transition family for different values of $t$, and thus gives
hints about what systems can occur for $t=0$ and $t=1$; once these two systems are found, we need to find a suitable one-parameter family interpolating between them. In our examples, the idea is to use the simplest example of such a family, namely a convex combination; of course, in general, this may not be sufficient.

For $n\in\Z_{\geq 0}$ let $(\Hirzscaled{n},\om_{\Hirzscaled{n}})$ denote the 
$n$-th Hirzebruch surface scaled by $\alpha, \beta \in \R_{> 0}$ with its standard symplectic form;
these symplectic manifolds are reviewed in Section~\ref{sec:hirzdef}.
Recall that the marked semitoric polygon is an invariant of semitoric systems which is essentially a combination of the number of focus-focus points, polygon, and height invariants from~\cite{PVNinventiones} and that we describe in Section~\ref{sec:semitoric}.
Making use of Theorem~\ref{thm:blowupintro} and other results we prove about semitoric families and toric type blowups and blowdowns, we obtain
the following theorem by performing alternating toric type blowups and blowdowns on the coupled angular momenta system on $W_0\cong \S^2\times \S^2$, which is a semitoric 1-transition family (see Section~\ref{sec:knownex}), to produce semitoric 1-transition families on $\Hirzscaled{n}$. 
\begin{thm}\label{thm:Wk}
 For each $n\in\Z_{\geq 0}$ and $\alpha,\beta \in \R_{>0}$ there exists a semitoric 1-transition family
 on $\Hirzscaled{n}$ with degenerate times $t^-, t^+$ satisfying $0<t^-<1/2<t^+<1$ such that a representative
 of the marked semitoric polygon of the system for $t^-<t<t^+$ is shown in Figure~\ref{fig:thmWk}.
\end{thm}

\begin{figure}[h]
\begin{center}
\begin{tikzpicture}[scale=.80]
\filldraw[draw=black,fill=gray!60] (0,0) node[anchor=north,color=black]{$(0,0)$}
  -- (2,2) node[anchor=south,color=black]{$(\beta,\beta)$}
  -- (4,2) node[anchor=south,color=black]{$(\alpha+\beta,\beta)$}
  -- (10,0) node[anchor=north,color=black]{$(\alpha + n \beta,0)$}
  -- cycle;

\draw (2,1) node[] {$\times$};

\draw [dashed] (2,2) -- (2,1);

\end{tikzpicture}
\end{center}
\caption{A representative of the marked semitoric polygon associated to the system on $\Hirzscaled{n}$ in Theorem~\ref{thm:Wk}.}
\label{fig:thmWk}
\end{figure}
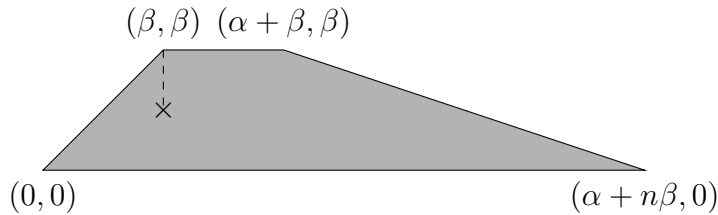

A more precise version of Theorem~\ref{thm:Wk} is stated as Theorem~\ref{thm:Wkdetailed}. 
Note that more is known about the system than its existence and marked
semitoric polygon, since Theorem~\ref{thm:Wkdetailed} is proved by constructing the systems in question via a specified sequence of toric type blowups and blowdowns on the coupled angular momenta system away from the fiber of $J$ containing the transition point.
For more details, see Remark~\ref{rmk:taylortwist}.

Theorem~\ref{thm:Wk} states that there exists a semitoric family on each Hirzebruch surface,
and gives some properties of that family, but it does not give explicit formulas for the elements of this family.
However, we also provide several families of integrable systems (depending on one or two parameters), which
include new explicit examples of semitoric systems.

\begin{thm}\label{thm:examplesintro}
We construct the following examples of families of semitoric systems:
\begin{itemize}[itemsep=0pt,leftmargin=15pt]
 \item the system from Equation~\eqref{system:W1_movingAB} is a semitoric $1$-transition family on $\Hirzscaled{1}$,
 \item the system from Equation~\eqref{system:W1_switch} is a semitoric family on $\Hirzscaled{1}$ with three degenerate times,
 \item the systems from Equations~\eqref{system:W2_transB} and~\eqref{system:W2_trans_C} are semitoric $1$-transition families on $\Hirzscaled{2}$,
 \item the system from Equation~\eqref{eqn:W2system2param} is a two-parameter family of systems which are semitoric for almost all choices of the parameters on $\Hirzscaled{2}$.
\end{itemize}
\end{thm}

\begin{rmk}
Note that $W_\ell(\alpha,\beta)$ and $W_k(\alpha',\beta')$ are diffeomorphic if and only if
the parity of $\ell$ and $k$ are equal, and in this case they are moreover symplectomorphic for some choice of $\alpha, \beta, \alpha', \beta'\in\R_{\geq 0}$ (this is well-known and proved for instance in \cite[Lemma 3]{Kar02}). Thus, the existence of semitoric 1-transition families on $W_0(\alpha,\beta)$ (from~\cite{SZ}) and
$W_1(\alpha,\beta)$ (from Theorem~\ref{thm:examplesintro}) implies that there exists a semitoric 1-transition family on all Hirzebruch surfaces
since the pull-back of a semitoric system by a symplectomorphism is an isomorphic semitoric system. These are not the semitoric families discussed in Theorem~\ref{thm:Wk} since the systems induced by the symplectomorphisms to $W_0(\alpha,\beta)$ (respectively $W_1(\alpha,\beta)$) will
all have the same semitoric polygon as the system on $W_0(\alpha,\beta)$  (respectively $W_1(\alpha,\beta)$)
described above, instead of the semitoric polygons
described in Theorem~\ref{thm:Wk}.
\end{rmk} 

For more precise statements of the results in Theorem~\ref{thm:examplesintro} see Sections~\ref{sec:W1_example} and~\ref{sec:W2examples}.
In particular, see Theorem~\ref{thm:W1_movingAB} for the system from Equation~\eqref{system:W1_movingAB},
Theorem~\ref{thm:W1_switch} for the system from Equation~\eqref{system:W1_switch},
Theorem~\ref{thm:W2_trans_B} for the system from Equation~\eqref{system:W2_transB},
and Theorem~\ref{thm:W2_trans_C} for the system from Equation~\eqref{system:W2_trans_C}.
Each of these systems displays unique behavior, which can be seen from images of their momentum maps
for various $t$.
The system from Theorem~\ref{thm:W1_movingAB} has fixed points which change as $t$ changes, which can only occur
on a fixed surface of the $\mathbb{S}^1$-action (see Figure~\ref{fig:moment_map_W1_noswitch}).
In the system from Theorem~\ref{thm:W1_switch} the fixed sphere of $J$ collapses to a single fiber for $t=1/2$, at which
time the images of the fixed points pass through one another (see Figure~\ref{fig:moment_map_W1_switch}).
The two systems on $\Hirzscaled{2}$ (from Theorems~\ref{thm:W2_trans_B} and~\ref{thm:W2_trans_C}) are packaged together into a two-parameter
family of systems (similar to~\cite{HohPal}) in Equation~\eqref{eqn:W2system2param} which transitions between having zero, one, or two focus-focus points depending on the parameters
(see Figure~\ref{fig:moment_map_W2_array}).

Note that in the semitoric transition families that we construct in Sections~\ref{sec:W1_example} and~\ref{sec:W2examples}, the limiting systems are of toric type, so in particular the momentum map of the underlying $\S^1$-action is the first component of some toric momentum map. The question of whether an effective Hamiltonian $\S^1$-action on a compact manifold can be extended in this way was completely resolved by Karshon~\cite[Section 5]{karshon}; this can be read on a graph associated with the action. Moreover, as explained in \cite{HSS}, given a compact semitoric system, Karshon's graph for the underlying $\S^1$-action can be read from the semitoric polygon invariant. So if one wants to construct a semitoric system with given polygon invariant, one may or may not be able to use toric type integrable systems as limiting systems, depending on this polygon. Building on the present paper, another semitoric transition family with limiting systems of toric type was constructed in \cite{HohMeu}.

\subsection{Structure of the paper}

In Section~\ref{sec:background} we describe some background material about integrable and semitoric systems and their singular points, and in
Section~\ref{sec:firstprop} we prove some basic facts about
\families and semitoric families; in particular, in Section~\ref{sec:semitoricfam-degen} we describe all possibilities of the behavior
of a semitoric family at the degenerate times.
In Section~\ref{sec:blowups} we describe toric type blowups and blowdowns
and how they interact with \families, which is useful in Section~\ref{sec:hirz},
in which we prove that every Hirzebruch
surface admits a semitoric family which has a specific semitoric polygon (Theorem~\ref{thm:Wk}).
In Sections~\ref{sec:W1_example} and~\ref{sec:W2examples} we give explicit examples 
on $\Hirzscaled{1}$ and $\Hirzscaled{2}$ (Theorems~\ref{thm:W1_movingAB}, \ref{thm:W1_switch}, \ref{thm:W2_trans_B},
and~\ref{thm:W2_trans_C}). In Appendix \ref{sect:appendix}, we compare the two-parameter family of systems
that we introduce in Section~\ref{sec:W2_s1s2} to the family of systems studied in~\cite{HohPal}.

\paragraph{Acknowledgements.} We would like to thank Konstantinos Efstathiou, Annelies De Meulenaere, and Chris Woodward for several useful discussions and comments which were very helpful when preparing this manuscript. Part of this work was completed while the first author was visiting the second author at Rutgers University. The first author also benefited from a PEPS Jeunes Chercheur-e-s 2017 grant. The second author was partially supported by the AMS-Simons travel grant (which also funded the first author's visit), and he would also like to thank the Institut des Hautes \'{E}tudes Scientifiques (IH\'{E}S) for inviting him to visit during the summer of 2017, during which some of the work for this project was completed. We thank two anonymous referees for their useful comments that helped improve the exposition of the paper. On a more personal note, we would like to express our gratitude to CK, $\aleph$, and $\nabla$ for their constant support. 


\section{Preliminaries on integrable and semitoric systems}
\label{sec:background}

\subsection{Integrable systems}

Let $(M,\om)$ be a symplectic manifold. Recall that given $f\in C^\infty(M,\R)$
there exists a unique vector field $X_f$ on $M$,
known as the \emph{Hamiltonian vector field} of $f$, 
such that $\omega(X_f, \cdot) + df = 0$.
We also define the \emph{Poisson bracket}
$\{\cdot,\cdot\}\colon C^\infty(M,\R)\times C^\infty(M,\R)\to C^\infty(M,\R)$
by $\{f,g\} = \om(X_f,X_g)$.
An \emph{integrable system} is a triple $(M,\om,F)$ where
$(M,\om)$ is a symplectic manifold of dimension $2n$ and
$F \in C^{\infty}(M,\R^n)$ has components $f_1,\ldots, f_n\colon M\to\R$ such that
\begin{enumerate}
 \item $\{f_i,f_j\}=0$ for all $i,j=1,\ldots, n$;
 \item $X_{f_1}(m),\ldots,X_{f_n}(m)$ are linearly independent for almost all $m\in M$.
\end{enumerate}

The function $F$ is called the \emph{momentum map} of the integrable system. A point $m \in M$ for which $X_{f_1}(m),\ldots,X_{f_n}(m)$ are linearly independent is called a \emph{regular point} of $F$.

\subsection{Singular points}
\label{sec:singularpoints}

Let $(M,\omega,F=(f_1,\ldots, f_n))$ be an integrable system. A \emph{singular point} of $F$ is a point $m \in M$ such that the family $(X_{f_1}(m), \ldots, X_{f_n}(m))$ is linearly dependent; the \emph{rank} of the singular point $m$ is the rank of this family. A singular point of rank zero is often called a \emph{fixed point}. As in Morse theory, there exists a notion of non-degenerate fixed point; given a fixed point $m \in M$, let $d^2 f_j(m)$ be the Hessian of $f_j$ at $m$, for $j=1, \ldots, n$ (we will often slightly abuse notation and use this for the matrix of this Hessian in any basis of $T_m M$).

\begin{dfn}[{\cite[Definition 1.23]{BolFom}}]
A fixed point $m$ is \emph{non-degenerate} if the Hessians $d^2 f_1(m)$, $\ldots$, $d^2 f_n(m)$ span a Cartan subalgebra of the Lie algebra of quadratic forms on $T_m M$.
\end{dfn}

The Lie algebra structure on the space of quadratic forms on $T_m M$ is defined so that the natural identification with $\mathfrak{sp}(2n,\R)$
induced by $\omega_m$ is a Lie algebra isomorphism.
There exists an analogous definition for singular points of all ranks, but we will not describe it here since it requires more background and notation (but we will give an equivalent definition in the case $n=2$ below); we refer the reader to \cite[Section 1.8.3]{BolFom}. The main idea is that this definition is conceived so that the following symplectic Morse lemma holds.

\begin{thm}[Eliasson normal form \cite{Eliasson-thesis,Eli90,miranda-zung}]
\label{thm:normalform)}
Let $m \in M$ be a non-degenerate singular point for $F$. Then there exist local symplectic coordinates $(x,y) = (x_1, \ldots, x_n, y_1, \ldots, y_n)$ on an open neighborhood $U \subset M$ of $m$ and $Q = (q_1, \ldots, q_n): U \to \R^n$ whose components $q_j$ are of the form
\begin{itemize}
\item $q_j(x,y) = \frac{1}{2}(x_j^2 + y_j^2)$ (elliptic),
\item $q_j(x,y) = x_j y_j$ (hyperbolic),
\item $q_j(x,y) = x_j y_{j+1} - x_{j+1} y_j$, $q_{j+1}(x,y) = x_j y_j + x_{j+1} y_{j+1}$ (focus-focus),
\item $q_j(x,y) = y_j$ (regular),
\end{itemize}
such that $m$ corresponds to $(x,y) = (0,0)$ and $\{f_j,q_k\} = 0$ for all $j,k \in \{1, \ldots, n\}$. Moreover, if there is no hyperbolic component, the following stronger result holds: there exists a local diffeomorphism $g: (\R^n,0) \to (\R^n,F(m))$ such that  
$ F(x,y) = (g \circ Q)(x,y)$ for every $(x,y) \in U$. 
\end{thm}

To our knowledge, in the literature there does not exist a unified complete proof of this theorem and many authors have contributed to it and its generalizations, see for instance \cite{Rus,Vey,DufMol,ColVey,ColVN,Mir,miranda-zung,MirVN,VNWac,Cha}; see also the discussion in \cite[Remark 4.16]{VNSepe}. The \emph{Williamson type} of the non-degenerate point $m$ is the quadruple of non-negative integers $(k,k_e,k_h,k_{ff})$ where $k$ (respectively $k_e$, $k_h$, $k_{ff}$) is the number of regular (respectively elliptic, hyperbolic, focus-focus) components of the above $Q$; note that $k \neq n$ and $k + k_e + k_h + 2 k_{ff} = n$. In the rest of the paper, a fixed point with $k_e = n$ will be called \emph{completely elliptic} (or elliptic-elliptic if $n=2$).

\begin{rmk}\label{rmk:gcycles}
At least for singularities with only elliptic or regular components, the
$i$-th component of the inverse $h = g^{-1}$ of $g$ from Theorem \ref{thm:normalform)} is given by the action integral
\[h^{(i)}(c) = \frac{1}{2\pi}\int_{\gamma_p^i}\alpha\] 
where $d\alpha=\om$ in a neighborhood of $F^{-1}(c)$ and $\{\gamma_c^1,\ldots, \gamma_c^n\}$ is any basis 
of $H^1(F^{-1}(c))$. This is well-known for regular points~\cite[Remark 3.38]{VNSepe}, and can be extended
to a neighborhood of a singular points with elliptic and regular components as well~\cite{miranda-zung}. 
\end{rmk}

\begin{rmk}
Note that the Williamson type of a singular point can be described independently of the normal form theorem, via the classification of Cartan subalgebras of $\mathfrak{sp}(2n,\R)$ \cite{Wil}. We only chose this order for exposition purposes. Note also that the description of singular points of a Hamiltonian invariant under the action of a compact Lie group has been studied, see for instance~\cite{MonRobSte}, leading to an equivariant version of Williamson's classification~\cite{MelDel93}. It is important to note that the Darboux theorem does not hold in this equivariant setting, since in this case there may not be a unique canonical symplectic form, as explained in~\cite{MelDel92}.
\end{rmk}

\subsection{Singularities of integrable systems in the case \texorpdfstring{$n=2$}{n=2}}
\label{subsect:sing_dim4}

In the rest of the paper, we will mostly be interested in the case $\dim M = 4$; in this case, we change our notation and consider an integrable system with momentum map $F = (J,H)$. We now describe in more detail the singular points.

\paragraph{Fixed points}

Let $m\in M$ be a fixed point. Choose any basis of $T_m M$ and consider the matrix $\Omega_m$ of the symplectic form $\omega_m$ and the Hessian matrices $d^2 J(m)$ and $d^2 H(m)$ in this basis. Let $\nu, \mu \in \R$. As a consequence of \cite[Proposition 1.2]{BolFom}, the characteristic polynomial of the matrix $A_{\nu,\mu} = \Omega_m^{-1} (\nu d^2 J(m) + \mu d^2 H(m))$ is of the form $X \mapsto \chi_{\nu,\mu}(X^2)$ for some quadratic polynomial $\chi_{\nu,\mu}$; we call $\chi_{\nu,\mu}$ the \emph{reduced characteristic polynomial} of $A_{\nu,\mu}$. Then $m$ is non-degenerate if and only if there exists $\nu, \mu \in \R$ such that $\chi_{\nu,\mu}$ has two distinct nonzero roots $\lambda_1, \lambda_2 \in \C$; furthermore, if this is case, then the type of $m$ is 
\begin{itemize}
\item elliptic-elliptic (Williamson type $(0,2,0,0)$) if $\lambda_1 < 0$ and $\lambda_2 < 0$, 
\item elliptic-hyperbolic (Williamson type $(0,1,1,0)$) if $\lambda_1 < 0$ and $\lambda_2 > 0$,
\item hyperbolic-hyperbolic (Williamson type $(0,0,2,0)$) if $\lambda_1 > 0$ and $\lambda_2 > 0$,
\item focus-focus (Williamson type $(0,0,0,1)$) if $\Im(\lambda_1) \neq 0$ and $\Im(\lambda_2) \neq 0$.
\end{itemize}

Note that if $m$ is non-degenerate, then the set of $(\nu, \mu) \in \R^2$ such that $\chi_{\nu,\mu}$ has two distinct roots is open and dense (and the signs of the roots do not depend on the choice of $(\nu, \mu)$ in this set, so in particular the type of a singular point is well-defined).

\paragraph{Rank one singular points}

Let $m \in M$ be a rank one singular point; then there exists $\nu, \mu \in \R$ such that $\nu dH(m) + \mu dJ(m) = 0$, and hence there exists a linear combination of $X_J$ and $X_H$ whose flow has a one-dimensional orbit through $m$. Let $L \subset T_m M$ be the tangent line of this orbit at $m$ and let $L^{\perp}$ be the symplectic orthogonal of $L$. Then $\Omega_m^{-1} \left( \nu d^2 H(m) + \mu d^2 J(m) \right)$ descends to an operator $A_{\nu,\mu}$ on the two-dimensional quotient $L^{\perp} \slash L$, and the eigenvalues of $A_{\nu,\mu}$ are of the form $\pm \lambda$ for $\lambda \in \C$ (again, as a consequence of \cite[Proposition 1.2]{BolFom}). 

\begin{dfn}[{\cite[Section 1.8.3]{BolFom}, see also \cite[Section 2.3]{HohPal}}]
\label{dfn:nondeg_rankone}
The rank one singular point $m$ is non-degenerate if and only if $A_{\nu,\mu}$ is an isomorphism of $L^{\perp} \slash L$. Moreover, the type of $m$ is 
\begin{itemize}
\item elliptic-transverse (Williamson type $(1,1,0,0)$) if the eigenvalues of $A_{\nu,\mu}$ are of the form $\pm i \alpha$, $\alpha \in \R$,
\item hyperbolic-transverse (Williamson type $(1,0,1,0)$) if the eigenvalues of $A_{\nu,\mu}$ are of the form $\pm \alpha$, $\alpha \in \R$,
\end{itemize}
and these are the only possibilities.
\end{dfn}

Now, assume that the Hamiltonian flow of $J$ is periodic; for $j \in \R$ a regular value of $J$, let $M_j^{\text{red}}$ be the symplectic reduction of $M$ at level $j$ by the $\S^1$-action generated by $J$. Since $\{J, H \} = 0$, the function $H$ descends to a function $H^{\text{red},j}$ on the surface $M_j^{\text{red}}$. If moreover $\S^1$ acts freely on a point $m \in M$ with $j = J(m)$, then $M_j^{\text{red}}$ is smooth in a neighbourhood of the image $[m]$ of $m$ in $M_j^{\text{red}}$, and so is $H^{\text{red},j}$. The next lemma follows from the proof of \cite[Lemma 2.4]{HohPal}~\footnote{the statement of~\cite[Lemma 2.4]{HohPal} contains a slight mistake: $\S^1$ must act freely on the point $m\in M$, but in~\cite[Lemma 2.4]{HohPal} it is only required that $m$ is a regular point of $J$.}, see also \cite[Definition 3]{ToZel}.

\begin{lm}
\label{lm:nondeg_rankone_red}
Let $m \in M$ be such that the $\S^1$-action generated by $J$ acts freely on $m$. Let $j=J(m)$.
Then $m$ is a rank one singular point of $F$ if and only if $[m]$ is a singular point of $H^{\text{red},j}$.
Moreover, $m$ is a non-degenerate singular point of $F$ if and only if $[m]$ is non-degenerate for $H^{\text{red},j}$ (in the Morse sense) and the type of $m$ is elliptic-transverse (respectively hyperbolic-transverse) if and only if $[m]$ is an elliptic (respectively hyperbolic) singular point of $H^{\text{red},j}$.
\end{lm}

Note that Lemma~\ref{lm:nondeg_rankone_red} implies that if one point in a fiber of $F$ is a rank one singular point then all points in that fiber must be rank one singular points.

\subsection{Semitoric systems and semitoric polygons}
\label{sec:semitoric}

We will be interested in a particular class of four-dimensional integrable systems, called semitoric systems.
These systems, defined in \cite{VNpoly}, are an example of the \emph{almost toric Lagrangian fibrations} introduced by Symington~\cite{Sym}.

\begin{dfn}\label{def:semitoric}
 A four-dimensional integrable system $(M,\om,F=(J,H))$ is \emph{semitoric} if
\begin{enumerate}[noitemsep]
\item\label{item:Jproper} $J$ is proper,
\item $J$ is the momentum map for an effective Hamiltonian $\S^1$-action,
\item all singular points of $F$ are non-degenerate and have no components of hyperbolic type ($k_h = 0$ in their Williamson type).
\end{enumerate}
\end{dfn}

This means that a semitoric system can only have elliptic-transverse, elliptic-elliptic, or focus-focus singular points. 
For this paper, we will mostly be interested in the case that $M$ is compact,
so the first condition in the above definition is automatically satisfied.

The systems that we consider in the present paper will all have at most one focus-focus point in each level set of $J$.
Semitoric systems with this property are called \emph{simple}.
Simple semitoric systems were classified in~\cite{PVNinventiones,PVNacta} in terms of five invariants: 
the number $m_f$ of focus-focus points (which is finite), the polygon invariant, the height invariant, the Taylor series invariant, and the twisting index invariant. We will only describe the ones that we will need in the rest of the paper. 
The semitoric classification was later extended to include all semitoric systems, simple or not, by
adapting the invariants~\cite{PPT}.

\paragraph{Marked semitoric polygons.}

In this section we describe an invariant of simple semitoric integrable systems which we call 
a marked semitoric polygon, and which encodes the data of the \emph{number of focus-focus points}, \emph{polygon}, and \emph{height}
invariants of the original classification in~\cite{PVNinventiones,PVNacta}.
This is similar to the notion of decorated polygon introduced in~\cite{VNSepe}, but the marked semitoric
polygon does not include data about the so-called \emph{twisting index} while the decorated polygon does.
The only remaining invariant involved in the complete symplectic classification of simple semitoric systems
is called the \emph{Taylor series invariant} and is not encoded in either the marked or decorated polygon invariants,
see Remark~\ref{rmk:taylor}.
A visual representation of a marked semitoric polygon is shown in Figure~\ref{fig:markedpoly}. 
Note that as in \cite{VNpoly}, we use the word \emph{polygon} to refer to a closed subset of $\R^2$ whose boundary is a continuous, piecewise linear curve with a finite number of vertices in any compact set, but in fact for a semitoric system on a compact manifold what we obtain is a polygon in the usual sense.

Let $\pi_j: \R^2 \to \R$ be the canonical projection to the $j$-th factor, $j=1,2$.
Let $\Delta \subset \R^2$ be a rational convex polygon (rational means that every edge is generated by an integral vector), let $s \in \Z_{\geq 0}$, let $\vec{c} = (c_1, \ldots, c_s) \in (\R^2)^s$, and let $\vec{\epsilon} = (\epsilon_1, \ldots, \epsilon_{s}) \in \{ -1,1 \}^{s}$.
We call the triple $\left(\Delta, \vec{c}, \vec{\epsilon} \right)$ a \emph{marked weighted polygon} if 
$c_1,\ldots, c_s\in\mathrm{int}(\Delta)$ and 
$ \pi_1(c_1) < \ldots < \pi_1(c_s)$.
Let
\begin{equation}\label{eqn:T} T = \begin{pmatrix} 1 & 0 \\ 1 & 1 \end{pmatrix} \in \text{SL}_2(\Z) \end{equation}
and let $\mathcal{T}$ be the subgroup of $GL_2(\Z) \ltimes \R^2$ consisting of the composition of $T^k$ for some $k \in \Z$ and a vertical translation (that is, the subgroup of integral affine transformations leaving the vertical direction invariant). The group $\mathcal{T}$ acts on the set of marked weighted polygons as follows; if $\left(\Delta, \vec{c} = (c_1, \ldots, c_s), \vec{\epsilon} = (\epsilon_1, \ldots, \epsilon_{s}) \right)$ is a marked weighted polygon and $\tau \in \mathcal{T}$, then
\begin{equation} \tau \cdot \left(\Delta, \vec{c}, \vec{\epsilon} \right) = \left(\tau(\Delta), \tau(\vec{c}), \vec{\epsilon} \right) \label{eq:T_action}\end{equation} 
where $\tau(\vec{c}) = \left( \tau(c_1), \ldots, \tau(c_s) \right)$. Note that this action preserves the convexity of the polygon. 

For $\lambda \in \R$, let $\ell_{\lambda} =\pi_1^{-1}(\lambda)$. Given $\lambda \in \R$ and $k \in \Z$, we fix an origin in $\ell_{\lambda}$ and define $t_{\ell_\lambda}^k: \R^2 \to \R^2$ as the identity on $\{x \leq \lambda\}$ and as $T^k$, relative to this origin, on $\{x \geq \lambda\}$ (note that $t_{\ell_\lambda}^k$ thus defined is independent of the choice of this origin). 
Moreover, for $\vec{u} = (u_1, \ldots, u_s) \in \{-1,1\}^s$ and $\vec{\lambda} = (\lambda_1, \ldots, \lambda_s)$, we define $t_{\vec{u},\vec{\lambda}} = t_{\ell_{\lambda_1}}^{u_1} \circ \ldots \circ t_{\ell_{\lambda_s}}^{u_s}$. 

The group $G_s = \{-1,1\}^s$ with componentwise multiplication $\vec{\epsilon'} * \vec{\epsilon} = ({\epsilon_1}' \epsilon_1, \ldots, {\epsilon_s}' \epsilon_s)$ acts on the set of all triples $(P,\vec{c},\vec{\epsilon})$ where $P \subset \R^2$ is a (possibly not convex) polygon,
$\vec{c}\in(\R^2)^s$, and $\vec{\epsilon}\in\{-1,1\}^s$ by 
\begin{equation} \vec{\epsilon'} \cdot (P,\vec{c},\vec{\epsilon}) = \left(t_{\vec{u},\vec{\lambda}}(P), t_{\vec{u},\vec{\lambda}}(\vec{c}), \vec{\epsilon'} * \vec{\epsilon} \right), \label{eq:Gs_action}\end{equation} 
where 
\begin{equation} \vec{u} = \left(\frac{\epsilon_1 - \epsilon_1 \epsilon_1'}{2}, \ldots, \frac{\epsilon_s - \epsilon_s \epsilon_s'}{2}\right), \qquad \vec{\lambda} = (\pi_1(c_1), \ldots, \pi_1(c_s)).\label{eq:u_lambda}\end{equation}
However, even if $P$ is convex
it is possible to have non-convex polygons in the $G_s$-orbit of $(P,\vec{c},\vec{\epsilon})$.
A marked weighted polygon $(\Delta,\vec{c},\vec{\epsilon})$ is called \emph{admissible} if its $G_s$-orbit contains
only convex polygons, and thus $G_s$ acts on admissible 
 marked weighted polygons by:
\[ \vec{\epsilon'} \cdot \left(\Delta, \vec{c}, \vec{\epsilon} \right) = \left(t_{\vec{u},\vec{\lambda}}(\Delta), t_{\vec{u},\vec{\lambda}}(\vec{c}), \vec{\epsilon'} * \vec{\epsilon} \right) \]
with $\vec{u}$ and $\vec{\lambda}$ as in Equation \eqref{eq:u_lambda}. Note that the actions of $\mathcal{T}$ and $G_s$ both preserve rationality of the polygon.

Let $q$ be a vertex of a rational polygon $\Delta$,
and let $u,v\in\Z^2$ be the primitive (minimal length) vectors directing the edges adjacent to $q$. Then 
we say that $q$ satisfies:
\begin{enumerate}[noitemsep]
\item the \emph{Delzant condition} if $\det(u,v) = 1$,
\item the \emph{hidden Delzant condition} if $\det(u,Tv) = 1$,
\item the \emph{fake condition} if $\det(u,Tv) = 0$.
\end{enumerate}
Given a marked weighted polygon $(\De,\vec{c},\vec{\epsilon})$ let
\begin{equation}\label{eqn:Lcuts}
 \mathcal{L}_{(\De,\vec{c},\vec{\epsilon})} = \cup_{j=1}^s\{(\pi_1(c_j),y)\mid\varepsilon_j y \geq \varepsilon_j \pi_2(c_j)\}
\end{equation}
so $\mathcal{L}_{(\De,\vec{c},\vec{\epsilon})}$ is the disjoint union of half lines emanating from each $c_j$ going up
if $\varepsilon_j =1$ and going down if $\varepsilon_j=-1$. The set $\mathcal{L}_{(\De,\vec{c},\vec{\epsilon})}$ is denoted
by dotted lines in the figures in this paper, see for instance Figure~\ref{fig:markedpoly}.
\begin{dfn}
A \emph{marked Delzant semitoric polygon} is the $G_s \times \mathcal{T}$-orbit of a marked weighted polygon $(\De,\vec{c},\vec{\epsilon})$
such that
\begin{enumerate}[noitemsep]
 \item each point in $\partial \De \cap \mathcal{L}_{(\De,\vec{c},\vec{\epsilon})}$ is a corner which satisfies either the 
  fake or hidden Delzant condition, in which case it is known as a \emph{fake} or \emph{hidden corner}, respectively,
 \item all other corners satisfy the Delzant condition, and are known as \emph{Delzant corners},
\end{enumerate}
in which case all marked weighted polygons in the orbit satisfy the same conditions and are admissible.
\end{dfn}


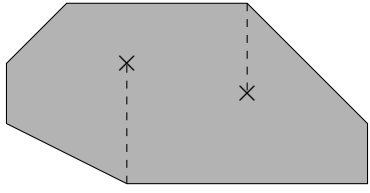
\begin{figure}
\begin{center}
\begin{tikzpicture}[scale=.8]
\filldraw[draw=black, fill=gray!60] (0,3) node[anchor=north,color=black]{}
  -- (0,3.5) node[anchor=south,color=black]{}
  -- (0.5,4) node[anchor=south,color=black]{}
  -- (4,4) node[anchor=north,color=black]{}
  -- (6,2) node[anchor=north,color=black]{}
  -- (6,1) node[anchor=north,color=black]{}
  -- (2,1) node[anchor=north,color=black]{}
  -- cycle;
\draw [dashed] (2,1) -- (2,3);
\draw (2,3) node[] {$\times$};
\draw [dashed] (4,4) -- (4,2.5);
\draw (4,2.5) node[] {$\times$};
\end{tikzpicture}
\end{center}
\caption{An example of a representative of a marked semitoric polygon associated to a system with 
two focus-focus singular points.}
\label{fig:markedpoly}
\end{figure}

We can also define an action of $G_s\times\mathcal{T}$ on pairs $((\De,\vec{c},\vec{\epsilon}),q)$, where $(\De,\vec{c},\vec{\epsilon})$
is a marked weighted polygon and $q\in\De$, by 
\begin{equation}\label{eqn:action_Deq}
(\vec{\epsilon'},\tau) \cdot ((\De,\vec{c},\vec{\epsilon}),q) = \left((\vec{\epsilon'},\tau) \cdot (\De,\vec{c},\vec{\epsilon}),t_{\vec{u},\vec{\lambda}}(q)\right)
\end{equation}
with $u, \lambda$ as in Equation \eqref{eq:u_lambda}. This will be useful later when we need to follow a given point through the orbit defining a marked semitoric polygon.

\paragraph{The marked semitoric polygon associated with a simple semitoric system.}

Now, coming back to the simple semitoric system $(M,\omega,F=(J,H))$, let $B = F(M)$ and let $B_{\text{reg}}$ be the set of regular values of $F$. Let $\vec{\epsilon} \in \{ -1,1 \}^{m_f}$; here we assume that $m_f > 0$ for simplicity, see Remark \ref{rmk:mf0} for the case $m_f = 0$. Let $(x_j,y_j)$, $j \in \{1, \ldots, m_f\}$, be the focus-focus values, with $x_1 < x_2 < \ldots < x_{m_f}$. Let $\ell_j^{\epsilon_j}$ be the vertical half-line starting from $(x_j,y_j)$ and going upwards if $\epsilon_j = 1$ and downwards if $\epsilon_j = -1$, and let 
\[ \ell^{\vec{\epsilon}} = \bigcup_{j=1}^{m_f} \ell_j^{\epsilon_j}. \] 

By \cite[Theorem 3.8]{VNpoly} there exists a homeomorphism $g_{\vec{\epsilon}}: B \to g_{\vec{\epsilon}}(B) \subset \R^2$ of the form 
\begin{equation} g_{\vec{\epsilon}}(x,y) = \left(x, g_{\vec{\epsilon}}^{(2)}(x,y)\right), \qquad \dpar{g_{\vec{\epsilon}}^{(2)}}{y} > 0, \label{eq:dev_map}\end{equation}
whose image $\poly{\vec{\epsilon}} = g_{\vec{\epsilon}}(B)$ is a rational convex polygon, and whose restriction ${g_{\vec{\epsilon}}}_{|B \setminus \ell^{\vec{\epsilon}}}$ is a diffeomorphism from $B \setminus \ell^{\vec{\epsilon}}$ to its image, sends the integral affine structure on $B_{\text{reg}} \setminus \ell^{\vec{\epsilon}}$ to the standard integral affine structure on $\R^2$, and extends to a smooth multivalued map from $B_{\text{reg}}$ to $\R^2$ such that for any $j \in \{1, \ldots, m_f\}$ and $c \in \ell_j^{\epsilon_j}$,
\[ \lim_{\substack{(x,y) \to c\\[.3mm]  x < x_j}} df(x,y) = T \left(\lim_{\substack{(x,y) \to c\\[.3mm]  x > x_j}} df(x,y)\right), \]  
where $T$ is as in Equation \eqref{eqn:T}.

The choice of $\vec{\epsilon}$ corresponds to a choice of cuts (downwards or upwards) at each focus-focus value. For example, in Figure~\ref{fig:markedpoly}
the first cut is downwards and the second is upwards, so $\vec{\epsilon}=(-1,1)$. The triple 
\[ \left( \Delta_{\vec{\epsilon}}, \vec{c} = (g_{\vec{\epsilon}}(x_1,y_1), \ldots, g_{\vec{\epsilon}}(x_{m_f},y_{m_f})), \vec{\epsilon} \right) \]
is an admissible marked weighted polygon. 

Given $\vec{\epsilon} \in \{ -1,1 \}^{m_f}$, the function $g_{\vec{\epsilon}}$ is unique up to left composition by an element of $\mathcal{T}$, and thus the polygon $\poly{\vec{\epsilon}}$ is unique up to the action of $\mathcal{T}$, and the same holds for the marked points. Indeed, the freedom in the construction of  $g_{\vec{\epsilon}}$ corresponds to the choice of a basis of action variables over $B_{\mathrm{reg}}$ as in Remark~\ref{rmk:gcycles}; in principle this is unique up to left composition by an element of $GL_2(\Z)$ and a translation, but because of the conditions in \eqref{eq:dev_map}, only elements in $\mathcal{T}$ are allowed. Moreover, choosing $\vec{\epsilon'} * \vec{\epsilon}$ instead of $\vec{\epsilon}$ (which corresponds to flipping the cut at $(x_j,y_j)$ if $\epsilon_j'=-1$ or taking the same cut otherwise) yields $\poly{\vec{\epsilon'}} = t_{\vec{u},\vec{\lambda}}(\poly{\vec{\epsilon}})$ where 
\[ \vec{u} = \left(\frac{\epsilon_1 - \epsilon_1 \epsilon_1'}{2}, \ldots, \frac{\epsilon_s - \epsilon_s \epsilon_s'}{2}\right), \quad \vec{\lambda} = (x_1, \ldots, x_{m_f}). \]
In other words, the freedom in this construction corresponds to the action of $G_{m_f} \times \mathcal{T}$ as described in Equations \eqref{eq:T_action} and \eqref{eq:Gs_action}. The \emph{marked semitoric polygon} $\poly{(M,\omega,F)}$ of $(M,\omega,F)$ is the orbit of any of the marked weighted polygons $\left( \Delta_{\vec{\epsilon}}, \vec{c} = (g_{\vec{\epsilon}}(x_1,y_1), \ldots, g_{\vec{\epsilon}}(x_{m_f},y_{m_f})), \vec{\epsilon} \right)$ constructed above under this group action.

The \emph{semitoric polygon} defined in~\cite[Section 4.3]{PVNinventiones} is similar to a marked semitoric polygon except that the height of the marks (the second coordinate of each of $\vec{c}_1, \ldots, \vec{c}_{m_f}$) is not included. We will call this an \emph{unmarked semitoric polygon} to emphasize the difference.

\begin{rmk}
\label{rmk:mf0}
The marked (and unmarked) semitoric polygon is still well-defined when $m_f = 0$ (the so-called \emph{toric type} systems), see for instance \cite[Definition 5.26]{VNSepe}. Of course, in that case, there is no marked point and no cut, and the only choice is the one of a global basis of action variables, giving the developing map $g$ (of the form $g(x,y) = (x,g^{(2)}(x,y))$ as above). This map is unique up to the action of $\mathcal{T}$, and the semitoric polygon is the orbit through $\mathcal{T}$ of $g(F(M))$. Furthermore, if $m_f=0$ then the marked and unmarked semitoric polygons are
the same. Of course, if the system is not only of toric type but toric (and $M$ is compact), this polygon is nothing but its Delzant polytope. As such, the construction of this invariant can be seen as a generalization of the convexity results for compact toric systems~\cite{Ati, GuiSte, Del}. Similar convexity results have been obtained in various contexts, such as for non compact toric manifolds~\cite{KarLer}, for toric orbifolds~\cite{LerTol}, for torus actions on some particular spaces called quasifolds to allow possibly non rational polytopes~\cite{Pra}, for presymplectic~\cite{RatZun} and log-symplectic~\cite{GLPR} toric manifolds. In~\cite{RatWacZun}, the authors prove several (positive and negative) convexity results for integrable systems possessing non-degenerate singularities with elliptic or focus-focus components only; in fact, the reader interested in convexity results in other contexts can find numerous references in~\cite{RatWacZun}, as well as in~\cite{GuiSja}, for instance.
\end{rmk}

\begin{rmk}
The idea behind the construction of the semitoric polygon in \cite{VNpoly} is as follows. Near any regular value of the momentum map $F$, one can apply the action-angle theorem to obtain a set of action variables of the form $(J,K)$ (or equivalently a local toric momentum map); then one wants to try to extend these actions to the whole set $B_{\text{reg}}$ of regular values of $F$. Of course, this cannot be done because of the non-trivial monodromy induced by the presence of focus-focus singularities~\cite{Zou, Mat96, Zung97}. The momentum map images of the focus-focus singularities are isolated singular values in $B$, which therefore obstruct $B_{\text{reg}}$ from being simply connected. The vertical cuts serve to create a simply connected set $B_{\text{reg}} \setminus \ell^{\vec{\epsilon}}$ from $B_{\text{reg}}$ in a way that is compatible with $J$; this set can then be endowed with global actions of the desired form, thus providing a toric momentum map on $F^{-1}(B_{\text{reg}} \setminus \ell^{\vec{\epsilon}})$ and its associated image, which is a polygon. The change of slope at a fake or hidden Delzant corner reflects the way that the affine structure changes when traveling along a loop in $B_{\mathrm{reg}}$ enclosing the corresponding focus-focus value, effectively taking into account the monodromy. 
\end{rmk}

\begin{rmk}
\label{rmk:taylor}
As already explained, in the present paper we are not concerned with the Taylor series invariant introduced in~\cite{San_Taylor}, which symplectically classifies the singular Lagrangian foliation in the neighbourhood of a focus-focus fiber. Computing explicitly this Taylor series invariant for concrete examples is a non-trivial task, and several authors calculated some or all terms of this invariant for various systems with focus-focus singularities (and other invariants when these systems were semitoric),
such as the spherical pendulum, the Jaynes-Cummings model, non-trivial couplings of two spins, and more~\cite{san-alvaro-spin,Dul13,LFP,AloDulHoh,AHS,jaume-thesis}. The Taylor series invariant
was generalized in~\cite{PT} to also apply to singular fibers that include multiple focus-focus singularities, and thus are topologically
multi-pinched tori. The semi-local \emph{topological}~\cite{Mat96, Zung97, Zung02} and \emph{smooth}~\cite{BolIzo} classifications of focus-focus fibers have also been investigated. These works are part of a long line of results on classifying integrable systems with two degrees of freedom, see for instance~\cite{Fom88, Fom91, BMF90, Bol94, BolFom94, BolFom94b, Bol95, BolMat96, BolMat98}; for a survey of various classification problems and more references, see~\cite{BolOsh}.
\end{rmk}

\subsection{Additional properties of semitoric systems}

We conclude by collecting some properties of semitoric systems that we will need in the rest of the paper. In such systems, the first component $J$ of the momentum map, which generates a $\S^1$-action, plays a special role. Following Karshon \cite{karshon}, we will call a value $j$ of $J$ (or a fixed point of $J$ in the fiber $J^{-1}(j)$) \emph{extremal} if $j$ is a global minimum or maximum of $J$, and \emph{interior} otherwise.

The following describes the fixed set of the $\S^1$-action generated by the first component, and is a consequence of~\cite[Proposition 3.4]{HSS} which is proved using~\cite[Lemma 2.1]{karshon}.

\begin{lm}\label{lem:dJzero}
 Let $(M,\om,(J,H))$ be a semitoric integrable system and let $M^{\S^1}$ denote the set of points
 fixed by the $\S^1$-action generated by the Hamiltonian flow of $J$, 
 equivalently $M^{\S^1}$ is the set of points $p$ such that $dJ(p)=0$.
 Then the connected components
 of $M^{\S^1}$ are all either points or diffeomorphic to $\S^2$; the latter case can only occur if the
 fixed surface is exactly the set $J^{-1}(j)$ where $j$ is an extremal value of $J$. 
\end{lm}

\begin{rmk}
 The fixed surface at an extremal value $j$ of $J$ described in Lemma~\ref{lem:dJzero} corresponds to a 
 component of the boundary of the momentum map image equal to an interval of the form
 $\{(j,y)\in\R^2\mid a\leq y \leq b\}$, and thus it is sometimes called
 a \emph{vertical wall}. Notice that the singular points at an extremal value of $J$ do not 
 necessarily form a vertical wall; they can also be isolated points.
\end{rmk}

We also want to understand the geometry of the reduced space $M_j^{\text{red}}$ for different values of $j$.

\begin{lm}
\label{lm:sphere}
Let $(M,\omega,(J,H))$ be a semitoric system. The reduced space $M_j^{\text{red}}$ at an interior level $j$ of $J$ is homeomorphic to a $2$-sphere and if $j$ is a regular value of $J$, then $M_j^{\text{red}}$ is diffeomorphic to a $2$-sphere.
If $j$ is an extremal value of $J$ then $M_j^{\text{red}}$ is either diffeomorphic to a $2$-sphere or a point.
\end{lm}

\begin{proof}

Suppose that $j$ is an interior value of $J$. Then the image of $H^{\text{red},j}$ is an interval $[a,b]$, and it suffices to prove that $(H^{\text{red},j})^{-1}(a)$ and $(H^{\text{red},j})^{-1}(b)$ are each a point and that for every $c\in(a,b)$, $(H^{\text{red},j})^{-1}(c)$ is homeomorphic to a circle. Let $c\in(a,b)$ and note
that $(H^{\text{red},j})^{-1}(c) = F^{-1}(j,c)/\S^1$. If $(j,c)$ is a regular value then $F^{-1}(j,c)$ is a $2$-torus and 
$(H^{\text{red},j})^{-1}(c)$ is homeomorphic to a circle. Since $(j,c)\in\mathrm{int}(F(M))$ the only other possibility is that
$(j,c)$ is the image of one or more focus-focus points so $F^{-1}(j,c)$ is a (multiply) pinched torus; since the $\S^1$-action associated to $J$
is in the direction of the vanishing cycle we again have that $(H^{\text{red},j})^{-1}(c)$ is homeomorphic to a circle. Next, notice
that in any case $(H^{\text{red},j})^{-1}(a)$ and $(H^{\text{red},j})^{-1}(b)$ are each a point; indeed, if $(j,a)$ is an elliptic-elliptic point then $F^{-1}(j,c)$ is already
a point and if $(j,a)$ is the image of an elliptic-regular point then $F^{-1}(j,c)$ is a circle and the flow of $J$ gives the
non-trivial $\S^1$-action on this space.

If $j$ is a regular value of $J$, then $H^{\text{red},j}$ is a Morse function with no hyperbolic points on the smooth closed surface $M_j^{\text{red}}$ (recall that the fibers of $J$ are compact); this is a consequence of Lemma~\ref{lm:nondeg_rankone_red} since the rank one singular points of $(J,H)$ are all non-degenerate of elliptic-transverse type. Hence $M_j^{\text{red}}$ is diffeomorphic to a $2$-sphere. 

If $j$ is an extremal value of $J$, then by Lemma~\ref{lem:dJzero}, $J^{-1}(j)$ is diffeomorphic to a sphere or a point, and $M_j^{\text{red}} = J^{-1}(j)$ since the action of $J$ on $J^{-1}(j)$ is trivial.
\end{proof}

\begin{lm}
\label{lm:fixed_semitoric}
Let $(M,\omega,F=(J,H))$ be a semitoric system, and let $m$ be a fixed point of $J$ which does not belong to a fixed sphere. 
Then $m$ is a fixed point of $F$.
\end{lm}

\begin{proof}
Since $dJ(m)=0$ we see that $m$ is a singular point; if $m$ is rank zero we are done, and otherwise it is rank one. By the normal form for non-degenerate rank one singular points (see Theorem~\ref{thm:normalform)}), there is a surface $\Sigma\subset M$ through $m$ consisting of rank one singular points. If $dH(m) \neq 0$ then $dH$ does not vanish is a neighborhood $U$ of $m$ and thus $dJ_{|\Sigma\cap U} = 0$; hence $m$ belongs to a fixed surface of $J$.
\end{proof}

\subsection{Examples on $\S^2\times \S^2$}
\label{sec:knownex}

\noindent We conclude this section by describing two examples of semitoric systems on $\S^2 \times \S^2$. Let $\om_{\mathbb{S}^2}$ be the standard symplectic form on $\mathbb{S}^2$ (the one giving total area $4\pi$)
and consider the manifold $\mathbb{S}^2\times\mathbb{S}^2$ with coordinates $(x_1, y_1, z_1, x_2, y_2, z_2)$ obtained from the
usual embedding of each sphere into $\R^3$.
Given $0<R_1<R_2$ and $0\leq t \leq 1$, the \emph{coupled angular momenta system} (introduced in~\cite{SZ} and further studied in~\cite{AHS, LFP}) is given by
$F_t = (J,H_t)$ where
\begin{equation}\label{eqn:coupledmomenta}
 J = R_1z_1 + R_2 z_2, \qquad H_t = (1-t)z_1 + t(x_1 x_2 + y_1 y_2 + z_1 z_2)
\end{equation}
on $\mathbb{S}^2\times\mathbb{S}^2$ with symplectic form $\om_{R_1,R_2} = R_1\om_{\S^2}\oplus R_2\om_{\S^2}$.
This is an example of a semitoric $1$-transition family, and a comparison of its momentum map for varying $t$ with the semitoric
polygons for the system with $t=1/2$ (which is semitoric with one focus-focus point) is shown in Figure~\ref{fig:coupledangular}.

\begin{figure}
\centering
\begin{subfigure}{.95\linewidth}
\centering
\includegraphics[width=450pt]{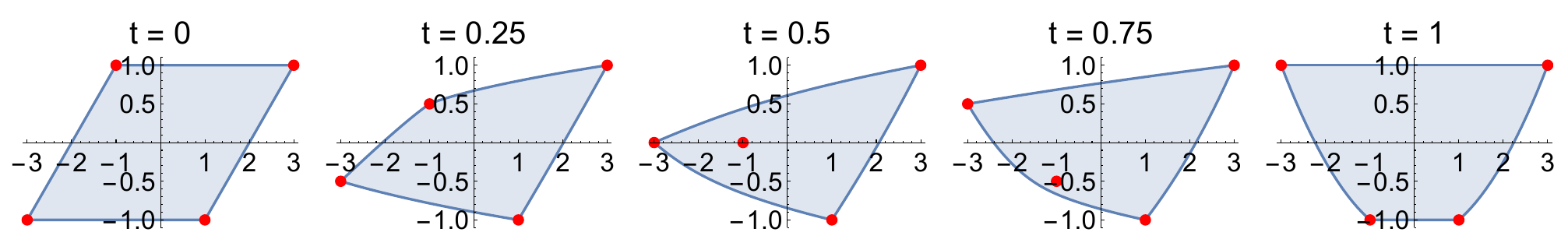}
\caption{The momentum map image for varying values of $t$ and $R_1 = 1$, $R_2 = 2$.}
\end{subfigure}
\begin{subfigure}{.95\linewidth}
\centering
\begin{tikzpicture}[scale=.75]
\filldraw[draw=black, fill=gray!60] (0,0) node[anchor=north,color=black]{}
  -- (2,2) node[anchor=south,color=black]{}
  -- (5,2) node[anchor=south,color=black]{}
  -- (3,0) node[anchor=north,color=black]{}
  -- cycle;
\draw [dashed] (2,1) -- (2,2);
\draw (2,1) node[] {$\times$};
\filldraw[draw=black, fill=gray!60] (8,2) node[anchor=north,color=black]{}
  -- (13,2) node[anchor=south,color=black]{}
  -- (11,0) node[anchor=north,color=black]{}
  -- (10,0) node[anchor=south,color=black]{}
  -- cycle;
\draw [dashed] (10,0) -- (10,1);
\draw (10,1) node[] {$\times$};
\end{tikzpicture}
\caption{Two representatives of the semitoric polygon of the system when $t=1/2$.}
\end{subfigure}
\caption{A comparison of the momentum map image and semitoric polygons for the coupled angular momenta system \eqref{eqn:coupledmomenta}.}
\label{fig:coupledangular}
\end{figure}
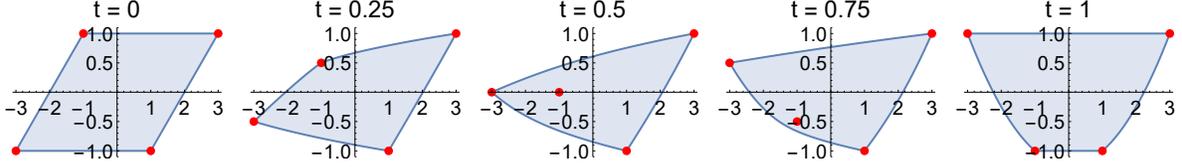
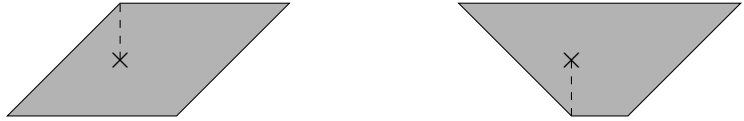
This system was generalized in~\cite{HohPal} to create a two-parameter family of systems $F_{s_1,s_2} = (J,H_{s_1,s_2})$ on the same symplectic manifold which has two focus-focus points for $s_1=s_2=\frac{1}{2}$ and transitions between systems with zero, one or two focus-focus points, depending on the parameters. Explicitly,
\begin{equation}\label{eqn:HPsystem}
 \begin{cases}
  J = R_1z_1 + R_2 z_2,\\
  H_{s_1,s_2} = (1-s_1)(1-s_2) z_1 + s_1 s_2 z_2      + s_1(1-s_2)(X + z_1 z_2) + s_2(1-s_1)(X - z_1 z_2),
 \end{cases}
\end{equation}
where $X = x_1 x_2 + y_1 y_2$ (this $X$ is the same as the one discussed in the beginning of Section~\ref{sec:semitorictransfam}).
The momentum map image is shown in Figure~\ref{fig:HPsystem}. The coupled angular momenta system is obtained by taking $(s_1,s_2) = (t,0)$. The parameter values with $R_1=R_2$ are usually excluded since for some values of $s_1, s_2$ this can cause degeneracies, but 	
in Section~\ref{sec:semitoricfam-degen} this constitutes an important example of the behavior of a semitoric family at degenerate times.

\begin{figure}
\centering
\begin{subfigure}{.95\linewidth}
\centering
\includegraphics[width=.94\linewidth]{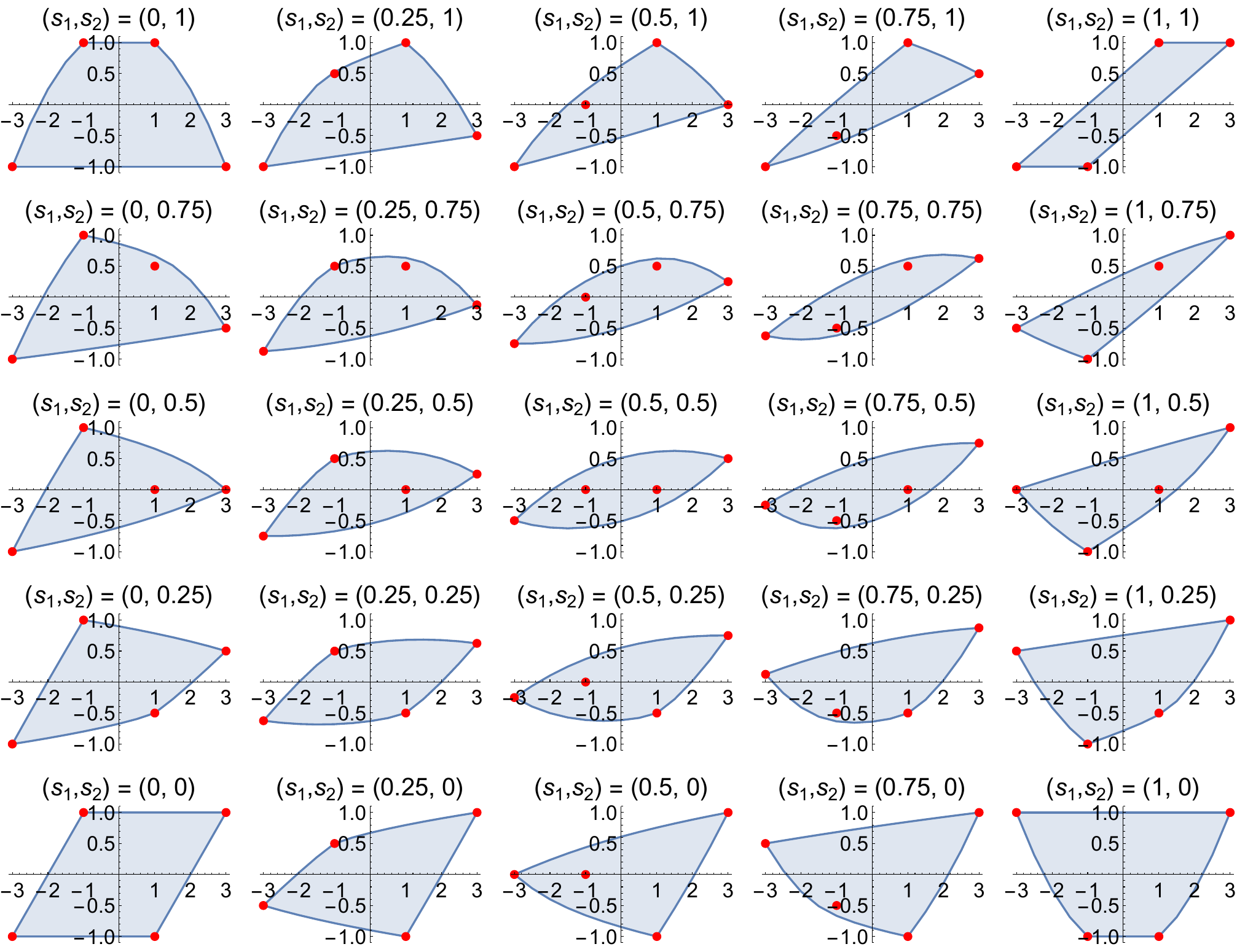}
\caption{The momentum map image of the system for various values of $s_1, s_2\in [0,1]$ with $R_1 = 1$, $R_2 = 2$.}
\end{subfigure}
\begin{subfigure}{.95\linewidth}
\centering
\begin{tikzpicture}[scale=.8]
\filldraw[draw=black, fill=gray!60] (0,0) node[anchor=north,color=black]{}
  -- (2,2) node[anchor=south,color=black]{}
  -- (5,2) node[anchor=south,color=black]{}
  -- (3,0) node[anchor=north,color=black]{}
  -- cycle;
\draw [dashed] (2,1) -- (2,2);
\draw (2,1) node[] {$\times$};
\draw [dashed] (3,1) -- (3,0);
\draw (3,1) node[] {$\times$};
\filldraw[draw=black, fill=gray!60] (6,2) node[anchor=north,color=black]{}
  -- (11,2) node[anchor=south,color=black]{}
  -- (9,0) node[anchor=north,color=black]{}
  -- (8,0) node[anchor=south,color=black]{}
  -- cycle;
\draw [dashed] (8,1) -- (8,0);
\draw (8,1) node[] {$\times$};
\draw [dashed] (9,0) -- (9,1);
\draw (9,1) node[] {$\times$};

\filldraw[draw=black, fill=gray!60] (6,2.5) node[anchor=north,color=black]{}
  -- (9,5.5) node[anchor=south,color=black]{}
  -- (11,5.5) node[anchor=north,color=black]{}
  -- (8,2.5) node[anchor=south,color=black]{}
  -- cycle;
\draw [dashed] (8,3.5) -- (8,2.5);
\draw (8,3.5) node[] {$\times$};
\draw [dashed] (9,5.5) -- (9,4.5);
\draw (9,4.5) node[] {$\times$};

\filldraw[draw=black, fill=gray!60] (0,3) node[anchor=north,color=black]{}
  -- (2,5) node[anchor=south,color=black]{}
  -- (3,5) node[anchor=north,color=black]{}
  -- (5,3) node[anchor=south,color=black]{}
  -- cycle;
\draw [dashed] (3,5) -- (3,4);
\draw (3,4) node[] {$\times$};
\draw [dashed] (2,4) -- (2,5);
\draw (2,4) node[] {$\times$};
\end{tikzpicture}
\caption{Four representatives of the semitoric polygon of the system when $s_1=s_2=\frac{1}{2}$.}
\label{fig:stpolygon-2ff}
\end{subfigure}
\caption{Comparison of the image of the momentum map for the system in Equation~\eqref{eqn:HPsystem} and representatives of the semitoric polygon for $s_1=s_2=\frac{1}{2}$.}
\label{fig:HPsystem}
\end{figure}
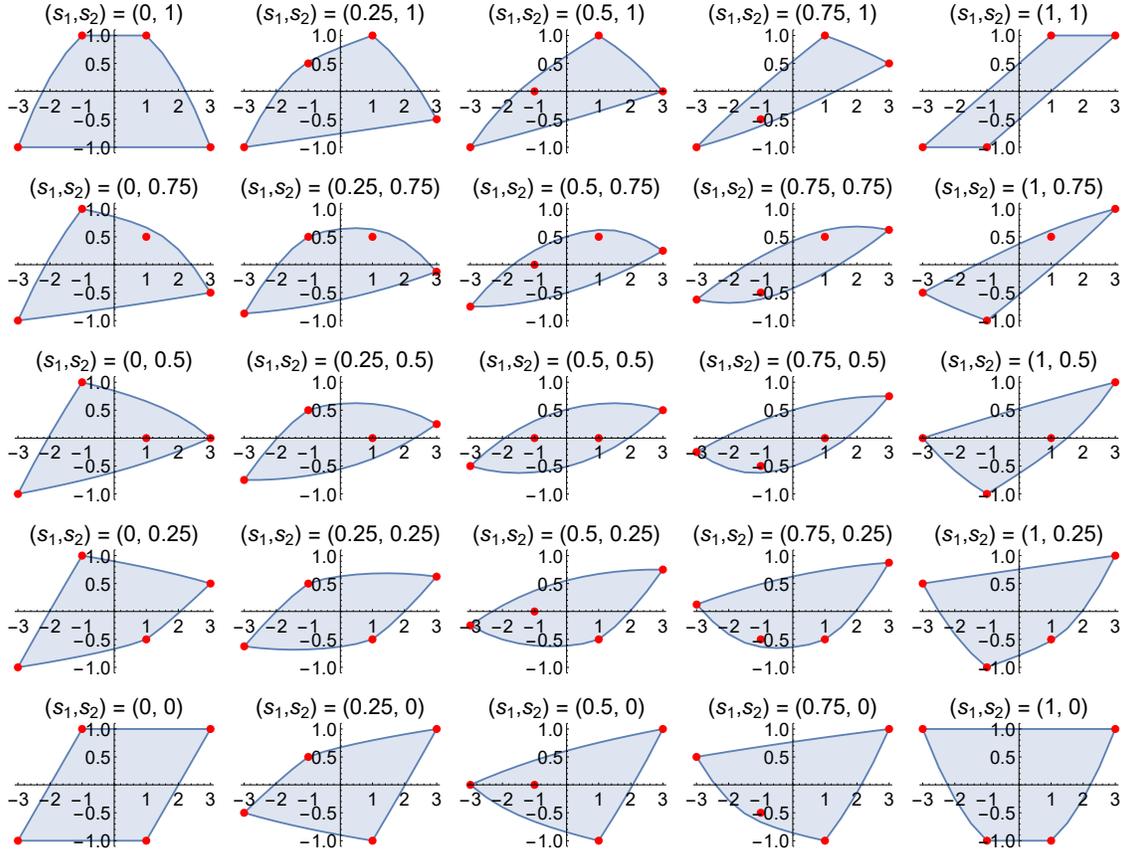
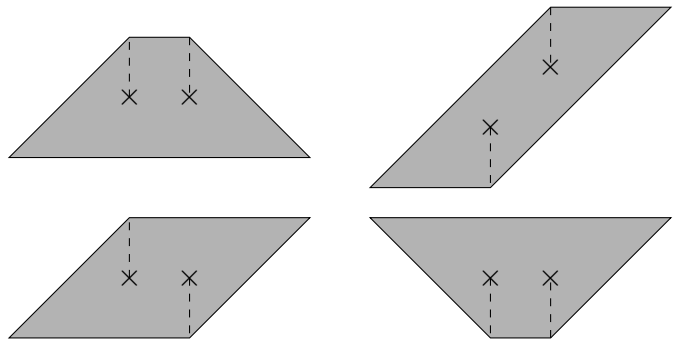

\section{\texorpdfstring{\Families}{Fixed-S1 families}}
\label{sec:firstprop}

In this section we prove some general facts about \families, semitoric families, and semitoric transition families, as defined in Definitions~\ref{def:family} and~\ref{dfn:transition_fam}.

\subsection{First properties}

Throughout this section let $(M,\om,F_t=(J,H_t))$ be a \family
as in Definition~\ref{def:family}. The following is a special case of~\cite[Proposition 2.8]{HohPal},
 and holds in general for families $(M,\om,(J_t,H_t))$ even in cases when $J_t$ does not generate an $\S^1$\--action or does not depend on $t$,
 see~\cite{HohPal}.

\begin{lm}\label{lem:HPdegen}
 Suppose that $p\in M$ and $t_0 \in (0,1)$ are such that $p$ is a fixed point
 for all $t$ in a neighborhood of $t_0$,
 $p$ is of focus-focus type for $t>t_0$,
 and of elliptic-elliptic type
 for $t<t_0$. Then $p$ is a degenerate fixed point for $t=t_0$.
\end{lm}

 We also have the following.

\begin{lm}\label{lem:stayssingular}
Let $(M,\om,(J,H_t))$ be a \family and
suppose that $p\in M$ is a rank zero singular point for some $t_0\in [0,1]$.
Then $p$ is a singular point for all $t\in [0,1]$ (but not necessarily of rank zero).
\end{lm}

\begin{proof}
Since $p$ is singular at $t_0$ we have that $dH_{t_0}(p) = dJ(p) = 0$, and since $J$ does not depend on $t$
we see that $dJ(p)=0$ implies that $p$ is singular for all $t$.
\end{proof}

In Theorem~\ref{thm:W1_movingAB} we give an example of a system for which some points change rank between zero and one.

There are many possible types of \families which would be interesting to study, for instance those in which hyperbolic singularities appear as the parameter changes. This would complement \cite{dullin-pelayo}, in which the authors show how to modify the Hamiltonian $H$ locally around a focus-focus point of $(J,H)$ to obtain hyperbolic singularities and get what they call a hyperbolic semitoric system. However, this is a problem of its own and in this paper we will focus on the case of semitoric families; those \families which are semitoric for all but finitely many values of the parameter. Note that hyperbolic semitoric systems naturally appear as well; for instance, such a non-trivial system on $\S^2 \times \R^2$ is explicitly described in \cite[Section 9]{dullin-pelayo} and Example 4.3 in the same article is a simple example of hyperbolic semitoric system with a different behavior. In Section~\ref{subsect:triangle} we give a new example of such a \family on a compact manifold, namely the first Hirzebruch surface.

\subsection{Semitoric families}

In this section let $(M,\om,F_t=(J,H_t))$, $0 \leq t \leq 1$, be a semitoric family with degenerate times $t_1,\ldots, t_k$,
as in Definition~\ref{def:family}.
In Section~\ref{sec:semitoricfam-general} we prove some general facts and in Section~\ref{sec:semitoricfam-degen} we discuss the possible behavior of the system at degenerate times. We will need the following standard lemma from Morse theory, which roughly says that a non-degenerate singular point cannot spontaneously appear, though it can move around in the manifold.
We give a short proof for the sake of completeness.

\begin{lm}
\label{lm:morse_family}
Let $X$ be a smooth manifold, $I$ be an open interval of $\R$, and $H: I \times X \to \R$ be smooth and such that there exists $t_0 \in I$ so that for $t \neq t_0$, $H_t = H(t,\cdot)$ is a Morse function on $M$. Assume that $x_0$ is a non-degenerate singular point of $H_{t_0}$; then there exists $\epsilon > 0$ and an open neighborhood $U$ of $x_0$ in $X$ such that for every $t \in (t_0 - \epsilon, t_0 + \epsilon) \subset I$, there exists a unique singular point $x(t)$ of $H_t$ in $U$. Moreover, $t \mapsto x(t)$ is smooth, and  for every $t \in (t_0 - \epsilon, t_0 + \epsilon) \subset I$, $x(t)$ has the same Morse index as $x_0$.
\end{lm}

\begin{proof}
We may work in a trivialization open set and assume that $X = \R^n$ for some $n$. The first part is then a simple application of the implicit function theorem, since singular points are solutions of $d_x H(t,x) = 0$, and since $d_x (d_x H)(t_0,x_0) = d^2_x H(t_0,x_0)$ is bijective by assumption. Since $H_t$ is Morse for $t \neq t_0$, the singular point $x(t)$ is automatically non-degenerate and its Morse index could only change if at least one eigenvalue of $d^2_x H(t,x(t))$ changed sign, which would imply that it vanishes for some $t_1$ and would contradict the fact that $x(t_1)$ is non-degenerate.
\end{proof}

\subsubsection{General results about semitoric families}
\label{sec:semitoricfam-general}

We start by proving some properties of the rank zero points of $F_t$.

\begin{lm}\label{lem:fixedpointsdontmove}
Let $t_0 \in [0,1]$ and let $p$ be a fixed point of $F_{t_0}$ which does not belong to a fixed surface of $J$. Then $p$ is a fixed point of $F_t$ for all $t\in [0,1]$.
\end{lm}

\begin{proof}
Since $p$ is a fixed point of $F_{t_0}$, in particular $dJ(p) = 0$. By Lemma \ref{lm:fixed_semitoric}, $p$ is a fixed point of $F_t$ for any $t \notin \{t_1, \ldots, t_k\}$, since for such $t$ the system is semitoric. This proves that $dH_t(p) = 0$ whenever $t \notin \{t_1, \ldots, t_k\}$; since $t \mapsto dH_t(p)$ is continuous, this in turn implies that $dH_{t_j}(p) = 0$ for $j = 1, \ldots, k$. Hence $p$ is a fixed point of $F_t$ for every $t$. 
\end{proof}

For $t \in [0,1]$, let $N_0(t)$ be the number of rank zero points of $F_t$. 
If $M$ is compact let $\chi(M)$ be the Euler characteristic of $M$.

\begin{lm}
\label{lm:euler}
Suppose $M$ is compact. For every $t \in [0,1] \setminus \{t_1, \ldots, t_k\}$, $N_0(t) = \chi(M)$. Moreover, if there is no fixed sphere for the $\S^1$-action generated by $J$, then $N_0(t) =  \chi(M)$ for every $t \in [0,1]$.
\end{lm}
\begin{proof}
It is well-known that if $\S^1$ acts smoothly on a closed smooth manifold $M$, then the fixed set $M^{\S^1}$ is a disjoint union of smooth manifolds and $\chi(M) = \chi(M^{\S^1})$. In our case, as discussed in Lemma~\ref{lem:dJzero}, $M^{\S^1}$ is a disjoint union of single points $p$, and $\chi(\{p\})=1$, and fixed surfaces of $J$, which are spheres; let $S$ be any of these spheres. For $t \notin \{ t_1, \ldots, t_k \}$, the system is semitoric so $S$ contains two rank zero points, and $\chi(S)=2$. So we indeed have $N_0(t) = \chi(M)$. Finally, if there is no fixed sphere for the $\S^1$-action generated by $J$, then Lemma \ref{lem:fixedpointsdontmove} ensures that the fixed points are the same for every $t \in [0,1]$, hence $N_0$ does not depend on $t$.
\end{proof}

When the $\S^1$-action generated by $J$ has a fixed sphere $S$ more complicated things can occur because the location of rank zero points in $S$ can change with $t$, and at a degenerate time additional rank zero points can appear. For instance, the system given in Equation \eqref{system:W1_switch} is a semitoric family with four rank zero points for $t \neq \frac{1}{2}$, but at $t = \frac{1}{2}$ there is an entire sphere of (degenerate) rank zero points.

\begin{lm}
For every $t \in [0,1]$, the system $(M,\om,F_t)$ does not have any singular
 points of hyperbolic-hyperbolic, hyperbolic-elliptic, or hyperbolic-transverse type.
\end{lm}

\begin{proof}
Notice that by the normal form given in Theorem~\ref{thm:normalform)}, any fixed point with a hyperbolic component comes with some nearby hyperbolic-transverse points (see also Figure 1 in \cite{dullin-pelayo}); thus, it is sufficient to discard the latter. Of course, they are excluded when $t \notin \{t_1, \ldots, t_k\}$ since for such $t$ the system is semitoric. So suppose that $p\in M$ is a singular point of hyperbolic-transverse type of $F_{t_{\ell}}$ for some ${\ell} \in \{1, \ldots, k\}$. Without loss of generality we may assume that $j = J(p)$ is a regular value of $J$ (again, see Theorem~\ref{thm:normalform)}). Hence, by Lemma \ref{lm:nondeg_rankone_red} the image of $p$ in the reduced manifold $M_j^{\text{red}}$ is a hyperbolic singular point of $H^{\text{red},j}_{t_{\ell}}$. Since $H^{\text{red},j}_t$ is a Morse function on $M_j^{\text{red}}$ for $t \neq t_{\ell}$ in an open interval $I$ containing $t_{\ell}$ (and chosen so small that it does not contain any of the $t_m$, $m \neq \ell$), Lemma \ref{lm:morse_family} implies that $H^{\text{red},j}_t$ has a singular point of hyperbolic type for $t \in I \setminus \{t_\ell\}$. This means that for such $t$, $F_t$ has some rank one points of hyperbolic-transverse type, which is absurd since it is semitoric.
\end{proof}

Let $(M,\om,F_t)$ be a semitoric family which is simple for all non-degenerate $t$.
Recall that the unmarked semitoric polygon is defined in Section~\ref{sec:semitoric}; now we show
that the number of focus-focus points and unmarked semitoric polygon can only change
at degenerate times.

\begin{lm}\label{lem:samepoly}
 Let $I\subset [0,1]\setminus\{t_1, \ldots, t_k\}$ be an interval. Then
 for any $t, t'\in I$ the semitoric systems $(M,\om,F_t)$ and $(M,\om,F_{t'})$
 have the same number of focus-focus points and the same unmarked semitoric polygon.
\end{lm}

In order to prove this lemma, we state a result which is well-known and follows, for instance, from~\cite[Theorem 5.3]{VNpoly}. As before, let $\pi_1, \pi_2$ be the natural projections on each factor of $\R^2$. Given a polygon $\Delta$ for which $\pi_1^{-1}(x) \cap \Delta$ is compact for every $x \in \R$ (which is in particular the case for semitoric polygons), we define the \emph{top boundary of $\De$} as the set 
\begin{equation} \partial^{\mathrm{top}}\De = \left\{(x, y) \in \Delta \ | \ y = \max_{\pi_1^{-1}(x) \cap \Delta} \pi_2   \right\}. \label{eq:top_bound}\end{equation}

\begin{lm}\label{lem:changeslope}
 Let $(\De,\vec{\ell},\vec{\epsilon})$ be a representative of the unmarked semitoric polygon of a simple semitoric system $(M,\om,F)$ associated
 to a developing map $g$, so that $\De = (g\circ F)(M)$. Given a corner $c = (c_1,c_2)$ of $\De$ on its top boundary, let $s_{\ell}$ (respectively $s_r$) be the slope of the edge of $\Delta$ adjacent to $c$ and consisting of points $(x,y) \in \R^2$ with $x \leq c_1$ (respectively $x \geq c_1$) and let $k$ be the number of focus-focus points in $J^{-1}(c_1)$. Then $s_r - s_{\ell} = w_e - k$ where $w_e = -\frac{1}{|ab|}$ if $(g\circ F)^{-1}(c)$ is an elliptic-elliptic point $p$ and $a,b\in\Z$ are the weights at $p$ of the $\S^1$-action induced by J, and $w_e = 0$ otherwise.
\end{lm}

\begin{proof}[Proof of Lemma \ref{lem:samepoly}]
Let $t \in I$; by Lemma \ref{lem:fixedpointsdontmove}, the focus-focus points of $F_t$ are also rank zero points (and thus either of elliptic-elliptic or focus-focus type) of $F_{s}$ for any $s \in I$, since focus-focus points cannot occur in fixed spheres of $J$. Because of Lemma \ref{lem:HPdegen}, they can only change type at a degenerate time.

Now, let $t' \in I$ and let $(\Delta',\vec{\ell},\vec{\epsilon})$ be a representative of the unmarked semitoric polygon of the system $(J,H_{t'})$. We will construct a representative of the unmarked semitoric polygon of $(J,H_t)$ which coincides with $(\Delta',\vec{\ell},\vec{\epsilon})$, which will imply that the unmarked semitoric polygon is the same at $t$ and $t'$. 
First, notice that the vertical lines $\vec{\ell} = (\ell_1,\ldots,\ell_{m_f})$ can be used to construct a representative of the unmarked semitoric polygon at $t$, since they are given by the $J$-values of the focus-focus points. Note also that the height of the unmarked semitoric polygon at each $J$-value $j$ is given by $(2\pi)^{-1}\mathrm{vol}(J^{-1}(j))$, which is independent of $t$, so it is sufficient to check that we can construct a polygon for $t$ with cut directions $(\epsilon_1, \ldots, \epsilon_{m_f})$ which has the same top boundary as $\De$. In what follows, let $g': \R^2 \to \R^2$ be as in Equation \eqref{eq:dev_map} such that $\Delta' = (g' \circ F_{t'})(M)$.
 
\paragraph{Step 1: the top left corner.} Even if $M$ is not compact, Corollary $5.5$ in \cite{VNpoly} ensures that, unless the system has no rank zero point at all (in which case we do not have to prove anything) $J$ has either a global minimum or a global maximum, and we may assume the former. Let $q'$ be the elliptic-elliptic point of $(J,H_{t'})$ with minimal $J$-value $j_{\min}$ and such that 
\[ H_{t'}(q') = \max_{J^{-1}(j_{\min})} H_{t'} \]
(recall that $J^{-1}(j_{\min})$ can be a fixed surface, in which case the left side of $\De'$ is a vertical wall), so that $q'$ is sent to the point $Q'$ on $\partial^{\mathrm{top}}\De'$ which has minimal $J$-value.
If $J^{-1}(j_{\min})$ is a fixed surface, then the edges of $\De'$ emanating from $Q'$ are directed along the vectors 
\[ \begin{pmatrix} 0 \\ -1 \end{pmatrix}, \qquad \begin{pmatrix} 1 \\ n \end{pmatrix} \]
for some $n \in\Z$, since it is a Delzant corner.
Similarly, let $q\in J^{-1}(j_{\text{min}})$ be the elliptic-elliptic point of $(J,H_t)$ with maximal $H_t$-value among the points with minimal $J$-value, and let $Q$ be its image in one representative $(g \circ F_t)(M)$ of the unmarked semitoric polygon of $F_t$ constructed using the cuts $(\epsilon_1, \ldots, \epsilon_{m_f})$. Then the edges of this polygon leaving $Q$ are also directed along
\[ \begin{pmatrix} 0 \\ -1 \end{pmatrix}, \qquad \begin{pmatrix} 1 \\ p \end{pmatrix} \]
for some $p \in \Z$. Therefore, by setting $\tilde{g} = u \circ T^{n-p} \circ g$ where $T$ is as in Equation \eqref{eqn:T} and $u$ is a vertical translation, we obtain another representative $\Delta = (\tilde{g} \circ F_t)(M)$ of the unmarked semitoric polygon of $F_t$ such that $\Delta$ agrees with $\Delta'$ in a neighborhood of the point $Q$. 
If $J$ has a unique minimum, it occurs at $q$, and thus $q$ is of elliptic-elliptic type for all $t$, so as a consequence of Lemma~\ref{lem:samecorner} below we can choose
such a $\tilde{g}$ (hence $\Delta$) in a similar way. Note that in both cases $\De'$ and $\De$ have the same leftmost point in their top boundary, and the first segment of the (piecewise linear) top
boundaries agree in slope.
\\
\paragraph{Step 2: the top boundary.} 
Now we only need to study the position of each vertex on the top boundary and the change in slope at each vertex for $\De$ and $\De'$. The vertices in the top boundary of $\De$ occur at points $c=(c_1,c_2)$ such that either:
\begin{itemize}[nosep]
 \item $(g\circ F_{t})^{-1}(c)$ is an elliptic-elliptic point whose $H_t$-value is maximum in $J^{-1}(c_1)$, or
 \item the $i$-th focus-focus point is in the set $J^{-1}(c_1)$ and $\epsilon_i=+1$.
\end{itemize}
We already saw that the focus-focus points of $F_t$ and $F_{t'}$ have the same $J$-value. Moreover, a similar argument shows that an elliptic-elliptic point of $F_t$ is also an elliptic-point of $F_{t'}$, with the same $J$-value $j$, and is a maximum of $H_t$ in $J^{-1}(j)$ if and only if it is a maximum of $H_{t'}$ in $J^{-1}(j)$ (otherwise, there would be a value of the parameter in $I$ for which this point is both a global maximum and a global minimum of the Hamiltonian on $J^{-1}(j)$, so it would be degenerate, see Section \ref{sec:semitoricfam-degen} below). Thus, the vertices in the top boundaries of $\De$ and $\De'$ occur at the same values of $J$.
Moreover, by Lemma~\ref{lem:changeslope}, the change of slope at each vertex only depends on the weights of the action of $J$ at the elliptic-elliptic points and the number of focus-focus points in the $J$-fiber; hence the change of slope at corresponding vertices of $\De$ and $\De'$ is the same. 
\end{proof}

A \emph{Delzant cone} is a set of the form $\{a_1 v_1 + a_2 v_2\mid a_i\geq 0\}\subset\R^2$ where
$v_1,v_2\in\Z^2$ satisfy $\mathrm{det}(v_1,v_2)=1$.
\begin{lm}
\label{lem:samecorner}
Let $(M,\om,F_t=(J,H_t))$ be a semitoric family and let $(p_t)_{0 \leq t \leq 1}$ be a continuous family of points of $M$ such that for every $t \in [0,1]$, $p_t$ is a fixed point of elliptic-elliptic type of $F_t$. Then there exists a toric momentum map $\Phi$ on $\C^2$ and for each $t$ there exists an open set $U_t$ containing $p_t$, a local diffeomorphism $g_t:(\R^2,F_t(p_t)) \to (\R^2,0)$, and a local symplectomorphism $\chi_t\colon (U_t,p_t)\to (\C^2,0)$ such that:
\begin{enumerate}[noitemsep]
 \item $g_t$ is of the form $g_t (x,y) = (x, g_t^{(2)}(x,y))$,
 \item $(g_t \circ F_t)_{| U_t}= \Phi\circ \chi_t$,
 \item $(g_t \circ F_t)(U_t)=K \cap \Omega_t$ where $K$ is a fixed Delzant cone and $\Omega_t$ is an open neighborhood of the origin.
\end{enumerate}
\end{lm}
\begin{proof}
By the local normal form for elliptic-elliptic points (Theorem~\ref{thm:normalform)}) for each $t$ there exist $U_t$, $g_t$, $\chi_t$, and $\Phi_t$ such that $g_t\circ F_t = \Phi_t\circ\chi_t$ on $U_t$, with $g_t$ satisfying the first property (see for instance step 2 in the proof of Theorem 3.8 in~\cite{VNpoly}). 
For $t\in [0,1]$, let $K_t = \Phi_t(\C^2)$; then $K_t$ is a Delzant cone determined by the weights of the $\T^2$-action generated by $g_t \circ F_t$. 
Since $(p_t)_{t \in [0,1]}$ is a continuous family of fixed points of $J$, all the points in this family belong to the same component of the fixed point set of $J$, hence Lemma \ref{lem:dJzero} implies that either $p_t = p$ does not depend on $t$ or $J(p_t)$ is the maximum (respectively minimum) of $J$ for every $t \in [0,1]$. In both cases, the weights of the first component of the $\T^2$-action generated by $g_t \circ F_t$ do not depend on $t$, so there exists $k_t\in\Z$ such that $K_0 = T^{k_t} K_t$ with $T$ as in Equation \eqref{eqn:T}.
Thus, by replacing $g_t$ with $(x,g_t^{(2)}(x,y) + k_t x)$ and $\Phi_t$ with $T^{k_t}\circ\Phi_t$, we obtain $\Phi_t=\Phi_0$ (hence $K_t=K_0$) while preserving the first two properties. 
\end{proof}

\subsubsection{Semitoric families at degenerate times}
\label{sec:semitoricfam-degen}

Now, we investigate the events that can occur in a semitoric family, and describe the mechanisms leading to these events. We also provide examples for many of the different behaviors. As before, let $(M,\omega,F_t = (J,H_t))$ be a semitoric family with degenerate times $\{t_1, \ldots, t_k \}$.

We already have a good understanding of the times that are not degenerate; indeed, Lemma \ref{lm:euler} tells us that on $[0,1] \setminus \{t_1, \ldots, t_k\}$ the number of fixed points does not change (if $M$ is compact), and Lemma \ref{lem:fixedpointsdontmove} ensures that fixed points themselves do not change, except maybe those that lie in a fixed sphere of $J$. In fact, the system described in Equation \eqref{system:W1_movingAB} displays such a behavior: the minimum of $J$ corresponds to a fixed sphere and rank zero points in this sphere change with $t$. Moreover, rank one singular points can change with time, as it is the case in every example in this paper. Hence, we now focus on what happens at a degenerate time.

Let $t_\ell \in \{t_1, \ldots, t_k\}$, let $I \subset [0,1]$ be an open interval containing $t_{\ell}$ and none of the $t_j$, $j \neq \ell$, and let 
\[ I_- = \{ t \in I \ | \ t < t_{\ell} \}, \qquad I_+ = \{ t \in I \ | \ t > t_{\ell} \}. \]

Here we have implicitly assumed that $t_\ell \notin \{0,1\}$. Note that if $t_\ell\in\{0,1\}$ then the possible behaviors are the same,
up to working with only $I_-$ or $I_+$.
To organize the possible cases, we use the following definition:
\begin{dfn}
We say that a singular point $p$ of $F_{t_{\ell}}$ \emph{appears} at $t_{\ell}$ if there do not exist two families $(p_t^\pm)_{t \in I_\pm}$ such that $p_t^\pm$ is a singular point of $F_t$ of the same rank as $p$ and converges to $p$ when $t$ goes to $t_{\ell}$.
\end{dfn}

Thus, there are three possible cases for degenerate times: appearance of a singular point (of rank one or zero), rank zero points becoming degenerate, or rank one points becoming degenerate.

\paragraph{Appearance of rank zero and rank one degenerate points.} A rank zero point can only appear in a fixed sphere of $J$ by Lemma \ref{lem:fixedpointsdontmove}. Thus, the only singular points that appear at $t_{\ell}$ are either rank zero points belonging to a fixed sphere of $J$ or rank one points. By Lemma \ref{lm:morse_family}, these new singular points are necessarily degenerate.

\begin{example}[Appearance]\label{ex:appearance}
Consider, on the manifold $\S^2 \times \S^2$ with symplectic form $\omega_{\S^2} \oplus \omega_{\S^2}$ and coordinates $(x_1,y_1,z_1,x_2,y_2,z_2)$, the \family obtained from
\begin{equation} J = z_1, \qquad H_t = z_2^3 + \left( (1 + z_1)^2 + (1-2t)^2 \right) z_2. \label{eq:system_appearance}\end{equation}
We claim that this is a semitoric family and at $t = \frac{1}{2}$ the circle given by $z_1 = -1$ and $z_2 = 0$ consists in degenerate rank zero points of the system (but these points are rank one singular points for $t \neq \frac{1}{2}$). In fact, if we replace this $H_t$ by $H_t = z_2^3 + \left( (z_1 - j_0)^2 + (1-2t)^2 \right) z_2$ for any $j_0 \in (-1,1)$, rank one points appear at $t=\frac{1}{2}$ on the regular level $J^{-1}(j_0)$.
\end{example}

\begin{rmk}
More generally, the following procedure can be used to construct examples on $\S^2 \times \S^2$ exhibiting various behaviors: let $J = z_1$ be the momentum map for the standard $\S^1$-action, let $h: [0,1] \times [-1,1] \times \S^2 \to \R$ be a smooth function, and let
\[ H: [0,1] \times \S^2 \times \S^2 \to \R, \qquad (t,x_1,y_1,z_1,x_2,y_2,z_2) \mapsto h(t,z_1,x_2,y_2,z_2). \]
For a given $t$, if $h_{t, z} = h(t,z,\cdot,\cdot,\cdot)$ is Morse for all $z\in[-1,1]$ then $(J,H_t= H(t,\cdot))$ is a semitoric system, so we can choose $h$ in such a way that $(J,H_t )$ is a semitoric family. 
This process can be used to produce examples of semitoric families displaying many different behaviors coming from what can happen in a two-parameter family $h_{t,z}$ of smooth functions on the sphere which are Morse for all but finitely many values of $t$ (note not all such functions are sufficient, since we must also assure that integrability holds at the degenerate times).
\end{rmk}

\paragraph{Rank zero points that become degenerate.} If an already existing rank zero point $p$ becomes degenerate at $t_{\ell}$ (on a fixed sphere $p$ may be the limit of a family $p_t$, where $p_t$ is a fixed point of $F_t$) one of the following happens:
\begin{enumerate} 
\item if $p$ belongs to a fixed sphere $S = J^{-1}(j_0)$ of $J$, then there exists a family $p_t$ such that $p_{t_\ell}=p$ and $p_t$ is the global maximum or minimum of $H_t^{\text{red},j_0} = {H_t}_{|S}$ (see Lemma \ref{lm:sphere}) for $t \in I \setminus \{t_{\ell}\}$; hence we are faced with two cases:
\begin{enumerate}
\myitem{(1a)}\label{item:fixedcollapse} $p_t$ is the global maximum (respectively minimum) of ${H_t}_{|S}$ for $t \in I_-$ and the global minimum (respectively maximum) of ${H_t}_{|S}$ for $t \in I_+$. Assume for instance that we are in the first case and let $m \in S$; from 
\[ \begin{cases} \forall t \in I_-, \quad H_t(m) \leq H_t(p_t),  \\ \forall t \in I_+, \quad H_t(m) \geq H_t(p_t), \end{cases} \]
and the continuity of $t \mapsto H_t$, we obtain that $H_{t_{\ell}}(m) = H_{t_{\ell}}(p)$. Consequently, we have that $F_{t_{\ell}}^{-1}(F_{t_{\ell}}(p)) = S$; in this case, we say that the $J$-fiber $S$ \emph{collapses} at $t = t_{\ell}$. Such a situation occurs at $t = \frac{1}{2}$ for the system given in Equation \eqref{system:W1_switch}, see Figure~\ref{fig:moment_map_W1_switch},
\item[(1b)] $p_t$ is the global maximum (respectively minimum) of ${H_t}_{|S}$ for $t \in I_-$ and $t \in I_+$. Then either $H_{t_{\ell}}$ is constant on $S$ (which means that $S$ collapses at $t=t_{\ell}$) or the same argument yields that $p=p_{t_\ell}$ is still a global maximum (respectively minimum) of ${H_{t_{\ell}}}_{|S}$, but $p$ may not be isolated anymore among rank zero points in $S$. For instance, if
\begin{equation} J = z_1, \qquad H_t = (z_2-1)^2 + ( (1-2t)^2 + (z_1 - j_0)^2) z_2 \label{eq:become_deg}\end{equation}
on $\S^2 \times \S^2$ with the standard symplectic form (see Example \ref{ex:appearance} for notation) and $j_0 = -1$, then $(0,0,-1,0,0,1)$ is elliptic-elliptic for $t \neq 1/2$ but degenerate for $t=1/2$. This follows from the fact that $(0,0,1)$ is a critical point of $(x,y,z) \mapsto (z-1)^2 + c z^2$ on $\S^2$ with Hessian matrix
\[ \begin{pmatrix} -c & 0 \\ 0 & -c \end{pmatrix}, \]
\end{enumerate}
\item otherwise, for every $t \in I \setminus \{ t_{\ell} \}$, $p$ is either an elliptic-elliptic or focus-focus point of $F_t$. Let $j = J(p)$; then either
\begin{enumerate}
\myitem{(2a)} \label{item:scenario_2a} $p$ is elliptic-elliptic for $t \in I_-$ and for $t \in I_+$. In this case, it is a global maximum or minimum of $H_t^{\text{red},j}$, and we can argue as in the case of a fixed sphere, using the continuity of $t \mapsto H_t^{\text{red},j}$. Hence the same scenarios can occur. For instance, let $R_2 > R_1 > 0$ and let $M = \S^2 \times \S^2$ with symplectic form $R_1 \omega_{\S^2} \oplus R_2 \omega_{\S^2}$ and coordinates $(x_1,y_1,z_1,x_2,y_2,z_2)$. Then for any $j_0 \in (-(R_1 + R_2), R_1 + R_2)$ (which in particular includes two of the $J$-fibers containing an elliptic-elliptic point), the \family given by 
\begin{equation} J = R_1 z_1 + R_2 z_2, \qquad H_t = (1-2t)z_1 + (J-j_0) (x_1 x_2 + y_1 y_2) \label{eq:system_collapse}\end{equation} 
exhibits a collapse of the level set $J^{-1}(j_0)$ at $t= \frac{1}{2}$. We claim that this is in fact a semitoric family (we do not give any details but display the image of the momentum map in Figure \ref{fig:moment_map_collapse_EE}),
\begin{figure}
\begin{center}
\includegraphics[scale=0.45]{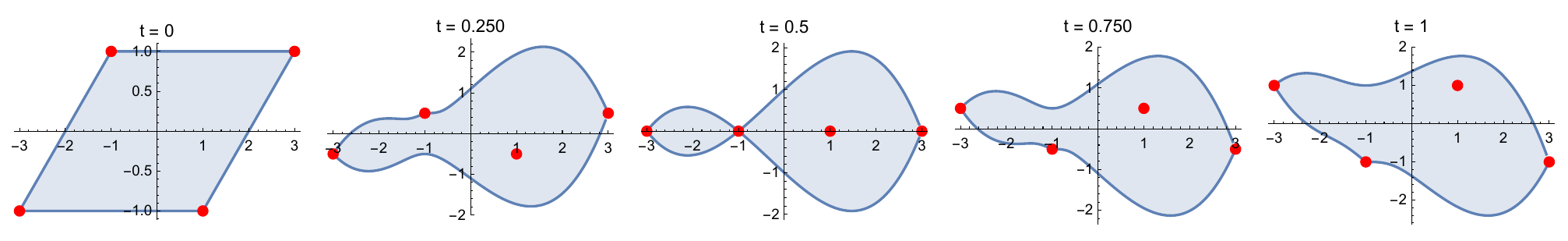}
\caption{Image of the momentum map for the system in Equation \eqref{eq:system_collapse} in the case $R_1=1$, $R_2=2$ and $j_0 = -1$.
Note that the point with $z_1=1$ and $z_2=-1$ is never of focus-focus type.}
\label{fig:moment_map_collapse_EE}
\end{center}
\end{figure}
\myitem{(2b)}\label{item:HHbif} or $p$ is elliptic-elliptic for $t \in I_-$ and focus-focus for $t \in I_+$ (or the converse which is similar). This means that for every $t \in I_-$, there exists $(\nu,\mu) \in \R^2$ such that the eigenvalues of $A_{\nu,\mu}^t = \Omega_p^{-1} (\nu d^2 J(p) + \mu d^2 H_t(p))$ (see Section \ref{subsect:sing_dim4} for notation) are of the form $\pm i \alpha, \pm i \beta$ for real and nonzero $\alpha \neq \beta$, and that for every $t \in I_+$, there exists $(\nu,\mu) \in \R^2$ such that the eigenvalues of $A_{\nu,\mu}^t$ are of the form $\pm \gamma \pm i \delta$ for real and nonzero $\gamma, \delta$ (note that $\nu,\mu,\alpha,\beta,\gamma$ and $\delta$ depend on $t$). Let $(\tau_n^-)_{n \geq 1}$ and $(\tau_n^+)_{n \geq 1}$ be two sequences converging to $t_{\ell}$ and such that for every $n \geq 1$, $\tau_n^- \in I_-$ and $\tau_n^+ \in I_+$. Since the set $E_t$ of $(\nu, \mu) \in \R^2$ satisfying the above properties for a given $t$ is open and dense, the intersection $\bigcap_{n \geq 1} (E_{\tau_n^-} \cap E_{\tau_n^+})$ is dense. Hence we can find a fixed $(\nu_0, \mu_0) \in \R^2$ such that for every $n \geq 1$, the eigenvalues of $A_{\nu_0, \mu_0}^{\tau_n^-}$ are of the form $\pm i \alpha_n, \pm i \beta_n$ and the eigenvalues of $A_{\nu_0, \mu_0}^{\tau_n^+}$ are of the form $\pm \gamma_n \pm i \delta_n$. Since $H_t$ depends continuously on $t$ and these eigenvalues are simple, they have limits
\[ \pm i \alpha_n \underset{n \to +\infty}{\longrightarrow} \pm i \alpha_{\infty}, \quad \pm i \beta_n \underset{n \to +\infty}{\longrightarrow} \pm i \beta_{\infty}, \quad \pm \gamma_n \pm i \delta_n \underset{n \to +\infty}{\longrightarrow} \pm \gamma_{\infty} \pm i \delta_{\infty}. \]
Since $\{ \pm i \alpha_{\infty}, \pm i \beta_{\infty} \} = \{ \pm \gamma_{\infty} \pm i \delta_{\infty} \}$, necessarily $\gamma_{\infty} = 0$ and $\alpha_{\infty} = \beta_{\infty} = \delta_{\infty}$. Hence $p$ undergoes a so-called Hamiltonian-Hopf bifurcation (see for instance \cite{Montaldi_notes, vdM, COR03, COR03_corr}) as shown in Figure \ref{fig:HH_bif_ev},
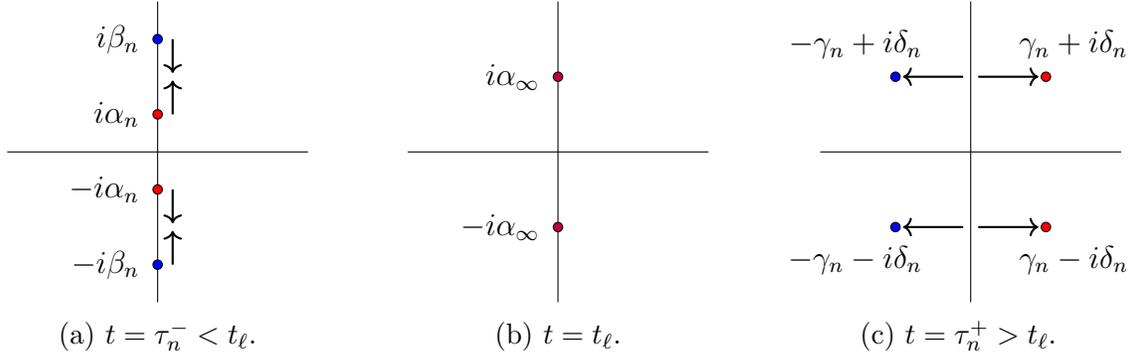
\begin{figure}
\begin{center}
\begin{subfigure}[b]{.3\textwidth}
\centering
\begin{tikzpicture}
\draw (-2,0) -- (2,0);
\draw (0,-2) -- (0,2);

\draw [fill=red] (0,1/2) circle (1/16);
\draw [fill=red] (0,-1/2) circle (1/16);
\draw [fill=blue] (0,3/2) circle (1/16);
\draw [fill=blue] (0,-3/2) circle (1/16);

\draw (-0.1,1/2) node [left] {$i \alpha_n$};
\draw (-0.1,-1/2) node [left] {$-i \alpha_n$};
\draw (-0.1,3/2) node [left] {$i \beta_n$};
\draw (-0.1,-3/2) node [left] {$-i \beta_n$};

\draw [->,thick] (0.2,1/2) -- (0.2,0.95);
\draw [->,thick] (0.2,3/2) -- (0.2,1.05);

\draw [->,thick] (0.2,-1/2) -- (0.2,-0.95);
\draw [->,thick] (0.2,-3/2) -- (0.2,-1.05);
\end{tikzpicture}
\caption{$t = \tau_n^- < t_{\ell}$.}
\end{subfigure} 
\hspace{2pt}
\begin{subfigure}[b]{.3\textwidth}
\centering
\begin{tikzpicture}
\draw (-2,0) -- (2,0);
\draw (0,-2) -- (0,2);

\draw [fill=purple] (0,1) circle (1/16);
\draw [fill=purple] (0,-1) circle (1/16);

\draw (-0.1,1) node [left] {$i \alpha_{\infty}$};
\draw (-0.1,-1) node [left] {$-i \alpha_{\infty}$};

\end{tikzpicture}
\caption{$t = t_{\ell}$.}
\end{subfigure}
\hspace{2pt}
\begin{subfigure}[b]{.3\textwidth}
\centering
\begin{tikzpicture}
\draw (-2,0) -- (2,0);
\draw (0,-2) -- (0,2);

\draw [fill=red] (1,1) circle (1/16);
\draw [fill=red] (1,-1) circle (1/16);
\draw [fill=blue] (-1,1) circle (1/16);
\draw [fill=blue] (-1,-1) circle (1/16);

\draw (0.5,1.1) node [above right] {$\gamma_n + i \delta_n$};
\draw (0.5,-1.1) node [below right] {$\gamma_n - i \delta_n$};
\draw (-0.5,1.1) node [above left] {$-\gamma_n + i \delta_n$};
\draw (-0.5,-1.1) node [below left] {$-\gamma_n - i \delta_n$};

\draw [->,thick] (0.1,1) -- (0.9,1);
\draw [->,thick] (-0.1,1) -- (-0.9,1);
\draw [->,thick] (0.1,-1) -- (0.9,-1);
\draw [->,thick] (-0.1,-1) -- (-0.9,-1);
\end{tikzpicture}
\caption{$t = \tau_n^+ > t_{\ell}$.}
\end{subfigure}
\end{center}
\caption{Eigenvalues during the Hamiltonian-Hopf bifurcation at $t_{\ell}$ described in case \ref{item:HHbif}.}
\label{fig:HH_bif_ev}
\end{figure}
\item[(2c)] or $p$ is focus-focus for $t \in I_-$ and for $t \in I_+$. 
In general, in this case the images of any number of focus-focus points and zero, one, or two elliptic-elliptic points can collide.
There are two subcases:
\begin{enumerate}
 \item $p$ is not a global extremum of $H_{t_\ell}$ on $J^{-1}(J(p))$.
  In this case, the fiber $F_{t_\ell}^{-1}(p)$ can include regular points, degenerate rank one points, or degenerate/focus-focus rank zero points.
 \item $p$ is a global extremum of $H_{t_\ell}$ on $J^{-1}(J(p))$. This means that the image $F_{t}(p)$ is going to the boundary of $F_t(M)$ as $t\to t_\ell$, but as in the previous case this can occur in many different ways, the only difference being that now the fiber may include one or two elliptic-elliptic points. For instance,
 \begin{itemize}
  \item $F_{t_\ell}(p) = F_{t_\ell}(q)$ where $q\neq p$ is a global extremum of $(H_t)_{|J^{-1}(J(p))}$ for $t\in I$, as occurs in the coupled angular momenta system~\cite{LFP} with $R_1=R_2$, see Figure \ref{fig:spins_equal_R}.
	\begin{figure}
\begin{center}
\includegraphics[scale=0.45]{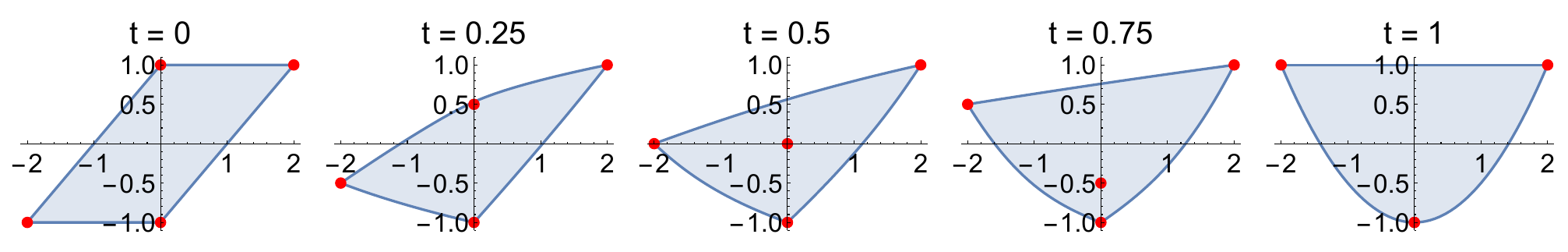}
\caption{Image of the momentum map for the coupled angular momenta with $R_1 = R_2 = 1$.}
\label{fig:spins_equal_R}
\end{center}
\end{figure}
  \item $F_{t_\ell}(p) = F_{t_\ell}(q)$ where $q\neq p$ is a focus-focus point for $t\in I_-$ or $t\in I_+$. For instance,
   this happens for the system from~\cite{HohPal} described in Equation \eqref{eqn:HPsystem} with $R_1=R_2$, $s_1 = t$ and $s_2=1-t^2$, see Figure \ref{fig:HP_FF_collide};
		\begin{figure}
\begin{center}
\includegraphics[scale=0.45]{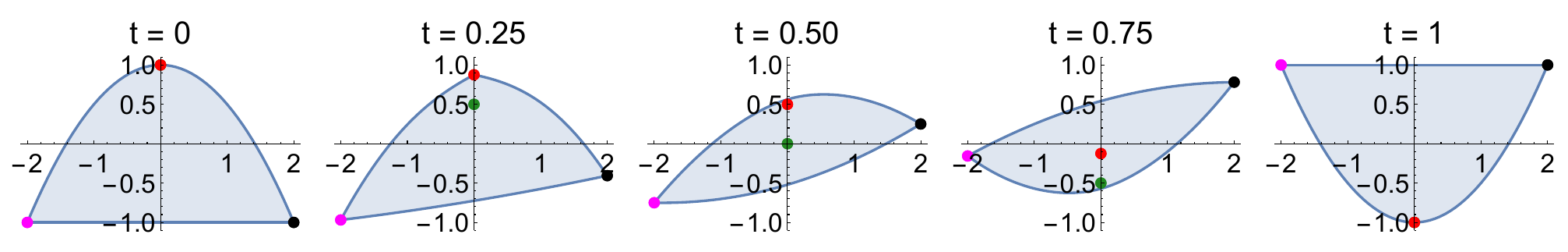}
\caption{Image of the momentum map for the system in Equation \eqref{eqn:HPsystem} with $R_1 = R_2 = 1$, $s_1 = t$ and $s_2 = 1-t^2$; the images of the fixed points in the fiber $J^{-1}(0)$ coincide only at $t=0$ and $t=1$.}
\label{fig:HP_FF_collide}
\end{center}
\end{figure}
 \end{itemize}
\end{enumerate}
\end{enumerate}
\end{enumerate}

\paragraph{Rank one points that become degenerate.} Rank one points can also become degenerate. Such points belonging to an interior fiber of $J$ are global minima or maxima of $H$ restricted to this fiber; so in this case the possible scenarios are the same as in item \ref{item:scenario_2a} above. For instance, Figure \ref{fig:moment_map_collapse} displays the image of the momentum map for the system in Equation \eqref{eq:system_collapse} with $j_0$ an interior value of $J$, and taking $j_0 \notin \{-1,1\}$ in Equation \eqref{eq:become_deg} yields a system where some of the rank one singular points on $J^{-1}(j_0)$ become degenerate. In an extremal fiber of $J$, however, rank one points can become degenerate rank zero points.

\begin{figure}
\begin{center}
\includegraphics[scale=0.45]{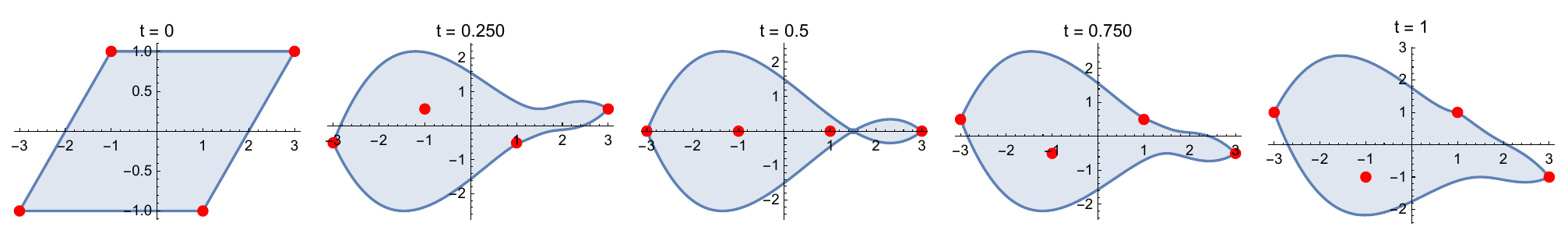}
\caption{Image of the momentum map for the system in Equation \eqref{eq:system_collapse} in the case $R_1=1$, $R_2=2$ and $j_0 = \frac{3}{2}$.}
\label{fig:moment_map_collapse}
\end{center}
\end{figure}

In the rest of this paper, we choose to focus on the case where one single rank zero point undergoes two Hamiltonian-Hopf bifurcations (see case~\ref{item:HHbif}), going from elliptic-elliptic as a global maximum on its $J$-fiber to focus-focus and then elliptic-elliptic as a global minimum on its $J$-fiber (or the converse); this leads us to the definition of a semitoric transition family, see Definition \ref{dfn:transition_fam}. We study the first properties of these families in the next subsection.

\subsection{Semitoric transition families}
\label{sec:semitorictransfam}

We start by giving the geometric idea of what happens in a semitoric transition family, for instance in the coupled angular momenta system (see Section~\ref{sec:knownex}). In order to do so, let $j$ be a singular level of $J$ containing one fixed point $p$ of $J$ and embed the singular manifold $M_j^{\text{red}}$ in $\R^3$ in such a way that the singularity (which is the projection of $p$) corresponds to the global maximum of the height function $Z$. Assume moreover that $H_0$ is such that Z is a function of the reduced Hamiltonian $R := H_0^{\text{red},j}$ (in which case the level sets of $Z$ and $R$ coincide), as in Figure \ref{fig:teardrops}, and that $H_{1/2}^{\text{red},j}$ corresponds to one of the horizontal coordinates $X$ or $Y$ on $M_j^{\text{red}}$. Then $p$ is an elliptic-elliptic singular point of $(J,H_0)$, and a focus-focus singular point of $(J,H_{1/2})$. Moreover, by transitioning from $H_0$ to $H_{1/2}$ in the most natural way, there exists $t^- \in (0,\frac{1}{2})$ such that the point $p$ stays elliptic-elliptic for $0 \leq t < t^-$, becomes degenerate at $t = t^-$ and is focus-focus for $t^- < t \leq \frac{1}{2}$. 
This process continues as $t$ increases from $t=\frac{1}{2}$ to $t=1$, until the singular point is the minimum of $H_t^{\text{red},j}$, which we do not include in the figure since it is similar.
Having this picture in mind helps to find explicit systems once we know enough about the reduced spaces; indeed, this is exactly the strategy we employ in Sections \ref{sec:W1_example} and \ref{sec:W2examples}, using the same notation $X,Y,Z,R$.
Thus, the first step is finding the coordinates $(X,Y,Z)$;
note that in principle the coordinates $X$ and $Y$ could depend on $j$, but in our examples they do not. For instance, the function $X = x_1 x_2 + y_1 y_2$ in Equation \eqref{eqn:HPsystem} is obtained exactly in this way.

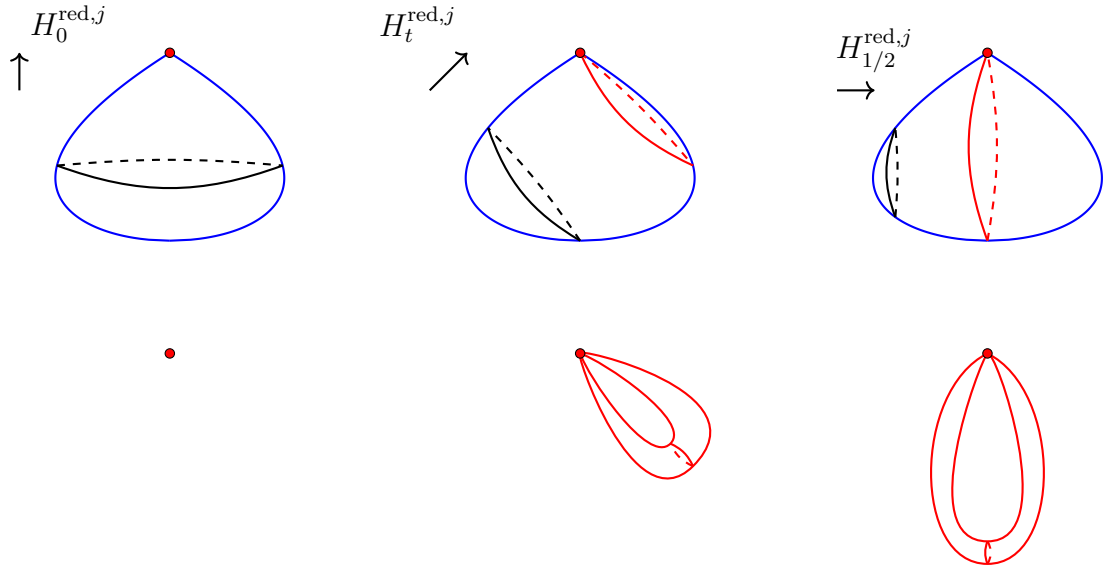
\begin{figure}[t]
\begin{center}
\begin{subfigure}[b]{.3\textwidth}
\centering
\def\const{2.5}
\begin{tikzpicture}
\draw [thick, ->] (-2,2)--(-2,2.5);
\draw (-2,2.5) node [above right] {$H_0^{\text{red},j}$};
\draw [domain=0:\const,samples=200,color=blue,thick] plot ({sqrt(\x*(\const-\x)^2)}, {\x});
\draw [domain=0:\const,samples=200,color=blue,thick] plot ({-sqrt(\x*(\const-\x)^2)}, {\x});
\draw [fill = red] (0,\const) circle (1/16);
\draw [thick] ({-(\const-1)},1) to[bend right=20] ({\const-1},1);
\draw [thick, dashed] ({-(\const-1)},1) to[bend left=5] ({\const-1},1);

\begin{scope}[yshift = -3.5cm]
\draw [fill = red] (0,\const) circle (1/16);
\draw [transparent] (0,-0.3) circle (1/16);
\end{scope}
\end{tikzpicture}
\caption{The singularity corresponds to an elliptic-elliptic point.}
\end{subfigure} 
\hspace{2pt}
\begin{subfigure}[b]{.3\textwidth}
\centering
\def\const{2.5}
\begin{tikzpicture}
\draw [thick, ->] (-2,2)--(-1.5,2.5);
\draw (-1.5,2.5) node [above left] {$H_t^{\text{red},j}$};
\draw [domain=0:\const,samples=200,color=blue,thick] plot ({sqrt(\x*(\const-\x)^2)}, {\x});
\draw [domain=0:\const,samples=200,color=blue,thick] plot ({-sqrt(\x*(\const-\x)^2)}, {\x});
\draw [thick, red] (0,\const) to[bend right=20] ({\const-1},{1});
\draw [thick, red, dashed] (0,\const) to[bend left=5] ({\const-1},{1});
\draw [fill = red] (0,\const) circle (1/16);
\draw [thick] (0,0) to[bend left=20] ({-sqrt(\const-1)},{\const-1});
\draw [thick, dashed] (0,0) to[bend right=5] ({-sqrt(\const-1)},{\const-1});

\begin{scope}[yshift = -3.5cm]

\draw [thick, red] (0,\const) ..controls +(.3,-0.1) and +(.3,.3).. ({\const-1-.3},{1+.3}) .. controls +(-.3,-.3) and +(0.1,-.3) .. (0,\const);
\draw [thick, red] (0,\const) ..controls +(-0.1,0.1) and +(1,1).. ({\const-1},1) .. controls +(-.8,-.8) and +(-0.1,0.1) .. (0,\const);
\draw [fill = red] (0,\const) circle (1/16);
\draw [thick, red] ({\const-1},1) to[bend right=20] ({\const-1-0.3},{1+0.3});
\draw [thick, red, dashed] ({\const-1},1) to[bend left=20] ({\const-1-0.3},{1+0.3});
\draw [transparent] (0,-0.3) circle (1/16);

\end{scope}
\end{tikzpicture}
\caption{The singularity corresponds to a focus-focus point.\\}
\end{subfigure}
\hspace{2pt}
\begin{subfigure}[b]{.3\textwidth}
\centering
\def\const{2.5}
\begin{tikzpicture}
\draw [thick, ->] (-2,2)--(-1.5,2);
\draw (-1.5,2.1) node [above] {$H_{1/2}^{\text{red},j}$};
\draw [domain=0:\const,samples=200,color=blue,thick] plot ({sqrt(\x*(\const-\x)^2)}, {\x});
\draw [domain=0:\const,samples=200,color=blue,thick] plot ({-sqrt(\x*(\const-\x)^2)}, {\x});
\draw [thick, red] (0,0) to[bend left=20] (0,\const);
\draw [thick, red, dashed] (0,0) to[bend right=10] (0,\const);
\draw [fill = red] (0,\const) circle (1/16);
\draw [thick] ({-sqrt(\const-1)},{0.5*(-sqrt((\const)^2+2*\const-3)+\const+1)}) to[bend left=20] ({-sqrt(\const-1)},{\const-1});
\draw [thick, dashed] ({-sqrt(\const-1)},{0.5*(-sqrt((\const)^2+2*\const-3)+\const+1)}) to[bend right=5] ({-sqrt(\const-1)},{\const-1});

\begin{scope}[yshift = -3.5cm]
\draw [thick, red] (0,\const) ..controls +(1,-0.5) and +(1,0).. (0,-0.3) .. controls +(-1,0) and +(-1,-0.5) .. (0,\const);
\draw [thick, red] (0,\const) ..controls +(0.1,0.1) and +(1,0).. (0,0) .. controls +(-1,0) and +(-0.1,-0.1) .. (0,\const);
\draw [fill = red] (0,\const) circle (1/16);
\draw [thick, red] (0,0) to[bend right=20] (0,-0.3);
\draw [thick, red, dashed] (0,0) to[bend left=30] (0,-0.3);
\draw [transparent] (0,-0.3) circle (1/16);
\end{scope}

\end{tikzpicture}
\caption{The singularity corresponds to a focus-focus point.\\}
\end{subfigure}
\end{center}
\caption{Top row: the singular level (in red, with the red circle indicating the singular point) and a regular level (in black) for the function $H_j^{\text{red}}$ (corresponding to the projection along a given line), for different choices of $H$, on the reduced space $M_j^{\text{red}}$ (in blue) where $j$ is a singular value of $J$. Bottom row: the corresponding singular fibers.}
\label{fig:teardrops}
\end{figure}

In what follows, by a \emph{simple semitoric transition family} we mean a semitoric transition family (as in Definition~\ref{dfn:transition_fam}) which is comprised of simple
semitoric systems for all $t\in[0,1]\setminus\{t^-, t^+\}$.

\begin{lm}\label{lem:onlyrankzeropt}
 The transition point in a simple semitoric transition family is the only rank zero point in its $J$-fiber.
\end{lm}

\begin{proof}
Let $(M,\om,(J,H_t))$ be a simple semitoric family with transition point $p$ and transition times $t^-$, $t^+$, and suppose
that $q\in J^{-1}(J(p))\setminus\{p\}$ is also a rank zero point.
By Lemma~\ref{lm:morse_family} the type of $q$ cannot change without it becoming degenerate, which cannot occur in a semitoric transition
family, so it is either focus-focus or elliptic-elliptic for all $t\in [0,1]$.
If $q$ is focus-focus then the system is not simple for $t\in(t^-,t^+)$.
If $q$ is elliptic-elliptic for some value of $t$, then it is either the maximum or minimum of $(H_t)_{|J^{-1}(J(p))}$ for all $t$. Then either for $t<t^-$
or $t>t^+$ we have that $p$ and $q$ are both at the maximum or minimum of $(H_t)_{|J^{-1}(J(p))}$, contradicting the fact
that an elliptic-elliptic point is the only point in its fiber in a semitoric system.
\end{proof}

We conclude by explaining the way in which the unmarked semitoric polygon of the systems in a simple semitoric transition family are related before and after the transition times in a special case. Define a function $\phi_j$ which removes the $j$-th line and the $j$-th cut from a representative of an unmarked semitoric polygon, that is
\[
\phi_j (\De, \vec{\ell}, \vec{\epsilon}) = \big(\De, (\ell_1, \ldots, \ell_{j-1}, \ell_{j+1}, \ldots, \ell_s), (\epsilon_1,\ldots, \epsilon_{j-1}, \epsilon_{j+1}, \ldots, \epsilon_s)\big).
\]
Then, given an unmarked semitoric polygon $\Delta_u$ we define
\[\De_u^{\epsilon_j=\pm 1} = \{\phi_j\big( \De, \vec{\ell}, \vec{\epsilon}\big) \mid (\De, \vec{\ell}, \vec{\epsilon})\in\De_u, \epsilon_j = \pm 1\},\]
see Figure~\ref{fig:phij}. 
Recall that by Lemma~\ref{lem:samepoly} the unmarked semitoric polygon of a semitoric family is fixed in any interval between two consecutive degenerate times.

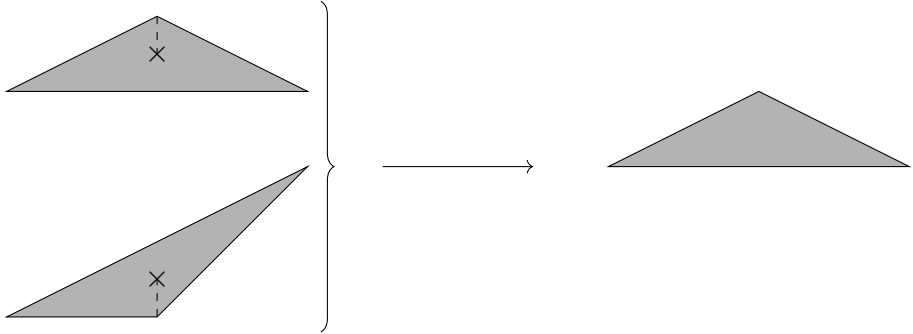
\begin{figure}
\begin{center}
\begin{tikzpicture}
\draw [->] (3,-1) -- (5,-1);

\draw[decoration={brace,raise=5pt, amplitude = 5pt},decorate] (2,1.2) -- (2,-3.2);

\filldraw[draw=black, fill=gray!60] (-2,0) node[anchor=north,color=black]{}
  -- (0,1) node[anchor=south,color=black]{}
  -- (2,0) node[anchor=south,color=black]{}
  -- cycle;
\draw [dashed] (0,1) -- (0,1/2);
\draw (0,1/2) node[] {$\times$};

\filldraw[draw=black, fill=gray!60] (-2,-3) node[anchor=north,color=black]{}
  -- (0,-3) node[anchor=south,color=black]{}
  -- (2,-1) node[anchor=south,color=black]{}
  -- cycle;
\draw [dashed] (0,-3) -- (0,-2.5);
\draw (0,-2.5) node[] {$\times$};

\filldraw[draw=black, fill=gray!60] (6,-1) node[anchor=north,color=black]{}
  -- (8,0) node[anchor=south,color=black]{}
  -- (10,-1) node[anchor=south,color=black]{}
  -- cycle;

\end{tikzpicture}
\end{center}
\caption{Two representatives of the same unmarked semitoric polygon $\Delta_u$ with opposite cuts (left) and one representative of $\Delta_u^{\epsilon_1 = 1}$ (right).}
\label{fig:phij}
\end{figure}
	
\begin{lm}
\label{lem:polygons_transition}
Let $(M,\omega,F_t = (J,H_t))$ be a simple semitoric $k$-transition family, $k\geq 1$, with transition point $p$ and transition times $t^-, t^+$. Let $\De_u$ be the unmarked semitoric polygon of the system for $t\in (t^-,t^+)$. Moreover, let $j \in \{1, \ldots, k\}$ be such that for $t \in (t^-, t^+)$, $p$ is the $j$-th focus-focus point of $F_t$ when we order these by $J$-value. Then the unmarked semitoric polygon of the system is $\De_u^{\epsilon_j=1}$ for $t \in [0,t^-)$ and $\De_u^{\epsilon_j=-1}$ for $t \in (t^+,1]$.
\end{lm}

Note that this result implies that as a set (forgetting about the vertical lines and cuts) the unmarked semitoric polygon for $t\in (t^-,t^+)$ is the union of the unmarked semitoric polygons for $t \in [0,t^-)$ and $t \in (t^+,1]$.

\begin{proof}
We may assume without loss of generality that $p$ is a maximum of $(H_t)_{|J^{-1}(J(p))}$ for $t \in [0,t^-)$ and a minimum of $(H_t)_{|J^{-1}(J(p))}$ for $t \in (t^+,1]$. We only prove the statement about the interval $[0,t^-)$, since the other one is similar. Let $t_1\in [0,t^-)$, $t_2\in (t^-, t^+)$, and let
\[(\De,(\ell_1,\ldots, \ell_{j-1}, \ell_{j+1}, \ldots,\ell_k),(\epsilon_1, \ldots, \epsilon_{j-1}, \epsilon_{j+1}, \ldots, \epsilon_k))\in\Delta_u^{\epsilon_j=1},\]
with $\De = (g\circ F_{t_2})(M)$ for some $g\colon \R^2\to \R^2$ as in Equation \eqref{eq:dev_map}. We will construct a representative of the unmarked semitoric polygon for $t=t_1$ that is equal to this. The idea is the same as in the proof of Lemma \ref{lem:samepoly}; since the vertical lines $(\ell_1,\ldots, \ell_{j-1}, \ell_{j+1}, \ldots,\ell_k)$ are the correct ones for any unmarked semitoric polygon for $t=t_1$, it suffices to prove that we can construct a polygon $\widetilde{\De}$ such that
\[(\widetilde{\De},(\ell_1,\ldots, \ell_{j-1}, \ell_{j+1}, \ldots,\ell_k),(\epsilon_1, \ldots, \epsilon_{j-1}, \epsilon_{j+1}, \ldots, \epsilon_k))\]
is an unmarked semitoric polygon associated to $(M,\om,F_{t_1})$ and $\partial^{\mathrm{top}}\widetilde{\De} = \partial^{\mathrm{top}}\De$ (see Equation \eqref{eq:top_bound} for the definition of $\partial^{\mathrm{top}}$). The same arguments as in the proof of Lemma \ref{lem:samepoly} imply that we can construct such a $\widetilde{\Delta}$ whose top left Delzant corner is the same as the one of $\Delta$. Moreover, the vertices of $\partial^{\mathrm{top}}\widetilde{\De}$ and $\partial^{\mathrm{top}}\De$ occur at the same $J$-value, and the change of slope at each of these vertices is the same for $\Delta$ and $\widetilde{\Delta}$. 
Indeed, the same arguments as in the proof of Lemma \ref{lem:samepoly} apply to every vertex on the top boundaries except the ones in the same $J$-fiber as the transition point $p$. For $t=t_2$, $p$ is a focus-focus point whose image lies on the vertical line $\ell_j$, and since $\Delta$ is constructed using the cut $\epsilon_j = 1$, it gives rise to a vertex $P$ on $\partial^{\mathrm{top}}\De$; furthermore, Lemma \ref{lem:changeslope} yields that the change of slope in $\Delta$ at $P$ equals $-1$ since $p$ is the only rank zero point in $J^{-1}(J(p))$ by Lemma~\ref{lem:onlyrankzeropt}.
But for $t=t_1$, $p$ is an elliptic-elliptic point of the system and its image in $\widetilde{\Delta}$ belongs to $\partial^{\mathrm{top}}\widetilde{\De}$, and Lemma \ref{lem:changeslope} again yields that the change of slope at this vertex equals $-1$, because the weights of the $\S^1$-action generated by $J$ at $p$ are equal to $(-1,1)$ because $p$ is a focus-focus point for $t \in (t^-,t^+)$, see for instance \cite{Zung02} and \cite[Proposition 3.12]{HSS}.
\end{proof}

\begin{rmk}\label{rmk:weightsnot1}
Note that a semitoric system $(M,\omega,(J,H))$ with an elliptic-elliptic point $p$ whose weights for $J$ are not $(-1,1)$ cannot be one of the limiting systems in a semitoric transition family with transition point $p$. For instance, the standard toric system on the second Hirzebruch surface $\Hirzscaled{2}$, whose Delzant polygon is displayed in Figure \ref{fig:poly_Wn}, cannot be part of a semitoric transition family, as the only possible transition point would be the preimage of the corner at $(\alpha,\beta)$; but one of the weights of this point equals $\pm 2$, see Lemma \ref{lem:changeslope}. This explains why in Section \ref{sec:W2examples}, to construct a semitoric transition family on $\Hirzscaled{2}$, we start with a non standard polygon (see Figure \ref{fig:W2_JR}), so as to change these weights.
\end{rmk}

\section{Toric type blowups and blowdowns}
\label{sec:blowups}

In this section, we explain how to perform a toric type blowup of size $\lambda$ at a completely elliptic point of an integrable system $(M,\omega,F)$ to obtain a new integrable system $(\mathrm{Bl}_p(M),\widetilde{\omega}_{\lambda},\widetilde{F}_{\lambda})$ on the blowup manifold,
where we denote by $\blowup{p}{}(M)$ the usual blowup of the manifold $M$ at the point $p$,
see for instance~\cite[Section 7.1]{McDuffSal}. Roughly speaking, this corresponds to performing a toric blowup on any toric manifold given by the Eliasson normal form (as in Theorem~\ref{thm:normalform)}) around a completely elliptic point. Since the normal form depends on several choices, in Proposition \ref{prop:blowupdefined} we prove that this operation is well-defined. 
In Section~\ref{sec:toricchops}, we recall the usual definition of the equivariant blowup of a toric manifold and in Section~\ref{sect:def_blowups} we introduce and explain the toric type blowups in complete detail.
Moreover, in Section~\ref{sec:blowups_semitoric} we study toric type blowups of semitoric systems and explain the relationship with operations on the semitoric polygons. Finally, in Section~\ref{sec:blowups_families} we prove that the toric type blowup of a \family is still a \family. We also define the inverse operation, toric type blowdowns.

\subsection{Toric systems and corner chops}
\label{sec:toricchops}

In $\C^n$ the blowup of the origin can be achieved, roughly speaking, by replacing the origin with $\mathbb{CP}^n$.
Given a symplectic manifold $(M,\om)$ a blowup can be defined via a symplectomorphism from an open set in $M$ into $\C^n$ via the standard blowup of $\C^n$ and the resulting space can also be equipped with a symplectic form depending on a parameter $\lambda>0$, called the \emph{size} of the blowup, see~\cite[Section 7.1]{McDuffSal} for a detailed description of this operation and a proof that it is well-defined.
Near a fixed point in a toric manifold of dimension $2n$ the torus action can be modeled on the standard action of $\mathbb{T}^n$ on $\C^n$, in the sense that there is locally a symplectomorphism into $\C^n$ intertwining the momentum maps,
and performing the standard blowup in $\C^n$ relative to one such map is called a toric blowup.

If $(M,\om,F)$ is a toric system it is well-known (see for instance \cite[Section 3.4]{Can_toric}) that performing
a toric blowup of size $\lambda>0$ on $(M,\om,F)$ corresponds to a \emph{corner chop} of the associated Delzant
polytope $\De = F(M)$, which we define now.
Given vectors $u_1,\ldots,u_n\in\R^n$ and a point $q\in\R^n$ define
\begin{equation}\label{eqn:simp}
 \mathrm{Simp}_q(u_1,\ldots,  u_n) = \left\{q+\sum_{j=1}^n t_j u_j \,\left|\, t_1 \geq 0, \ldots, t_n \geq 0, \sum_{j=1}^n t_j\leq 1\right\}\right..
\end{equation}
If $p$ is a completely elliptic fixed point then $q=F(p)$ is a corner of the Delzant polytope
and primitive integral vectors directing the edges of $\De$ which are adjacent to $q$ generate $\Z^n$,
call them $v_1,\ldots,v_n\in\Z^n$.
Let $V$ be the set of all corners of $\De$, then a corner chop of size $\lambda>0$ at $q$
is possible if
\[
 \mathrm{Simp}_q(\lambda v_1,\ldots, \lambda v_n)\cap V = \{q\}
\]
in which case the corner chop of size $\lambda$ of $\De$ at $q$ is given by the new Delzant polytope
\[
 \widetilde{\De} = \overline{\De\setminus\mathrm{Simp}_q(\lambda v_1,\ldots, \lambda v_n)}.
\]
If $(\mathrm{Bl}_p(M),\widetilde{\om}_\lambda,\widetilde{F}_\lambda)$ is the toric blowup of $(M,\om,F)$ at $p$ of size $\lambda$ then
$\widetilde{F}_\lambda(\mathrm{Bl}_p(M)) = \widetilde{\De}$, see Figure~\ref{fig:CP2blowup}.
The \emph{exceptional divisor} of the blowup is 
\[\Sigma = \widetilde{F}^{-1}_\lambda(\widetilde{\De}\cap\mathrm{Simp}_q(\lambda v_1,\ldots, \lambda v_n)),\]
the preimage of the new face of the Delzant polytope.
If $(M,\om,F)$ can be obtained from $(\check{M},\check{\omega}_\lambda, \check{F}_\lambda)$ by a 
blowup of size $\lambda$ which has exceptional divisor $\Sigma$, then we say that $(\check{M},\check{\omega}_\lambda, \check{F}_\lambda)$ is the
\emph{blowdown of size $\lambda$ of $(M,\om,F)$ at $\Sigma$}.
The blowdown corresponds to performing a \emph{corner unchop} (the inverse of a corner chop, that is gluing the $\mathrm{SL}_n(\Z)$-image of a simplex onto a face) on the face of the Delzant polytope which is the image
of $\Sigma$.
Note that performing a blowup on an $n$-dimensional toric integrable system removes one completely elliptic point and introduces
$n$ new completely elliptic points (corresponding to the vertices incident to the new face of the Delzant polytope), so when performing a
blowdown of a toric system at a submanifold $\Sigma$ we must have that $\Sigma$ is the image of an embedding of $\mathbb{CP}^n$ which
includes exactly $n$ completely elliptic points of the integrable system and has self-intersection $-1$.

\begin{rmk}\label{rmk:SL2Z-length}
 Recall that given a line segment $\ell\subset\R^n$ with rational slope the \emph{$\mathrm{SL}_n(\Z)$\--length of $\ell$} 
 is the unique $a\in\R_{\geq 0}$ such that $A \ell$ is the line segment from the origin to $(a,0,\ldots,0)$
 for some $A\in\mathrm{SL}_n(\Z)$. If $q$ is the corner of a Delzant polytope then it is possible to perform a corner chop
 of size $\lambda$ at $q$ if and only if each edge incident to $q$ has $\mathrm{SL}_n(\Z)$\--length strictly greater than $\lambda$.
\end{rmk}

\begin{figure}
\begin{center}
\begin{tikzpicture}

\draw [->] (2.25,1) -- (3.5,1);
\draw [->] (6.25,1) -- (7.5,1);

\filldraw[draw=black, fill=gray!60] (0,0) node[anchor=north,color=black]{}
  -- (0,2) node[anchor=south,color=black]{}
  -- (2,0) node[anchor=north,color=black]{}
  -- cycle;

\filldraw[draw=black, fill=gray!60] (4,0) node[anchor=north,color=black]{}
  -- (4,2) node[anchor=south,color=black]{}
  -- (6,0) node[anchor=north,color=black]{}
  -- cycle;
  
\filldraw[draw=black, fill=gray!60, pattern=north east lines] (4,1) node[anchor=north,color=black]{}
  -- (4,2) node[anchor=south,color=black]{}
  -- (5,1) node[anchor=north,color=black]{}
  -- cycle;
  
\filldraw[draw=black, fill=gray!60] (8,0) node[anchor=north,color=black]{}
  -- (8,1) node[anchor=south,color=black]{}
  -- (9,1) node[anchor=south,color=black]{}
  -- (10,0) node[anchor=north,color=black]{}
  -- cycle;  
\end{tikzpicture}
\end{center}
\caption{The momentum map image of $\CP^2$ (left) and a blowup of $\CP^2$ of size $\lambda \in (0,1)$ (right), obtained
by ``chopping off'' the top corner. In the middle image $\mathrm{Simp}_q(\lambda v_1,\lambda v_2)$ is shaded.}
\label{fig:CP2blowup}
\end{figure}
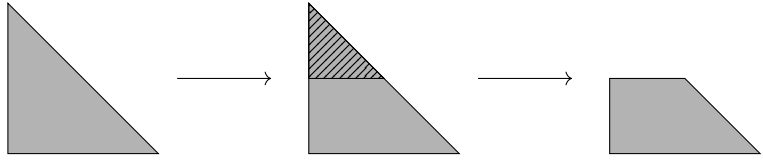

\subsection{Toric type blowups of integrable systems}
\label{sect:def_blowups}

Let $(M,\om,F)$ be an integrable system of dimension $2n$ and let $p\in M$ be a completely elliptic fixed point of $F$.
Then there exists an open neighborhood $U$ of $p$, a symplectomorphism $\chi\colon U\to\C^n$, and a local diffeomorphism $g\colon (\R^n,F(p))\to(\R^n,0)$ such that $g\circ F = \Phi \circ \chi$, where $\Phi$ is the momentum map for a symplectic $\T^n$-action on $\C^n$ which fixes only the origin (such a momentum map is equal to the standard one composed on the left with an element of $\mathrm{GL}_n(\Z)$), so that the diagram in Figure~\ref{fig:normalform} commutes.
\begin{figure}
\centering
\begin{subfigure}[b]{.45\linewidth}
\centering
\begin{tikzcd}
(U,p) \arrow[r,"\chi"] \arrow[d,"F"]  &  (\C^n,0) \arrow[d,"\Phi"]\\
(\R^n,F(p)) \arrow[r,"g"]             &  (\R^n,0)
\end{tikzcd}
\caption{Locally, a completely elliptic point can be modeled after $\C^n$ with the usual action of $\T^n$.}
\label{fig:normalform}
\end{subfigure} 
\hspace{20pt}
\begin{subfigure}[b]{.45\linewidth}
\centering
\begin{tikzcd}
\widetilde{U} \arrow[hookrightarrow]{r}{\rho} \arrow[d,"\pi^U"]  &  \mathrm{Bl}_p(M) \arrow[d,"\pi"]\\
U \arrow[hookrightarrow]{r}{\iota}                               &  M
\end{tikzcd}
\caption{Relationships between the blowup in $U$ and the blowup in $M$.}
\label{fig:UMblowup}
\end{subfigure}
 \caption{Diagrams for defining the toric type blowup.}
\end{figure}
In particular, this implies that $(g\circ F)_{|U}$ is a momentum map for a $\T^n$\--action.
Let $\blowup{p}{}(M)$ be the blowup of $M$ at $p$, so there is a projection map $\pi\colon\blowup{p}{}(M)\to M$ which is a 
diffeomorphism on $\blowup{p}{}(M)\setminus\pi^{-1}(p)$. 
Let $\lambda>0$ such that $B_\lambda=\{z\in\C^n\mid ||z||^2<2\lambda\}\subset\chi(U)$; we will endow
$\blowup{p}{}(M)$ with a symplectic form depending on $\lambda$, which we
call the \emph{size} of the blowup.
Let $V \subsetneq U$ be an open set such that $B_\lambda \subset \chi(V)$.
Since $(U,\om_{| U},(g\circ F)_{| U})$ is an (open) toric manifold we may perform a blowup of size $\lambda$
at $p$, which amounts to performing a corner chop of size $\lambda$ to the corner of the image $(g\circ F)(U)$ at the origin in $\R^n$, resulting in a new open toric manifold $(\widetilde{U}, \om_{\widetilde{U}, \lambda},\mu)$ with
exceptional divisor $\Sigma_U$.
There exists an embedding $\rho\colon \widetilde{U}\to\blowup{p}{}(M)$ determined uniquely by $\pi \circ \rho = \iota\circ\pi_U$ where 
$\pi_U\colon \widetilde{U}\to U$ is the projection associated to the blowup and $\iota\colon U\to M$ is inclusion, as
in Figure~\ref{fig:UMblowup}.
Then (see for instance~\cite[Section 7.1]{McDuffSal}) there exists a symplectic form 
$\widetilde{\om}_\lambda$ on $\blowup{p}{}(M)$ such that
$\widetilde{w}_{\lambda|_{\blowup{p}{}(M)\setminus \pi^{-1}(V)}} = \pi^* \om$
and $\rho^* \widetilde{\om}_\lambda = \om_{\widetilde{U},\lambda}$.
We define a momentum map 
\begin{equation}\label{eqn:Flambda}
\widetilde{F}_\lambda (m)= 
\begin{cases}
 (F\circ\pi)(m) &\textrm{ if } \pi(m)\notin V,\\
 (g^{-1}\circ\mu\circ\rho^{-1})(m)  & \textrm{ if }\pi(m)\in U.
\end{cases}
\end{equation}
Again, we refer the reader to~\cite[Section 7.1]{McDuffSal} to see that $\widetilde{F}_\lambda$ is smooth and 
that $F\circ\pi = g^{-1}\circ\mu\circ\rho^{-1}$ in $U\setminus V$, and that
$(\blowup{p}{}(M), \widetilde{\omega}_{\lambda}, \widetilde{F}_\lambda)$ is an integrable system. Since this construction depends on several choices, we need to prove that the blowup of size $\lambda$ is well-defined.
\begin{prop}\label{prop:blowupdefined}
 Given $(M,\om,F)$ with $p\in M$ an elliptic fixed point, and $\lambda>0$, the triple
 $(\blowup{p}{}(M),\widetilde{\om}_\lambda, \widetilde{F}_\lambda)$ is a well-defined
 integrable system.
\end{prop}

\begin{proof}
 The only choices in the construction are $U$, $V$, $\Phi$, and $g$; so suppose
 we have two such sets of choices $U$, $V$, $\Phi$, $g$ and $U'$, $V'$, $\Phi'$, $g'$.
 Since $g$ is determined by a choice of action variables (see Remark~\ref{rmk:gcycles}),
 $g' = A\circ g$ for some $A\in\mathrm{GL}_n(\Z)$.
 Given $\lambda$ and the image of $g\circ F$ (respectively $g'\circ F$) the image
 of the blowup of the open toric manifold $(U,\om_{|U}, (g\circ F)_{| U})$ 
 (respectively $(U,\om_{|U}, (g'\circ F)_{|U})$) is determined, since it is 
 the original set minus a simplex centered at the origin, as described in Section~\ref{sec:toricchops}, call
 it $T$ (respectively $T'$). Then, since $g' = A\circ g$ we see that $T' = A(T)$. To perform a blowup of size $\lambda$
 we only need $U$ and $U'$ to include $(\Phi\circ\chi)^{-1}(T)$
 and $ (\Phi'\circ\chi')^{-1}(A(T))$ respectively, but
 \[ 
  (\Phi\circ\chi)^{-1}(T) = (g\circ F)^{-1}(T) = (g'\circ F)^{-1}(A(T)) = (\Phi'\circ\chi')^{-1}(A(T))
 \]
 so we may assume that $U=U'$ and $V = V'$. 
 Let $(\widetilde{U}, \om_{\widetilde{U}, \lambda},\mu)$ and $(\widetilde{U}', \om_{\widetilde{U}', \lambda},\mu')$ be the open toric manifolds obtained after performing toric blowups of size $\lambda$ on $(U,\om, g\circ F)$ and $(U,\om,g'\circ F)$ respectively.
 Note that the blowup of $(U,\om,A\circ g\circ F)$ is $(\widetilde{U},\om_{\widetilde{U}, \lambda},A\circ\mu)$, unique up to equivariant symplectomorphism (for instance see~\cite[Proposition 6.5]{KarLer}),
 and thus since $g'\circ F = A \circ g \circ F$ there exists a symplectomorphism
 $\Psi\colon (\widetilde{U}, \om_{\widetilde{U}, \lambda}) \to (\widetilde{U}', \om_{\widetilde{U}', \lambda})$ such that $\mu'\circ\Psi = A\circ\mu$.
 Let $\rho\colon \widetilde{U}\to\blowup{p}{}(M)$ and  $\rho'\colon \widetilde{U}'\to\blowup{p}{}(M)$ be the embeddings associated to the two blowups.
 Notice that $\pi_U'\circ \Psi\colon\widetilde{U}\to U$ sends $\Sigma_U$ to $p$ and is a diffeomorphism on $\widetilde{U}\setminus\Sigma_U$, so it is equal to $\pi_U$ up to a diffeomorphism, therefore we may choose $\Psi$ so that $\pi_U'\circ\Psi=\pi_U$, except that $\Psi$ is only a diffeomorphism \emph{a priori}. This equality implies that $\pi \circ \rho = \iota\circ\pi_U$ and $\pi \circ \rho'\circ\Psi = \iota\circ\pi_U$, hence $\rho=\rho'\circ\Psi$ on the dense set $\widetilde{U}\setminus\Sigma_U$, and thus on all of $\widetilde{U}$. But by definition of the symplectic structure on the blowups, we have that
\[ \om_{\widetilde{U},\lambda} = \rho^* \widetilde{\omega}_{\lambda} = \left( \rho' \circ \Psi \right)^* \widetilde{\omega}_{\lambda} = \Psi^* \left( (\rho')^* \widetilde{\omega}_{\lambda} \right) = \Psi^* \om_{\widetilde{U}',\lambda}, \]
so $\Psi$ is in fact a symplectomorphism. 
Thus, the two symplectic forms defined on $\blowup{p}{}(M)$ are equal.
Furthermore, in Equation~\eqref{eqn:Flambda} we see that only the second case depends on the choices, and if $\pi(m)\in U$ we have 
 \[ (g'^{-1} \circ \mu' \circ \rho^{-1}) (m)= (g^{-1} \circ A^{-1} \circ A \circ \mu \circ \Psi^{-1} \circ \Psi \circ \rho^{-1}) (m)= (g^{-1}\circ\mu \circ \rho^{-1}) (m)\]
 so the new momentum maps agree as well.
\end{proof} 

Because of Proposition~\ref{prop:blowupdefined} we can make the following definition:

\begin{dfn}\label{def:blowup}
 The integrable system $(\blowup{p}{}(M),\widetilde{\om}_\lambda, \widetilde{F}_\lambda)$ described above is the \emph{toric type
 blowup of $(M,\om,F)$ at $p$ of size $\lambda>0$}, and the \emph{exceptional divisor}
 of this blowup is the symplectically embedded sphere $\Sigma = \rho(\Sigma_U)\subset \blowup{p}{}(M)$.
\end{dfn}

In Figure \ref{fig:torictypeblowup}, we show what happens for the images of the momentum maps under a toric type blowup.

\begin{dfn}\label{def:blowdown}
 If $(M,\om,F)$ can be obtained from $(\check{M},\check{\om}_\lambda, \check{F}_\lambda)$
 by performing a blowup of size $\lambda$ which has exceptional divisor $\Sigma$, then we say
 that $(\check{M},\check{\om}_\lambda, \check{F}_\lambda)$ is the \emph{toric type
 blowdown of $(M,\om,F)$ at $\Sigma$ of size $\lambda$}.
\end{dfn}

As in the toric case, this implies that a blowdown is only possible at a submanifold $\Sigma$ with self-intersection $-1$
which is the image of an embedding of $\mathbb{CP}^n$ which contains exactly $n$ singular
points of the integrable system which are all of completely elliptic type.

\begin{rmk}
In the case that $(M,\om,F)$ is a toric integrable system then $g$ can be taken to be a translation on $\R^n$ in the above construction.
\end{rmk}

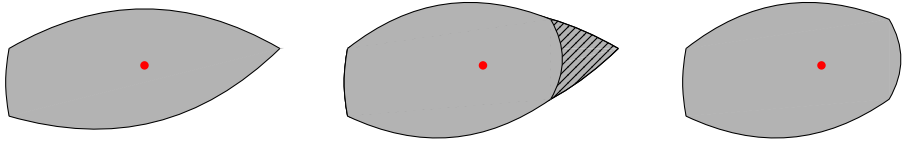
\begin{figure}
\begin{center}
\begin{tikzpicture}[scale=.9]

\filldraw[draw=gray!60,fill=gray!60] (0,0) node[anchor=north,color=black]{}
  -- (0,1) node[anchor=south,color=black]{}
  -- (4,1) node[anchor=south,color=black]{}
  -- cycle;  

\draw[bend left=10, fill = gray!60] (0,0) to (0,1);
\draw[bend left, fill = gray!60] (0,1) to (4,1);
\draw[bend left, fill = gray!60] (4,1) to (0,0);

\fill[red] (2,0.75) circle (1/16);	

\fill[fill=gray!60] (5,0) node[anchor=north,color=black]{}
  -- (5,1) node[anchor=south,color=black]{}
  -- (9,1) node[anchor=south,color=black]{}
  -- cycle;  
\draw[bend left=10, fill = gray!60] (5,0) to (5,1);
\draw[bend left, fill = gray!60] (5,1) to (9,1);
\draw[bend left, fill = gray!60] (9,1) to (5,0);  
  
\fill[fill=gray!60,pattern=north east lines] (5,0) node[anchor=north,color=black]{}
  -- (5,1) node[anchor=south,color=black]{}
  -- (9,1) node[anchor=south,color=black]{}
  -- cycle;  
\draw[bend left=10, fill = gray!60,pattern=north east lines] (5,0) to (5,1);
\draw[bend left, fill = gray!60,pattern=north east lines] (5,1) to (9,1);
\draw[bend left, fill = gray!60,pattern=north east lines] (9,1) to (5,0);

\filldraw[draw=gray!60,fill=gray!60] (5,0) node[anchor=north,color=black]{}
  -- (5,1) node[anchor=south,color=black]{}
  -- (8,1.43) node[anchor=south,color=black]{}
  -- (8,0.25) node[anchor=south,color=black]{}
  -- cycle;  
\draw[bend left=10, fill = gray!60] (5,0) to (5,1);
\draw[bend left, fill = gray!60] (5,1) to (8,1.43);
\draw[bend left, fill = gray!60] (8, 1.43	) to (8,0.25);
\draw[bend left, fill = gray!60] (8,0.25) to (5,0);
\fill[red] (7,0.75) circle (1/16);

\filldraw[draw=gray!60,fill=gray!60] (10,0) node[anchor=north,color=black]{}
  -- (10,1) node[anchor=south,color=black]{}
  -- (13,1.43) node[anchor=south,color=black]{}
  -- (13,0.25) node[anchor=south,color=black]{}
  -- cycle;  
\draw[bend left=10, fill = gray!60] (10,0) to (10,1);
\draw[bend left, fill = gray!60] (10,1) to (13,1.43);
\draw[bend left, fill = gray!60] (13, 1.43) to (13,0.25);
\draw[bend left, fill = gray!60] (13,0.25) to (10,0);
\fill[red] (12,0.75) circle (1/16);	

\end{tikzpicture}
\end{center}
\caption{The momentum map image after performing a toric type blowup at the far right elliptic-elliptic point where the dot represents a focus-focus value. Compare with Figure~\ref{fig:CP2blowup}.}
\label{fig:torictypeblowup}
\end{figure}

\subsection{Toric type blowups of simple semitoric systems}
\label{sec:blowups_semitoric}

Analogous to the case of toric systems and corner chops described in Section~\ref{sec:toricchops}, performing a toric type blowup
on a semitoric system changes the associated marked Delzant semitoric polygon in a relatively simple way.
Recall that associated to a simple semitoric system is a marked Delzant semitoric polygon $\poly{(M,\omega,F)}$ (described
in Section~\ref{sec:semitoric}), which is the $G_s\times\mathcal{T}$-orbit of a marked
weighted polygon $(\De, \vec{c}, \vec{\epsilon})$; in this section we define an operation on such an object, which we again
call a corner chop, which consists in performing a corner chop coherently on every representative of the semitoric polygon invariant, see Figure~\ref{fig:semitoricchop}.

Recall the action of $G_s\times \mathcal{T}$ on pairs of marked Delzant semitoric polygons and points given
in Equation~\eqref{eqn:action_Deq}. Let $Q$ be an equivalence class of this action, which amounts to
a marked Delzant semitoric polygon with a point coherently chosen for each representative. 
If there exists a representative
$((\De, \vec{c}, \vec{\epsilon}),q)$ of $Q$ such that $q$ is a Delzant corner of $\De$ and 
\[
 \mathrm{Simp}_q(\lambda v_1, \lambda v_2) \cap (\mathcal{L}_{(\De,\vec{c},\vec{\epsilon})}\cup V) = \{q\}
\]
where $\mathrm{Simp}_q$ is as in Equation~\eqref{eqn:simp}, $\mathcal{L}_{(\De,\vec{c},\vec{\epsilon})}$ is as in Equation~\eqref{eqn:Lcuts}, $V$ is the set of all corners
of $\De$, and $v_1,v_2\in\Z^2$ are the primitive vectors directing the edges of $\De$ adjacent to $q$, then we define the \emph{semitoric corner chop of size $\lambda$ of $Q$} to be the $G_s\times \mathcal{T}$-orbit of $(\overline{\De\setminus(\mathrm{Simp}_q(\lambda v_1, \lambda v_2))}, \vec{c}, \vec{\epsilon})$.

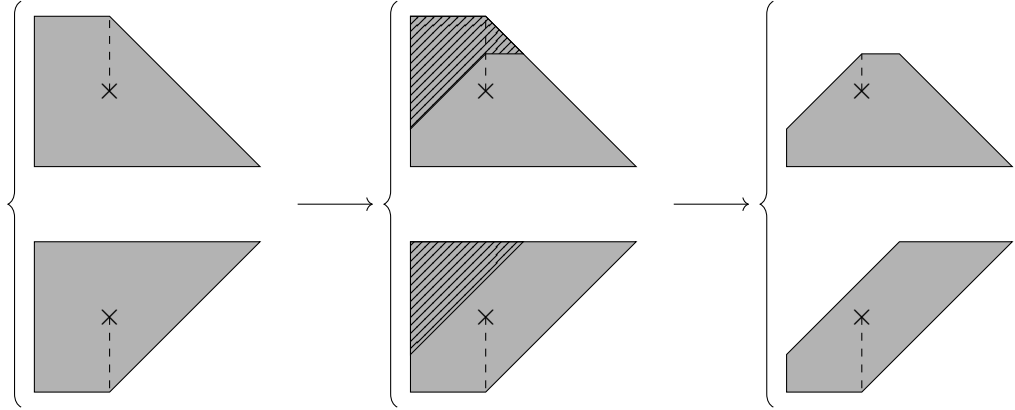
\begin{figure}
\begin{center}
\begin{tikzpicture}
\draw [->] (3.5,-0.5) -- (4.5,-0.5);
\draw [->] (8.5,-0.5) -- (9.5,-0.5);

\draw[decoration={brace,raise=5pt, amplitude = 5pt},decorate] (0,-3.2) -- (0,2.2);
\draw[decoration={brace,raise=5pt, amplitude = 5pt},decorate] ( 5,-3.2) -- (5,2.2);
\draw[decoration={brace,raise=5pt, amplitude = 5pt},decorate] ( 10,-3.2) -- (10,2.2);

\filldraw[draw=black, fill=gray!60] (0,0) node[anchor=north,color=black]{}
  -- (0,2) node[anchor=south,color=black]{}
  -- (1,2) node[anchor=south,color=black]{}
  -- (3,0) node[anchor=north,color=black]{}
  -- cycle;
\draw [dashed] (1,1) -- (1,2);
\draw (1,1) node[] {$\times$};

\filldraw[draw=black, fill=gray!60] (5,0) node[anchor=north,color=black]{}
  -- (5,2) node[anchor=south,color=black]{}
  -- (6,2) node[anchor=south,color=black]{}
  -- (8,0) node[anchor=north,color=black]{}
  -- cycle;
\draw [dashed] (6,1) -- (6,2);
\draw (6,1) node[] {$\times$};
  
\filldraw[draw=black, fill=gray!60, pattern=north east lines] (5,0.5) node[anchor=north,color=black]{}
  -- (5,2) node[anchor=south,color=black]{}
  -- (6,2) node[anchor=south,color=black]{}
  -- (6.5,1.5) node[anchor=north,color=black]{}
  -- (6,1.5) node[anchor=north,color=black]{}
  -- cycle;
  
\filldraw[draw=black, fill=gray!60] (10,0) node[anchor=north,color=black]{}
  -- (10,0.5) node[anchor=south,color=black]{}
  -- (11,1.5) node[anchor=south,color=black]{}
  -- (11.5,1.5) node[anchor=south,color=black]{}
  -- (13,0) node[anchor=north,color=black]{}
  -- cycle;
\draw [dashed] (11,1) -- (11,1.5);
\draw (11,1) node[] {$\times$};
  
\filldraw[draw=black, fill=gray!60] (0,-3) node[anchor=north,color=black]{}
  -- (0,-1) node[anchor=south,color=black]{}
  -- (3,-1) node[anchor=north,color=black]{}
  -- (1,-3) node[anchor=south,color=black]{}
  -- cycle;
\draw [dashed] (1,-2) -- (1,-3);
\draw (1,-2) node[] {$\times$};

\filldraw[draw=black, fill=gray!60] (5,-3) node[anchor=north,color=black]{}
  -- (5,-1) node[anchor=south,color=black]{}
  -- (8,-1) node[anchor=north,color=black]{}
  -- (6,-3) node[anchor=south,color=black]{}
  -- cycle;
\draw [dashed] (6,-2) -- (6,-3);
\draw (6,-2) node[] {$\times$};
  
\filldraw[draw=black, fill=gray!60, pattern=north east lines] (5,-2.5) node[anchor=north,color=black]{}
  -- (5,-1) node[anchor=south,color=black]{}
  -- (6.5,-1) node[anchor=north,color=black]{}
  -- (6,-1.5) node[anchor=north,color=black]{}
  -- cycle;
  
\filldraw[draw=black, fill=gray!60] (10,-3) node[anchor=north,color=black]{}
  -- (10,-2.5) node[anchor=south,color=black]{}
  -- (11.5,-1) node[anchor=south,color=black]{}
  -- (13,-1) node[anchor=north,color=black]{}
  -- (11,-3) node[anchor=south,color=black]{}
  -- cycle;  
\draw [dashed] (11,-2) -- (11,-3);
\draw (11,-2) node[] {$\times$};

\end{tikzpicture}
\end{center}
\caption{An example of performing a semitoric corner chop. The set $\mathrm{Simp}_q(\lambda v_1, \lambda v_2)$
is shaded in the lower middle figure. Note that the removed region only has to be a simplex for one representative of
the polygon.}
\label{fig:semitoricchop}
\end{figure}

\begin{lm}\label{lem:blowups_semitoric_chops}
The marked Delzant semitoric polygon associated to the toric type blowup of size $\lambda$ of a simple semitoric system $(M,\omega,F)$ at the completely elliptic fixed point $p$ is the marked Delzant semitoric polygon obtained by performing a semitoric corner chop of size $\lambda$ on $[((\De,\vec{c},\vec{\epsilon}),(g_{\vec{\epsilon}}\circ F)(p))]$
where $g_{\vec{\epsilon}}$ is the developing map associated to $(\De,\vec{c},\vec{\epsilon})$ so $\De = (g_{\vec{\epsilon}}\circ F)(M)$.
Moreover, if the marked semitoric polygon admits a semitoric corner chop of size $\lambda$ at the corner $F(p)$ then a blowup of size $\lambda$ may be performed at $p$.
\end{lm}

\begin{proof}
Suppose that $(M,\om,F)$ admits a toric type blowup of size $\lambda$ at $p$; by definition this means there exists $g$, $\chi$, $U$ and $\Phi$ as in the beginning of Section~\ref{sect:def_blowups} such that $g\circ F = \Phi\circ\chi$ and $B_\lambda\subset\chi(U)$. Choosing $\vec{\epsilon}$ so that the cuts do not intersect $(F\circ\chi^{-1})(B_\lambda)$ (this is possible since elements of this set cannot be both directly above and below a focus-focus value) and extending this $g$ we obtain a developing map $g_{\vec{\epsilon}}$ as in the definition of the semitoric corner chop above.
The result follows since a toric type blowup is locally the same as a toric blowup
relative to the toric momentum map $g_{\vec{\epsilon}}\circ F$ on $U$, which corresponds to a corner chop of the polytope near $(g_{\vec{\epsilon}}\circ F)(p)$.

To prove the converse, suppose that the marked semitoric polygon admits a semitoric corner chop of size $\lambda$ at the corner corresponding to $p$ so we can find a $\vec{\epsilon}$ and $g_{\vec{\epsilon}}$ as in the definition. Let $R=\mathrm{Simp}_{g_{\vec{\epsilon}}\circ F(p)}(\lambda v_1, \lambda v_2)$,
the simplex which will be removed when the corner chop is performed. Taking the preimage under $g_{\vec{\epsilon}}\circ F$ of a small open neighborhood of $R$ we obtain an open set $U\subset M$ whose only singularities are a completely elliptic fixed point and the two families of elliptic-regular points adjacent to it. 
Thus the elliptic-elliptic normal form can be used in the entire set $U$, and hence, since the size of $R$ guarantees that $B_\lambda$ is a subset of the open set in $\C^2$ used in the local normal form, a blowup of size $\lambda$ can be performed with respect to $U$.
\end{proof}

\begin{rmk}
 In~\cite{KPP_min} the notion of a blowup of a semitoric system is defined from the operation of performing a semitoric corner chop on semitoric polygons, and agrees with the notion in this section.
 In the present paper, we instead develop the more general notion of a blowup at a completely elliptic point so
 that we can perform blowups and blowdowns on integrable systems that are not semitoric (for instance, the systems in semitoric families
 with degenerate points).
\end{rmk}

\subsection{Toric type blowups of \texorpdfstring{\families}{fixed-S1 families}}
\label{sec:blowups_families}

If a \family admits a blowup of size $\lambda$ at some point $p$ for all $t$ then
we may take the blowup of each system in the family to produce a new set of integrable systems.
In this section we show that a collection of integrable systems produced in such a way can 
be associated with a \family, which essentially amounts to checking that the new momentum maps are still smooth
and still have a first component which is independent of $t$.

\begin{thm}\label{thm:blowups}
Let $(M,\om,F_t)$, $0\leq t\leq 1$, be a \family and let $\lambda>0$. Let $(p_t)_{0 \leq t \leq 1}$ be a continuous family of points in $M$ such that for every $t \in [0,1]$, $p_t$ is an elliptic-elliptic point of $F_t$ and $(M,\omega,F_t)$ admits a toric type blowup of size $\lambda$ at $p_t$. Then there exists a \family
 $(\mathrm{Bl}_{p_t}(M),\widetilde{\om}_{\lambda,t}, \widetilde{F}_{\lambda,t})$, $0\leq t\leq 1$, such that for each $t\in [0,1]$,  
 $ (\mathrm{Bl}_{p_t}(M),\widetilde{\om}_{\lambda,t}, \widetilde{F}_{\lambda,t}) $ is symplectomorphic to 
 the blowup of size $\lambda$ of $(M,\om,F_t)$ at $p_t$ by a symplectomorphism which intertwines the momentum maps.
\end{thm}

Note that $p_t$ can only depend on $t$ in the case that it belongs to a fixed surface of $J$ for all $t$. This is the case, for instance, in one step of the proof of Theorem \ref{thm:Wkdetailed}.

\begin{proof}

Let $(M,\om,F_t)$, $\lambda$, and $p_t$ be as in the statement of the proposition. 
By changing $F_t$ to $F_t - F_t(p_t)$ if necessary, we may assume that $F_t(p_t) = (0,0)$ for all $t$.
 Then, by Lemma~\ref{lem:samecorner}, there exist a neighborhood $U_t$,
 local diffeomorphisms $g_t\colon (\R^2,0)\to(\R^2,0)$ such that $g_t(x,y) = (x,g^{(2)}_t(x,y))$ where $\dpar{g_t^{(2)}}{y}(x,y) > 0$, and
 symplectomorphisms $\chi_t\colon U_t\to \C^2$ (onto their images)
 such that
 \[  (g_t \circ F_t)_{| U_t} = \Phi \circ \chi_t \]
 where $\Phi \colon \C^2\to \R^2$ is a momentum map for a $\T^2$-action on $\C^2$ which fixes the origin only.
 This implies that $J = \Phi^{(1)}\circ\chi_t$.
 Also, since we have assumed that a blowup of size $\lambda$ is possible for each $t$, we see that
 $\Phi(B_{\lambda})\subset (g_t\circ F_t)(U_t)$, so we may choose $U_t$ such that $(g_t\circ F_t)(U_t) = \Omega\cap \Phi(\C^2)$ for some fixed
 open set $\Omega$.

{\bf Performing blowup and defining $\widetilde{F}_{\lambda,t}$:}
Performing a blowup of $(M,\om,F_t)$ at $p_t$ of a fixed size $\lambda$ at each $t$ yields
 a new integrable system $(\widetilde{M}_t,\widetilde{\om}_{\lambda,t},\widetilde{G}_{\lambda,t})$ with projection
 maps $\pi_t\colon \widetilde{M}_t\to M$.
 Let $\widetilde{\Phi}_{\lambda}=(\widetilde{\Phi}^{(1)}_{\lambda},\widetilde{\Phi}^{(2)}_{\lambda})$ be the toric momentum
 map on $\widetilde{\C^2}$, the blowup of size $\lambda$ at the origin relative to $\Phi$, 
 so there is a map
 $\widetilde{\chi}_t\colon \widetilde{U}_t\to \widetilde{\C^2}$
 which is a symplectomorphism onto its image
 and such that $\widetilde{G}^{(1)}_{\lambda,t} = \widetilde{\Phi}^{(1)}_\lambda\circ\widetilde{\chi_t}$
  where $\widetilde{G}^{(1)}_{\lambda,t}$ is the first component of $\widetilde{G}_{\lambda,t}$.
 For each $t$ let $\rho_t\colon \widetilde{U}_t\hookrightarrow \widetilde{M}_t$ be the associated embedding,
 as in Definition~\ref{def:blowup}.
 By uniqueness of blowups for each $t$ there exist symplectomorphisms $\Psi_t^M\colon \widetilde{M}_0\to \widetilde{M}_t$
 and $\Psi_t^U\colon \widetilde{U}_0\to \widetilde{U}_t$ which satisfy
 $\Psi_t^M\circ\rho_0 = \rho_t\circ\Psi_t^U$,
 and we can choose these so that $\widetilde{\chi}_t^{-1} \circ \widetilde{\chi_0} = \Psi_t^M$
 on the set $\rho_0(\widetilde{U}_0)$ and
\begin{equation}\label{diag:nonlocal-blowup}
\begin{tikzcd}
\widetilde{M}_0 \setminus \rho_0(\widetilde{U}_0) \arrow[rr,"\Psi_t^M"] \arrow[dr,"\pi_0"']
& & \widetilde{M}_t \setminus \rho_t(\widetilde{U}_t) \arrow[dl,"\pi_t"]\\
& M \setminus U_t
\end{tikzcd}
\end{equation}
commutes;
see Figure~\ref{fig:Jdiagram}.
Define $\widetilde{F}_{\lambda,t}\colon \widetilde{M}_0\to\R^2$ by $\widetilde{F}_{\lambda,t} = \widetilde{G}_{\lambda,t}\circ \Psi_t^M$,
 and write $\widetilde{F}_{\lambda,t} = (\widetilde{J}_{\lambda,t},\widetilde{H}_{\lambda,t})$. 
 Also notice that since $(g_t\circ F_t)(U_t) = \Omega\cap\Phi(\C^2)$ for all $t$ we have that
\begin{equation}\label{eqn:sameimage}
 (g_t \circ \widetilde{G}_{\lambda,t} \circ \rho_t)(\widetilde{U}_t) = \Omega\cap \widetilde{\Phi}(\C^2)
\end{equation}
for all $t$.
The relevant maps are shown in Figure~\ref{fig:bendingarrows_diagram}.

\begin{figure}
\centering
\begin{subfigure}[b]{.45\linewidth}
\centering
\[ \begin{tikzcd}
      & \widetilde{U}_0 \arrow[hookrightarrow]{r}{\rho_0} \arrow[d,"\Psi_t^U"]  &  \widetilde{M}_0 \arrow[d,"\Psi_t^M"]\\
      & \widetilde{U}_t \arrow[hookrightarrow]{r}{\rho_t} \arrow[d,"\pi^U_t"]  &  \widetilde{M}_t \arrow[d,"\pi_t"]\\
 \C^2 \arrow{d}{\Phi} & U_t \arrow[hookrightarrow]{r}{\iota_t} \arrow[swap]{l}{\chi_t} \arrow[swap]{d}{{F_t}_{|U_t}}         &  M \arrow{dl}{F_t}\\
 \R^2 & \R^2 \arrow{l}{g_t} &
\end{tikzcd}
\] 
\caption{Performing a blowup in $U_t$ for each $t$ and then identifying all of the resulting manifolds
with $\widetilde{M}_0$.}
\label{fig:Jdiagram}
\end{subfigure} 
\hspace{20pt}
\begin{subfigure}[b]{.45\linewidth}
\centering
\[ \begin{tikzcd}
\widetilde{U}_0 \arrow[bend right = 90, looseness = 2,swap]{ddr}{g_t^{-t}\circ\mu_t\circ\Psi_t^U} \arrow[hookrightarrow]{r}{\rho_0} \arrow[d,"\Psi_t^U"']  &  \widetilde{M}_0 \arrow[swap]{d}{\Psi_t^M} \arrow[bend left = 40]{dd}{\widetilde{F}_{\lambda,t}}\\
\widetilde{U}_t \arrow[hookrightarrow]{r}{\rho_t} \arrow[d,"\mu_t"']   &  \widetilde{M}_t \arrow[swap]{d}{\widetilde{G}_{\lambda,t}}\\
\R^2                                                                  &  \R^2 \arrow[l,"g_t"']
\end{tikzcd}
\]
\caption{The new momentum map is equal to $g_t^{-t}\circ\mu_t\circ\Psi_t^U$ in $\widetilde{U}_0$.\newline}
\label{fig:bendingarrows_diagram}
\end{subfigure}
 \caption{Diagrams related to the proof of Theorem~\ref{thm:blowups}.}
\end{figure}

{\bf Proof that $\widetilde{J}_{\lambda,t}$ is fixed:} Now we must show that $\widetilde{J}_{\lambda,t}$ does
not depend on $t$. Diagram~\eqref{diag:nonlocal-blowup} implies that $\widetilde{J}_{\lambda,t} = \widetilde{J}_{\lambda,0}$ on
$\widetilde{M}\setminus\rho_0(\widetilde{U})$
and $\widetilde{J}_{\lambda,t} = \widetilde{J}_{\lambda,0}$ on $\rho_0(\widetilde{U})$ because
$\widetilde{\chi}_0 = \widetilde{\chi}_t\circ \Psi_t^M$ implies
\[
 \widetilde{J}_{\lambda,0} =\widetilde{G}^{(1)}_{\lambda,0} = \widetilde{\Phi}^{(1)}\circ\widetilde{\chi_0} = \widetilde{\Phi}^{(1)}\circ \widetilde{\chi}_t\circ \Psi_t^M = \widetilde{G}^{(1)}_{\lambda,t}\circ\Psi_t^M = \widetilde{J}_{\lambda,t}.
\]

{\bf Proof that $\widetilde{F}_{\lambda,t}$ is smooth:} Now we only need to show that $\widetilde{F}_{\lambda,t}$ is smooth in $(t,m)$. On $\widetilde{M}\setminus\rho_t(\widetilde{U_t})$
we have that 
\[
(\widetilde{F}_{\lambda,t})_{|\widetilde{M}\setminus\rho_t(\widetilde{U_t})} = (\Psi^M_t \circ \pi_t \circ F_t)_{| \widetilde{M}\setminus\rho_t(\widetilde{U_t})} = (\pi_0 \circ F_t)_{| \widetilde{M}\setminus\rho_t(\widetilde{U_t})}
\]
which is clearly smooth in $(t,m)$. Now we show $\widetilde{F}_{\lambda,t}$ is smooth on $\rho_0(\widetilde{U}_0)$.
Letting $\mu_t=g_t\circ\widetilde{G}_{\lambda,t}\circ \rho_t$, Equation~\eqref{eqn:sameimage} implies
that $(\widetilde{U}_t, \mu_t)$ is an open toric manifold which has
the same image for all $t$ and thus by~\cite[Proposition 6.5]{KarLer}
$(\widetilde{U}_t, \mu_t)$ and $(\widetilde{U}_s, \mu_s)$ are equivariantly symplectomorphic for all $t, s \in [0,1]$. We may assume that we have chosen $\Psi_t^M$ to restrict to this equivariant 
symplectomorphism so
\begin{equation}\label{eqn:mu_t}
 \mu_t\circ\Psi_t^U = g_0\circ \widetilde{G}_{\lambda, 0}\circ \rho_0
\end{equation}
for all $t$.
From the diagram in Figure~\ref{fig:bendingarrows_diagram} and Equation~\eqref{eqn:mu_t} we see that
\[
 (\widetilde{F}_{\lambda,t})_{|\rho_0(\widetilde{U}_0)} = g_t^{-1}\circ \mu_t\circ\Psi_t^U\circ \rho_0^{-1}
                                                      = g_t^{-1}\circ g_0 \circ \widetilde{G}_{\lambda,0}, 
\] 
so we only must show that $(t,u,v) \mapsto g_t^{-1}(u,v)$ is smooth.

In fact, by the implicit function theorem it is sufficient to show that $(t,x,y) \mapsto g_t(x,y)$ is smooth, 
since $\dpar{g_t^{(2)}}{y}(x,y) > 0$.
Recall (as in Remark~\ref{rmk:gcycles}) that $2\pi g_t(x,y)$ is the integral of $\alpha_t = \chi^*_t(\alpha)$ over a cycle
in the fiber $F^{-1}(x,y)\cong \mathbb{T}^k$ (where $k\in\{0,1,2\}$ depends on the type of the fiber: regular, elliptic-transverse, or elliptic-elliptic), where
$\alpha$ is the standard primitive of the symplectic form
in $\C^2$. So we have that
\[
 g_t^{(1)}(x,y) = \frac{1}{2\pi}\int_{\gamma_1(x,y)}\alpha_t = x
\]
where $\gamma_1(x,y)$ is any orbit of $X_J$ contained in $F^{-1}_t(x,y)$, and similarly
\[
 g_t^{(2)}(x,y) = \frac{1}{2\pi}\int_{\gamma_2(t,x,y)}\alpha_t
\]
for some cycle $\gamma_2(t,x,y)\subset F^{-1}_t(x,y)$ such that $\{[\gamma_1],[\gamma_2]\}$ is a basis of $H_1(F^{-1}_t(x,y); \Z)$.
To complete the proof we must argue that $\gamma_2(t,x,y)$ can be chosen so that 
$g_t^{(2)}(x,y)$ is smooth in $t,x,y$. 
Choose a trivialization of
the torus bundle given by the fibers of $F\colon C \to \R^3, (t,p_t)\mapsto (t, F_t(p_t))$ over its image, where $C$ is the fiber bundle over $[0,1]$ with fiber $U_t$ over $t$.
Note that the integral along a given cycle will vary smoothly in $(t,x,y)$ with
the fiber $F^{-1}_t(x,y)$, so we must only argue that the cycle
$\gamma_2(t,x,y)$ can be chosen so that it does not change homotopy type on the torus 
as $t$ varies, but this is possible because
the image $F(C)$ is the fiber over $[0,1]$ with fiber $F_t(U_t)$ which 
is topologically trivial. In fact, we must already have that
$g_t^{(2)}$ is smooth because changing the homotopy type of $\gamma_2(t,x,y)$
would change the image of the Delzant corner in $\Phi(\C^2)\cap\Omega$ by composing
with some power of the matrix $T$, as in Equation~\eqref{eqn:T}, but we have already seen that the corner
in the image does not depend on $t$, see Lemma~\ref{lem:samecorner}.
\end{proof}

We state the following without proof, since it is nearly identical to the proof of Theorem~\ref{thm:blowups}. 

\begin{prop}\label{prop:blowdowns}
Let $(M,\om,F_t)$, $0\leq t\leq 1$ be a \family and let $\lambda>0$. Suppose
 there exist continuous families $(p_t)_{0 \leq t \leq 1}, (q_t)_{0 \leq t \leq 1}$ of points of $M$ such that for every $t \in [0,1]$, $p_t$ and $q_t$ are distinct elliptic-elliptic points of $F_t$ and there exists a surface $\Sigma_t\subset M$ with $p_t, q_t \in \Sigma_t$ such that $(M,\om,F_t)$ admits a blowdown  of size $\lambda$ at $\Sigma_t$. Then there exists a \family
 $(\check{M},\check{\om}_{\lambda,t}, \check{F}_{\lambda,t})$, $0\leq t\leq 1$, such that for each $t$
 $ (\check{M},\check{\om}_{\lambda,t}, \check{F}_{\lambda,t}) $ is symplectomorphic to 
 the blowdown of size $\lambda$ at $\Sigma_t$ of $(M,\om,F_t)$ by a symplectomorphism which intertwines the momentum maps.
\end{prop}

Note that since $M$ admits a blowdown at $\Sigma_t$ in the above proposition, it is implicit
that $\Sigma_t$ is a symplectically embedded $2$-sphere with self intersection $-1$.

\begin{dfn}
 The \emph{blowup} (respectively \emph{blowdown}) of a \family is the \family described
 in Theorem~\ref{thm:blowups} (respectively Proposition~\ref{prop:blowdowns}).
\end{dfn}

From the proof of Theorem~\ref{thm:blowups}, we see that performing a blowup at a given $t$ does not create a degenerate point, and thus if a \family is semitoric at time $t$ then so is its blowup (respectively blowdown); this implies the following.

\begin{cor}
\label{cor:blowup_down_family}
The blowup (respectively blowdown) of a semitoric family is still a semitoric family with the same 
degenerate times.
\end{cor}

We also have a result for degenerate times:
\begin{lm}\label{lem:degenerate-blowupdown}
Suppose that $(M,\om,F_t)$ is a semitoric transition family with transition times $t^-, t^+$ and $q\in M$ is a fixed point of elliptic-elliptic type for all $t\in[0,1]$.
If $(M,\om,F_t)$ admits a toric type blowup of size $\lambda>0$ at $q$ for all $t\in[0,1]\setminus\{t^-,t^+\}$, then for any positive $\tilde{\lambda}<\lambda$ the system $(M,\om,F_t)$ admits a blowup of size $\tilde{\lambda}$ at $q$ for all $t\in [0,1]$.
\end{lm}

\begin{proof}
Let $(M,\om,F_t)$, $\lambda$, and $q$ be as in the statement and let $p$ be the transition point of the semitoric transition family; as above, we may assume that $F_t(q) = (0,0)$ for all $t$. Let $t\in[0,1]\setminus\{t^-,t^+\}$ and let $g_t: (\R^2,0) \to (\R^2,0)$, chosen smooth in $t$ (see proof of Theorem~\ref{thm:blowups}), such that $g_t \circ F_t$ is the momentum map for a $\T^2$-action with image equal to the intersection of an open neighborhood of the origin and a Delzant cone spanned by some $v_1, v_2 \in \Z^2$. Since $(M,\om,F_t)$ admits a toric type blowup of size $\lambda$ at $q$, the point $p$ does not belong to $V_t(\lambda) = (g_t \circ F_t)^{-1} (\text{Simp}_0(\lambda v_1, \lambda v_2))$. By continuity, this implies that $p$ does not belong to the interior of $V_s(\lambda)$ if $s\in\{t^-,t^+\}$, and
 similarly, there are no fixed points other than $q$ in the interior of $V_s(\tilde{\lambda})$. Indeed, by the discussion in Section~\ref{sec:semitoricfam-degen} (following from Lemma~\ref{lm:morse_family}), the only points which are singular at $t^\pm$ are limits of singular points as $t\to t^\pm$ and degenerate singular points, which do not occur
 by assumption in a semitoric transition family except for the transition point $p$. Hence if $0 < \tilde{\lambda} < \lambda$, the only fixed point which belongs to $V_t(\tilde{\lambda})$ for any $t$ is $q$, and the only rank one points are the two families of elliptic-regular points emanating from it. Thus, $(M,\om,F_t)$ admits a blowup of size $\tilde{\lambda}$ at $q$ for all $t\in [0,1]$.
\end{proof}

We also need the following lemma whose proof is similar to the proof of the previous lemma.
The only difference being that, as discussed in Lemma~\ref{lem:fixedpointsdontmove}, fixed points can change with $t$ only in a fixed sphere of $J$, which occur at its maximum and minimum values.

\begin{lm}\label{lem:degenerate-blowupdown-wall}
Suppose that $(M,\om,F_t)$ is a semitoric transition family with transition times $t^-, t^+$ and $j\in\R$ is a maximum or minimum value of $J$.
Suppose $p_t\in J^{-1}(j)$ is the point realizing the maximum or minimum of $(H_t)_{|J^{-1}(j)}$ for all $t\in [0,1]\setminus\{t^-,t^+\}$.
If $(M,\om,F_t)$ admits a toric type blowup of size $\lambda>0$ at $p_t$ for all $t\in[0,1]\setminus\{t^-,t^+\}$, then for any positive $\tilde{\lambda}<\lambda$ the system $(M,\om,F_t)$ admits a blowup of size $\tilde{\lambda}$ at $p_t$ for all $t\in [0,1]$.
\end{lm}

\section{Semitoric transition families on Hirzebruch surfaces}
\label{sec:hirz}

In this section, we recall the definition of Hirzebruch surfaces and some of their properties, and prove the existence of certain semitoric 1-transition families with prescribed semitoric polygon invariants on each Hirzebruch surface, which is the content of Theorem \ref{thm:Wkdetailed} (a more precise version of Theorem \ref{thm:Wk}).

\subsection{Definition of Hirzebruch surfaces through symplectic reduction}
\label{sec:hirzdef}

There are several equivalent definitions of Hirzebruch surfaces; we use the Delzant construction (see \cite{Del} or \cite[Section 2.5]{Can_toric} for more details) to define them starting from their standard Delzant polygons. More precisely, for $\alpha, \beta > 0$ and $n \in \Z_{\geq 0}$, we define the $n$-th Hirzebruch surface $\Hirzscaled{n}$ with scalings $\alpha, \beta$ as the toric manifold whose Delzant polygon is displayed in Figure \ref{fig:poly_Wn}, and we denote by $\omHirz{n}$ the symplectic form on $\Hirzscaled{n}$ obtained in this way.

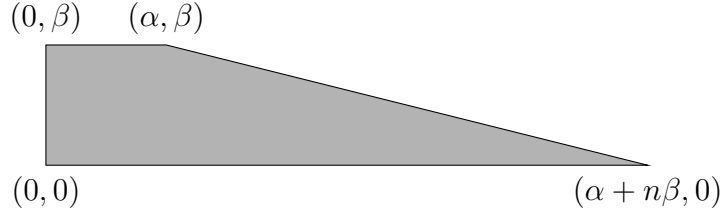
\begin{figure}
\begin{center}

\begin{tikzpicture}[scale=.8]
\filldraw[draw=black,fill=gray!60] (0,0) node[anchor=north,color=black]{$(0,0)$}
  -- (0,2) node[anchor=south,color=black]{$(0,\beta)$}
  -- (2,2) node[anchor=south,color=black]{$(\alpha,\beta)$}
  -- (10,0) node[anchor=north,color=black]{$(\alpha + n \beta,0)$}
  -- cycle;

\end{tikzpicture}
\caption{The standard Delzant polygon for $\Hirzscaled{n}$.}
\label{fig:poly_Wn}
\end{center}
\end{figure}

This construction yields that $\Hirzscaled{n}$ identifies with the symplectic reduction at level zero of $\C^4$ with symplectic form $\omega_{std} = \frac{i}{2} \sum_{j=1}^4 du_j \wedge d\bar{u}_j$ with respect to the Hamiltonian $\T^2$-action generated by
\[ N(u_1,u_2,u_3,u_4) = \frac{1}{2} \left( |u_1|^2 + |u_2|^2 + n |u_3|^2, |u_3|^2 + |u_4|^2 \right) - \left( \alpha + \beta n, \beta \right) \]
and that the standard toric system on $\Hirzscaled{n}$ is $(J_{std},H_{std})$ with $J_{std} = \frac{1}{2} |u_2|^2$, $H_{std} = \frac{1}{2} |u_3|^2$. We will denote by $[u_1, u_2, u_3, u_4] \in \Hirzscaled{n}$ the equivalence class of $(u_1, u_2, u_3, u_4) \in \C^4$ in this symplectic quotient.

\subsection{Local coordinates}
\label{subsect:coords_W}

We now construct an atlas for $\Hirzscaled{n}$. Define, for $\ell, m \in \llbracket 1,4 \rrbracket$, $\ell \neq m$, the open set
\[ U_{\ell,m} = \left\{ [u_1,u_2,u_3,u_4] \ | \ u_{\ell} \neq 0, u_m \neq 0 \right\} \subset \Hirzscaled{n}. \] 
Then we get an open cover $\Hirzscaled{n} \subset  U_{1,3} \cup U_{1,4} \cup U_{2,3} \cup U_{2,4}$.
Indeed, the complementary set of the right hand side is 
\[ S = \left\{u_1 = 0 \ \text{or} \ u_3 = 0 \right\} \cap \left\{ u_1 = 0 \ \text{or} \ u_4 = 0 \right\}  \cap \left\{ u_2 = 0 \ \text{or} \ u_3 = 0 \right\}  \cap \left\{ u_2 = 0 \ \text{or} \ u_4 = 0 \right\}. \]
But $u_3$ and $u_4$ cannot vanish at the same time (since $|u_3|^2 + |u_4|^2 = 2 \beta$ on $N^{-1}(0)$), so if $[u]$ belongs to $S$, then necessarily $u_1 = 0 = u_2$. This implies that $|u_3|^2 = 2\left(\beta + \frac{\alpha}{n} \right)$, which is impossible since $|u_3|^2 \leq 2 \beta$. 

Now, we can get local coordinates on $U_{\ell,m}$, for $(\ell,m) = (1,3), (1,4), (2,3)$ or $(2,4)$, as follows. Using the action of $N$, we may find a representative of $[u]$ such that both $u_{\ell}$ and $u_m$ are real and positive; we write $u_{\ell} = x_{\ell} > 0$ and $u_m = x_m > 0$. Then, we write $\{1,2,3,4 \} \setminus \{ \ell,m \} = \{p, q\}$ and $u_p = x_p + i y_p$, $u_q = x_q + i y_q$ with $x_p, y_p, x_q, y_q \in \R$. Then the latter give local coordinates on $U_{\ell,m}$, and we use the equations given by $N=0$ to write $x_{\ell}$ and $x_m$ in these coordinates. More precisely,

\begin{itemize}
\item for $(\ell,m) = (1,3)$: $x_1 = \sqrt{2\alpha + n(x_4^2 + y_4^2) - (x_2^2 + y_2^2)}$, $x_3 = \sqrt{2\beta - (x_4^2 + y_4^2)}$,
\item for $(\ell,m) = (1,4)$: $x_1 = \sqrt{2(\alpha + n \beta) - n(x_3^2 + y_3^2) - (x_2^2 + y_2^2)}$, $x_4 = \sqrt{2\beta - (x_3^2 + y_3^2)}$,
\item for $(\ell,m) = (2,3)$: $x_2 = \sqrt{2\alpha + n(x_4^2 + y_4^2) - (x_1^2 + y_1^2)}$, $x_3 = \sqrt{2\beta - (x_4^2 + y_4^2)}$,
\item for $(\ell,m) = (2,4)$: $x_2 = \sqrt{2(\alpha + n \beta) - n(x_3^2 + y_3^2) - (x_1^2 + y_1^2)}$, $x_4 = \sqrt{2\beta - (x_3^2 + y_3^2)}$.
\end{itemize}
In any of these coordinates, the symplectic form is $\omHirz{n} = dx_p \wedge dy_p + dx_q \wedge dy_q$.

\subsection{A semitoric system on \texorpdfstring{$\Hirzscaled{0}$}{W0}: the coupled angular momenta}
\label{sec:coupledspins}

We already know an example of semitoric system on the Hirzebruch surface $\Hirzscaled{0}$. Indeed, recall the coupled angular momenta system on $M = \S^2\times \S^2$ with symplectic form $\omega_{R_1,R_2} = R_1\om_{S^2}\oplus R_2\om_{\S^2}$ described in Section \ref{sec:knownex}. In the language of the present paper, Corollary 2.11 in \cite{LFP} states that this system is a semitoric 1-transition family with transition point $(0,0,1,0,0,-1)$ and transition times 
\[ t^- = \frac{R_2}{2 R_2 + R_1 + 2 \sqrt{R_1 R_2}}, \qquad t^+ = \frac{R_2}{2 R_2 + R_1 - 2 \sqrt{R_1 R_2}}. \]

Since $\Hirzscaled{0}$ is symplectomorphic to $\S^2\times \S^2$ with symplectic form $\omega_{R_1,R_2} $ where $R_1 = \frac{\alpha}{2}$ and $R_2 = \frac{\beta}{2}$, the following result holds.

\begin{lm}\label{lem:coupledspinsW0}
There exists a semitoric 1-transition family on $\Hirzscaled{0}$ with transition point given by $\left[ \sqrt{2 \alpha}, 0, \sqrt{2\beta}, 0 \right]$ and transition times
\[ t^- = \frac{\beta}{2 \beta + \alpha + 2 \sqrt{\alpha \beta}}, \qquad t^+ = \frac{\beta}{2 \beta + \alpha - 2 \sqrt{\alpha \beta}}. \]
\end{lm}

\subsection{Constructing semitoric transition families on every Hirzebruch surface}
\label{sec:thmWkproof}

It is well-known that the standard toric system on $\Hirzscaled{n+1}$ can be obtained from the standard toric system on $W_n(\alpha',\beta)$, for some other $\alpha'>0$, by a  blowup followed by a blowdown, see for instance the proof of Theorem 5 in~\cite{PPRStoric}. The same argument can be applied to prove the existence of the semitoric system on $\Hirzscaled{n}$ described in Theorem \ref{thm:Wk}, using the properties of the toric type blowups and blowdowns established in the previous sections. For the basic
idea of the construction, see Figures~\ref{fig:obtainingWk} and~\ref{fig:Wkfamilies}.

\begin{figure}
\begin{center}
\begin{subfigure}[b]{.8\linewidth}
\begin{center}
\begin{tikzpicture}[scale=.60]

\filldraw[draw=black,fill=gray!60] (0,-2) node[anchor=south,color=black]{\footnotesize{$(0,\beta)$}}
  -- (4,-2) node[anchor=south,color=black]{\footnotesize{$(\alpha+\beta,\beta)$}}
  -- (8,-4) node[anchor=north,color=black]{\footnotesize{$(\alpha + n \beta,0)$}}
  -- (2,-4) node[anchor=north,color=black]{\footnotesize{$(\beta,0)$}}
  -- cycle;
\draw [dashed] (2,-4) -- (2,-3);
\draw (2,-3) node[] {$\times$};
\end{tikzpicture}
\end{center}
\caption{A representative of $[(\polygonnzero, (\beta, y(t)),(+1))]$.}
\label{fig:polygon_Wnhalf}
\end{subfigure}

\vspace{15pt}

 \begin{subfigure}[b]{.45\linewidth}
  \begin{center}
   \begin{tikzpicture}[scale=.6]
    \filldraw[draw=black,fill=gray!60] (0,0) node[anchor=north,color=black]{\footnotesize{$(0,0)$}}
      -- (2,2) node[anchor=south east,color=black]{\footnotesize{$(\beta,\beta)$}}
      -- (4,2) node[anchor=south west,color=black]{\footnotesize{$(\alpha+\beta,\beta)$}}
      -- (8,0) node[anchor=north,color=black]{\footnotesize{$(\alpha + n \beta,0)$}}
      -- cycle;
   \end{tikzpicture}
  \end{center}
 \caption{The polygon $\polygonnzero$.}
 \label{fig:polygonnzero}
 \end{subfigure}\,\,\,
 \begin{subfigure}[b]{.45\linewidth}
  \begin{center}
   \begin{tikzpicture}[scale=.6]
    \filldraw[draw=black,fill=gray!60] (0,2) node[anchor=south,color=black]{\footnotesize{$(0,\beta)$}}
      -- (4,2) node[anchor=south,color=black]{\footnotesize{$(\alpha+\beta,\beta)$}}
      -- (8,0) node[anchor=north,color=black]{\footnotesize{$(\alpha + n \beta,0)$}}
      -- (2,0) node[anchor=north,color=black]{\footnotesize{$(\beta,0)$}}
      -- cycle;
   \end{tikzpicture}
  \end{center}
 \caption{The polygon $\polygonnone$.}
 \label{fig:polygonnone}
 \end{subfigure}
\end{center}
\caption{The polygons associated with the semitoric family on $\Hirzscaled{n}$ for (a) $t^-<t<t^+$, (b)  	$t<t^-$, and (c) $t>t^+$.}
\label{fig:toricpolygons}
\end{figure}
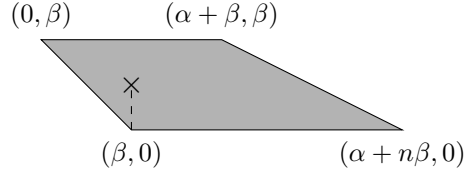
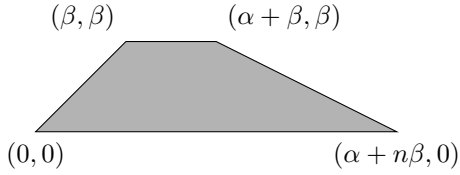
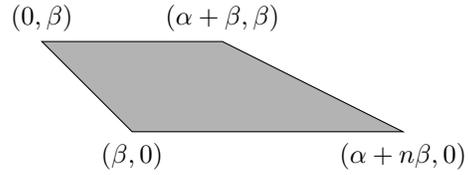

We can now prove the main result of this section, which is a more detailed version of Theorem~\ref{thm:Wk}.
Let $\polygonnzero$ and $\polygonnone$ be the convex hulls of $\{(0,0), (\beta, \beta), (\alpha+\beta, \beta), (\alpha + n\beta, 0)\}$ and $\{(0,\beta), (\beta, 0), (\alpha+\beta, \beta), (\alpha + n\beta, 0)\}$ respectively, as shown in Figures~\ref{fig:polygonnzero} and~\ref{fig:polygonnone}.

\begin{thm}\label{thm:Wkdetailed}
For every $n \in \Z_{\geq 0}$ and $\alpha, \beta > 0$
there exists a semitoric 1-transition family on $W_n(\alpha,\beta)$ with transition point $\left[ \sqrt{2 \alpha}, 0, \sqrt{2\beta}, 0 \right]$ and transition times $t^-_n, t^+_n\in (0,1)$ satisfying 
\[\frac{\beta}{2 \beta + \alpha + 2 \sqrt{\alpha \beta}} \leq t^-_n < \frac{1}{2} < t^+_n \leq \frac{\beta}{2 \beta + \alpha - 2 \sqrt{\alpha \beta}}
\]
with equality on both sides if and only if $n=0$,
such that
\begin{enumerate}
 \item for $t^-_n<t< t^+_n$ the system is semitoric and has associated marked semitoric polygon $[(\polygonnzero, (\beta, y(t)),(+1))]$
  where $y(t)$ belongs to the interval $(0,\beta)$, shown in Figure~\ref{fig:polygon_Wnhalf};
 \item if $t = t^-_n$ or $t = t^+_n$ the system has exactly one degenerate point at $\left[ \sqrt{2 \alpha}, 0, \sqrt{2\beta}, 0 \right]$;
 \item if $t<t^-_n$ the system is semitoric with zero focus-focus points and associated marked semitoric polygon $[(\polygonnzero, \varnothing,\varnothing)]$, shown in Figure~\ref{fig:polygonnzero};
 \item if $t>t^+_n$ the system is semitoric with zero focus-focus points and associated marked semitoric polygon $[(\polygonnone, \varnothing,\varnothing)]$, shown in Figure~\ref{fig:polygonnone};
\end{enumerate} 
Moreover, such a system can be obtained from the scaled coupled angular momenta system on
$W_0(\alpha',\beta)$ for some choice of $\alpha'\geq \alpha$ by alternately performing blowups and blowdowns, each $n$ times.
\end{thm}

Note that the unmarked semitoric polygons obey the relationship for semitoric transition families described in Lemma~\ref{lem:polygons_transition} because the weights of the $\mathbb{S}^1$-action at the transition point are $\pm 1$,
hence items 3.~and 4.~are automatic from item 1.

\begin{proof}
We proceed by induction on $n$; if $n=0$ then the result follows from Lemma~\ref{lem:coupledspinsW0}. Let $n \geq 0$ and assume that 
such a semitoric family exists on $W_n(\alpha',\beta')$ for all $\alpha',\beta'>0$.
Given $\alpha,\beta>0$, let $0<\lambda < \beta$
and consider the system on $W_n(\alpha+\lambda, \beta)$ with transition times $t^-_n$, $t^+_n$, and which
has marked semitoric polygon (depending on $t$) as described in the theorem.
For each $t\in [0,1]$ let $p_t \in W_n(\alpha+\lambda, \beta)$ be the point whose image 
in any representative with $\epsilon = +1$ (upwards cut) of the marked semitoric polygon is the upper right corner (marked with a black dot in the first column of Figures~\ref{fig:obtainingWk} and~\ref{fig:Wkfamilies}); in fact $p_t$ can also be determined by fixing $p_0=[0,\sqrt{2(\alpha+\lambda)},\sqrt{2\beta},0]$ and requiring that $p_t$ is an elliptic-elliptic point of $F_t$ and $t\mapsto p_t$ is continuous.
Note that if $n\neq 1$ then $p_t = p_0$ for all $t$, but in the case $n=1$ the point may move since it belongs to a fixed surface of $J$.
By Theorem~\ref{thm:blowups}, to conclude that a blowup of size $\lambda$ of this family is possible we have to argue that a blowup of size $\lambda$ is possible for all $t\in [0,1]$.
If $t\notin\{t_n^-,t_n^+\}$ note that the edges adjacent to the image of $p_t$ in the semitoric polygon
have $\mathrm{SL}_2(\Z)$-length (as in Remark~\ref{rmk:SL2Z-length}) given by $\alpha+\lambda$, $\beta$, or $\alpha+\beta+\lambda$ (depending on the value of $t$
relative to $t^\pm$, and the choice of representative), all of which are greater than $\lambda$,
and the only marked point is not in the region to be removed with the 
desired corner chop, so as in Remark~\ref{rmk:SL2Z-length}
we conclude that a blowup at $p_t$ is possible.
Also, this implies that if $t\in\{t_n^-,t_n^+\}$ a blowup of size $\lambda$ is possible by Lemmas~\ref{lem:degenerate-blowupdown}
and~\ref{lem:degenerate-blowupdown-wall}
(taking a smaller $\lambda$ if necessary).
Since the semitoric family on $W_n(\alpha+\lambda, \beta)$ admits a blowup of size $\lambda$ at $p_t$
for all $t\in [0,1]$, by Theorem~\ref{thm:blowups} we can perform a blowup of size $\lambda$ at $p_t$ for the family and by 
Corollary~\ref{cor:blowup_down_family} the resulting \family is a semitoric family on $\blowup{{p_t}}{}(W_n(\alpha+\lambda,\beta))$
which has degenerate times $t_n^-,t_n^+$. Moreover, by Lemma~\ref{lem:blowups_semitoric_chops} we know the semitoric polygons of the new family,
they are the ones which are obtained by performing a semitoric corner chop of size $\lambda$ at the upper right corner of the polygons
for $W_n(\alpha+\lambda,\beta)$, as shown in the middle column of Figure~\ref{fig:obtainingWk} for $n=0,1,2$.

Next, we show the family we just constructed on $M=\blowup{{p_t}}{}(W_n(\alpha+\lambda,\beta))$ admits a blowdown.
Let $\Sigma_t$ be the sphere whose image is marked in bold in the middle column
of Figures~\ref{fig:obtainingWk} and~\ref{fig:Wkfamilies}.
Notice that if
$t\notin\{t_n^-, t_n^+\}$ then the semitoric polygon associated to the new family admits a corner unchop of size $\beta-\lambda$, since
it can be seen as the corner chop of $\polygonnzero$ or $\polygonnone$ (depending on the value of $t$ and the choice of semitoric polygon),
at the point $p_t'$ which is the preimage of the lower right corner in the semitoric polygon of $\Hirzscaled{n+1}$
(marked in blue in the right column of Figures~\ref{fig:obtainingWk} and~\ref{fig:Wkfamilies}, and as before note that $p_t'$ is independent of $t$ unless $n=0$).
Again using Lemmas~\ref{lem:degenerate-blowupdown} and~\ref{lem:degenerate-blowupdown-wall}, we see this also implies such a blowdown is possible
if $t\in\{t_n^-,t_n^+\}$ (again, taking a smaller $\lambda$ if needed).
Thus, by Proposition~\ref{prop:blowdowns} and Corollary~\ref{cor:blowup_down_family}, we have that the blowdown of the family
on $\blowup{{p_t}}{}(W_n(\alpha+\lambda,\beta))$ at $\Sigma_t$ produces a semitoric family with the same degenerate times $t_n^-$
and $t_n^+$ on a manifold $M$ which is the blowdown at $\Sigma_t$ of $\blowup{{p_t}}{}(W_n(\alpha+\lambda,\beta))$.
Furthermore, if $t=0$ in the family on $M$ we obtain a system which becomes toric when scaling the second component appropriately (like the coupled angular momenta) with Delzant polygon the one associated
to $W_{n+1}(\alpha,\beta)$, so we conclude that $M\cong W_{n+1}(\alpha,\beta)$ and thus we have produced
the required system.
Note that $t_0^-$ and $t_0^+$ satisfy the required equality by Lemma~\ref{lem:coupledspinsW0}, and by performing a blowup of size
$\lambda$ on $W_0(\alpha, \beta+\lambda)$ we have
\[ t_1^- = \frac{\beta + \lambda}{2 \beta + 2\lambda+ \alpha + 2 \sqrt{\alpha (\beta+\lambda)}}, \qquad t_1^+ = \frac{\beta + \lambda}{2 \beta + 2\lambda+ \alpha - 2 \sqrt{\alpha (\beta+\lambda)}} \]
which satisfy the claimed inequalities.
Finally, if $n>0$, note that $t^-_{n+1}$ and $t^+_{n+1}$ are the transition times for the coupled angular momenta system on $W_0(\alpha',\beta)$ where $\alpha'$ is determined by the sizes of each of the blowups used in the construction of $W_{n+1}(\alpha,\beta)$, but in any case $\alpha'>\alpha$ and thus
$t^+_{n+1}$ and $t^-_{n+1}$ satisfy the claimed inequalities.
\end{proof}

\begin{figure}[h]
\begin{center}
\begin{tikzpicture}[scale=.7]

\node[label={$W_n$}] at (3,2){};
\node[label={$\widetilde{W_n}$}] at (8.5,2){};
\node[label={$W_{n+1}$}] at (14,2){};

\node[label={$n=0$}] at (-1.25,0.5){};
\node[label={$n=1$}] at (-1.25,-2.5){};
\node[label={$n=2$}] at (-1.25,-5.5){};

\draw [->] (4.5,1) -- (5.5,1);
\draw [->] (10,1) -- (11,1);
\draw [-, rounded corners] (13.4, -0.1) |- (12.0, -0.5);
\draw [->, rounded corners] (12.0, -0.5)  -| (2.6,-0.9);
\draw [->] (4.5,-2) -- (5.5,-2);
\draw [->] (10,-2) -- (11,-2);
\draw [-, rounded corners] (13.4, -3.1) |- (12.0, -3.5);
\draw [->, rounded corners] (12.0, -3.5)  -| (2.6,-3.9);
\draw [->] (4.5,-5) -- (5.5,-5);
\draw [->] (10,-5) -- (11,-5);

\filldraw[draw=black,fill=gray!60] (0,0) node[anchor=north,color=black]{}
  -- (2,2) node[anchor=south,color=black]{}
  -- (5,2) node[anchor=south,color=black]{}
  -- (3,0) node[anchor=south,color=black]{}
  -- cycle;  
\draw [dashed] (2,2) -- (2,1);
\draw (2,1) node[] {$\times$};
\fill[black] (5,2) circle (.1);	

\filldraw[draw=black,fill=gray!60] (5.5,0) node[anchor=north,color=black]{}
  -- (7.5,2) node[anchor=south,color=black]{}
  -- (9.5,2) node[anchor=south,color=black]{}
  -- (9.5,1) node[anchor=south,color=black]{}
  -- (8.5,0) node[anchor=south,color=black]{}
  -- cycle;  
\draw [dashed] (7.5,2) -- (7.5,1);
\draw (7.5,1) node[] {$\times$};
\draw[ultra thick] (9.5,1) to (8.5,0);
\fill[black] (9.5,1) circle (1/25);	
\fill[black] (8.5,0) circle (1/25);

\filldraw[draw=black,fill=gray!60] (11,0) node[anchor=north,color=black]{}
  -- (13,2) node[anchor=south,color=black]{}
  -- (15,2) node[anchor=south,color=black]{}
  -- (15,0) node[anchor=south,color=black]{}
  -- cycle;  
\draw [dashed] (13,2) -- (13,1);
\draw (13,1) node[] {$\times$};
\fill[blue] (15,0) circle (.1);	

\filldraw[draw=black,fill=gray!60] (0,-3) node[anchor=north,color=black]{}
  -- (2,-1) node[anchor=south,color=black]{}
  -- (4,-1) node[anchor=south,color=black]{}
  -- (4,-3) node[anchor=south,color=black]{}
  -- cycle;  
\draw [dashed] (2,-1) -- (2,-2);
\draw (2,-2) node[] {$\times$};
\fill[black] (4,-1) circle (.1);	
  
\filldraw[draw=black,fill=gray!60] (5.5,-3) node[anchor=north,color=black]{}
  -- (7.5,-1) node[anchor=south,color=black]{}
  -- (8.5,-1) node[anchor=south,color=black]{}
  -- (9.5,-2) node[anchor=south,color=black]{}
  -- (9.5,-3) node[anchor=south,color=black]{}
  -- cycle;  
\draw [dashed] (7.5,-1) -- (7.5,-2);
\draw (7.5,-2) node[] {$\times$};
\draw[ultra thick] (9.5,-2) to (9.5,-3);
\fill[black] (9.5,-2) circle (1/25);	
\fill[black] (9.5,-3) circle (1/25);

\filldraw[draw=black,fill=gray!60] (11,-3) node[anchor=north,color=black]{}
  -- (13,-1) node[anchor=south,color=black]{}
  -- (14,-1) node[anchor=south,color=black]{}
  -- (16,-3) node[anchor=south,color=black]{}
  -- cycle;
\draw [dashed] (13,-1) -- (13,-2);
\draw (13,-2) node[] {$\times$};
\fill[blue] (16,-3) circle (.1);	
 
\filldraw[draw=black,fill=gray!60] (0,-6) node[anchor=north,color=black]{}
  -- (2,-4) node[anchor=south,color=black]{}
  -- (3,-4) node[anchor=south,color=black]{}
  -- (5,-6) node[anchor=south,color=black]{}
  -- cycle;  
\draw [dashed] (2,-4) -- (2,-5);
\draw (2,-5) node[] {$\times$};
\fill[black] (3,-4) circle (.1);	  
  
\filldraw[draw=black,fill=gray!60] (5.5,-6) node[anchor=north,color=black]{}
  -- (7.5,-4) node[anchor=south,color=black]{}
  -- (8,-4) node[anchor=south,color=black]{}
  -- (9.5,-5) node[anchor=south,color=black]{}
  -- (10.5,-6) node[anchor=south,color=black]{}
  -- cycle;  
\draw [dashed] (7.5,-4) -- (7.5,-5);
\draw (7.5,-5) node[] {$\times$};
\draw[ultra thick] (9.5,-5) to (10.5,-6);
\fill[black] (9.5,-5) circle (1/25);	
\fill[black] (10.5,-6) circle (1/25);

\filldraw[draw=black,fill=gray!60] (11,-6) node[anchor=north,color=black]{}
  -- (13,-4) node[anchor=south,color=black]{}
  -- (13.5,-4) node[anchor=south,color=black]{}
  -- (17.5,-6) node[anchor=south,color=black]{}
  -- cycle;
\draw [dashed] (13,-4) -- (13,-5);
\draw (13,-5) node[] {$\times$};
\fill[blue] (17.5,-6) circle (.1);	
  
\end{tikzpicture}
\end{center}
\caption{One representative of the marked semitoric polygon associated with each step in the construction of a semitoric $1$-transition family 
on $\Hirzscaled{n}$ from the scaled coupled spins system on $W_0(\alpha',\beta)$.
We perform a blowup at the black point and then a blowdown at the bold edge.
Compare with Figure~\ref{fig:Wkfamilies}.
The scalings are not indicated for clarity; performing the sequence of blowup and blowdown on $W_n(\alpha,\beta)$ yields $W_{n+1}(\alpha-\lambda,\beta)$
where $\lambda$ is the size of the blowup.}
\label{fig:obtainingWk}
\end{figure}
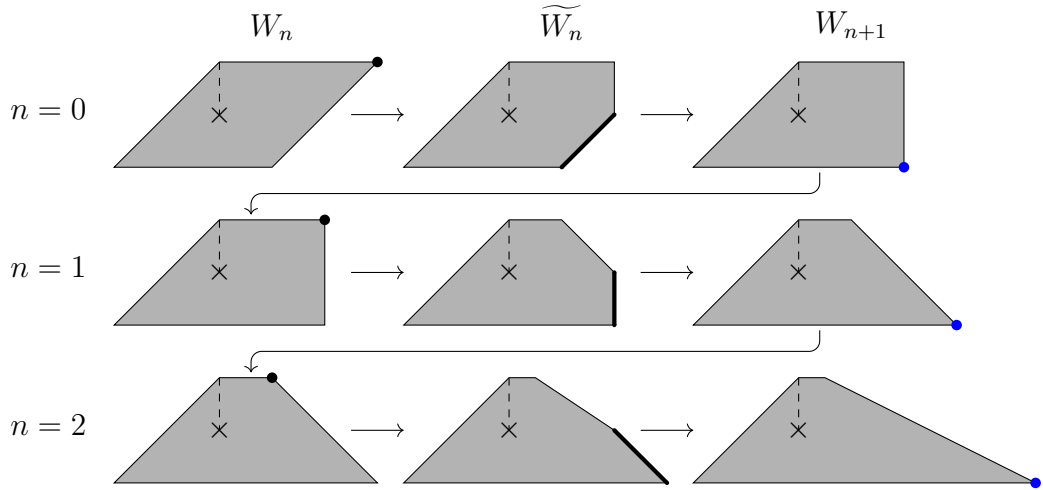

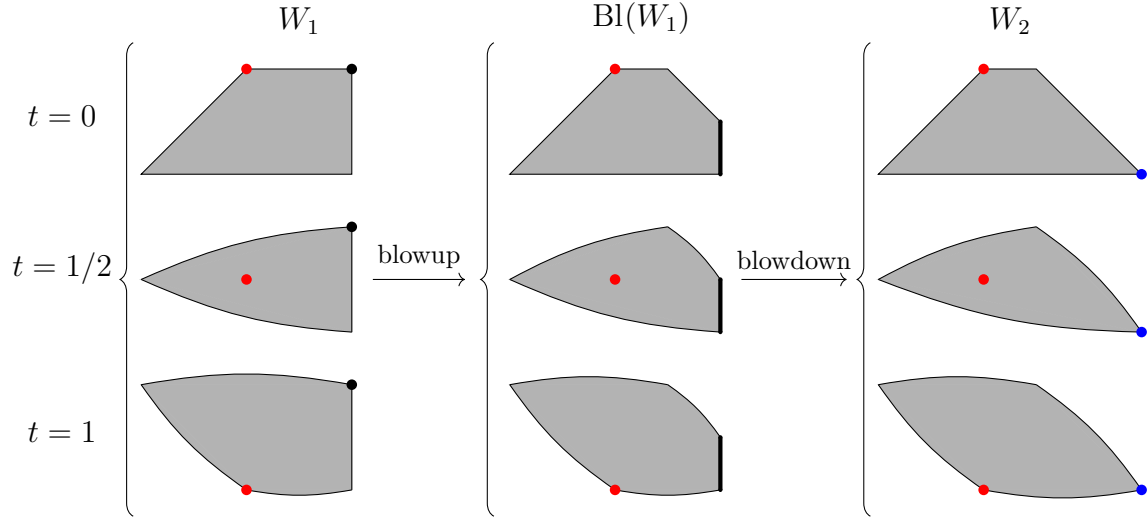
\begin{figure}
\begin{center}
\begin{tikzpicture}[scale=.7]

\node[label={$W_1$}] at (3,2.25){};
\node[label={$\blowup{}{}(W_1)$}] at (9.5,2.25){};
\node[label={$W_{2}$}] at (16.5,2.25){};

\node[label={$t=0$}] at (-1.5,0.5){};
\node[label={$t=1/2$}] at (-1.5,-2.5){};
\node[label={$t=1$}] at (-1.5,-5.5){};

\draw [->] (4.4,-2) -- node[above]{\footnotesize{blowup}} ++ (1.75,0); 
\draw [->] (11.4,-2) -- node[above]{\footnotesize{blowdown}} ++ (2.00,0);

\draw[decoration={brace,raise=5pt, amplitude = 5pt},decorate] (0.1,-6.5) -- (0.1,2.5);
\draw[decoration={brace,raise=5pt, amplitude = 5pt},decorate] ( 6.95,-6.5) -- (6.95,2.5);
\draw[decoration={brace,raise=5pt, amplitude = 5pt},decorate] ( 14.1,-6.5) -- (14.1,2.5);

\filldraw[draw=black,fill=gray!60] (0,0) node[anchor=north,color=black]{}
  -- (2,2) node[anchor=south,color=black]{}
  -- (4,2) node[anchor=south,color=black]{}
  -- (4,0) node[anchor=south,color=black]{}
  -- cycle;  
\fill[red] (2,2) circle (.1);	
\fill[black] (4,2) circle (.1);	

\filldraw[draw=black,fill=gray!60] (7,0) node[anchor=north,color=black]{}
  -- (9,2) node[anchor=south,color=black]{}
  -- (10,2) node[anchor=south,color=black]{}
  -- (11,1) node[anchor=south,color=black]{}
  -- (11,0) node[anchor=south,color=black]{}
  -- cycle;  
\fill[red] (9,2) circle (.1);	
\draw[ultra thick] (11,1) to (11,0);
\fill[black] (11,1) circle (1/25);	
\fill[black] (11,0) circle (1/25);

\filldraw[draw=black,fill=gray!60] (14,0) node[anchor=north,color=black]{}
  -- (16,2) node[anchor=south,color=black]{}
  -- (17,2) node[anchor=south,color=black]{}
  -- (19,0) node[anchor=south,color=black]{}
  -- cycle;  
\fill[red] (16,2) circle (.1);
\fill[blue] (19,0) circle (.1);	

\filldraw[draw=gray!60,fill=gray!60] (0,-2) node[anchor=north,color=black]{}
  -- (4,-1) node[anchor=south,color=black]{}
  -- (4,-3) node[anchor=south,color=black]{}
  -- cycle;  
\draw[bend left=10, fill = gray!60] (0,-2) to (4,-1);
\draw[bend left=0, fill = gray!60] (4,-1) to (4,-3);
\draw[bend left=10, fill = gray!60] (4,-3) to (0,-2);
\fill[red] (2,-2) circle (.1);	
\fill[black] (4,-1) circle (.1);

\filldraw[draw=gray!60,fill=gray!60] (7,-2) node[anchor=north,color=black]{}
  -- (10,-1) node[anchor=south,color=black]{}
  -- (11,-2) node[anchor=south,color=black]{}
  -- (11,-3) node[anchor=south,color=black]{}
  -- cycle;  
\draw[bend left=10, fill = gray!60] (7,-2) to (10,-1);
\draw[bend left=10, fill = gray!60] (10,-1) to (11,-2);
\draw[ultra thick, bend left=0, fill = gray!60] (11,-2) to (11,-3);
\draw[bend left=10, fill = gray!60] (11,-3) to (7,-2);
\fill[red] (9,-2) circle (.1);	
\fill[black] (11,-2) circle (1/25);	
\fill[black] (11,-3) circle (1/25);	

\filldraw[draw=gray!60,fill=gray!60] (14,-2) node[anchor=north,color=black]{}
  -- (17,-1) node[anchor=south,color=black]{}
  -- (19,-3) node[anchor=south,color=black]{}
  -- cycle;  
\draw[bend left=10, fill = gray!60] (14,-2)to (17,-1);
\draw[bend left=10, fill = gray!60] (17,-1) to (19,-3);
\draw[bend left=10, fill = gray!60] (19,-3) to (14,-2);
\fill[red] (16,-2) circle (.1);
\fill[blue] (19,-3) circle (.1);	

\filldraw[draw=gray!60,fill=gray!60] (0,-4) node[anchor=north,color=black]{}
  -- (4,-4) node[anchor=south,color=black]{}
  -- (4,-6) node[anchor=south,color=black]{}
  -- (2,-6) node[anchor=south,color=black]{}
  -- cycle;  
\draw[bend left=10, fill = gray!60] (0,-4) to (4,-4);
\draw[bend left=0, fill = gray!60] (4,-4) to (4,-6);
\draw[bend left=10, fill = gray!60] (4,-6) to (2,-6);
\draw[bend left=10, fill = gray!60] (2,-6) to (0,-4);
\fill[red] (2,-6) circle (.1);	
\fill[black] (4,-4) circle (.1);	

\filldraw[draw=gray!60,fill=gray!60] (7,-4) node[anchor=north,color=black]{}
  -- (10,-4) node[anchor=south,color=black]{}
  -- (11,-5) node[anchor=south,color=black]{}
  -- (11,-6) node[anchor=south,color=black]{}
  -- (9,-6) node[anchor=south,color=black]{}
  -- cycle;  
\draw[bend left=10, fill = gray!60] (7,-4) to (10,-4);
\draw[bend left=10, fill = gray!60] (10,-4) to (11,-5);
\draw[ultra thick, bend left=0, fill = gray!60] (11,-5) to (11,-6);
\draw[bend left=10, fill = gray!60] (11,-6) to (9,-6);
\draw[bend left=10, fill = gray!60] (9,-6) to (7,-4);
\fill[red] (9,-6) circle (.1);	
\fill[black] (11,-5) circle (1/25);	
\fill[black] (11,-6) circle (1/25);	

\filldraw[draw=gray!60,fill=gray!60] (14,-4) node[anchor=north,color=black]{}
  -- (17,-4) node[anchor=south,color=black]{}
  -- (19,-6) node[anchor=south,color=black]{}
  -- (16,-6) node[anchor=south,color=black]{}
  -- cycle;  
\draw[bend left=10, fill = gray!60] (14,-4) to (17,-4);
\draw[bend left=10, fill = gray!60] (17,-4) to (19,-6);
\draw[bend left=10, fill = gray!60] (19,-6) to (16,-6);
\draw[bend left=10, fill = gray!60] (16,-6) to (14,-4);
\fill[red] (16,-6) circle (.1);
\fill[blue] (19,-6) circle (.1);	
\end{tikzpicture}
\end{center}
\caption{Performing a blowup followed by a blowdown on the semitoric family on $\Hirzscaled{1}$ to produce
$\Hirzscaled{2}$, showing $t=0,1/2,1$.
We perform a blowup at the black point and then a blowdown at the bold edge.
This shows three elements of the family constructed in the proof of Theorem~\ref{thm:Wkdetailed}; the associated semitoric polygons are shown in the second row of Figure~\ref{fig:obtainingWk}.}
\label{fig:Wkfamilies}
\end{figure}

\begin{rmk}\label{rmk:taylortwist}
 Note that the systems discussed in Theorem~\ref{thm:Wkdetailed} are all obtained from the coupled angular momenta system by performing
 operations which do not affect a neighborhood of the $J$-fiber containing the focus-focus point (when $t^-<t<t^+$). Thus, for corresponding values of
 $t$ the Taylor series invariant for each of these systems is the same.
 Similarly, the integer labels of polygons obtained by the twisting index invariant are equal for corresponding polygons (those which are equal as sets in the region $\{(j_0-\varepsilon, j_0+\varepsilon\}$ where $j_0$ is the $J$-value of the focus-focus point).
\end{rmk}

\section{Explicit semitoric families on \texorpdfstring{$\Hirzscaled{1}$}{W1}}
\label{sec:W1_example}

By Theorem~\ref{thm:Wkdetailed} we know that a semitoric 1-transition family exists on $\Hirzscaled{1}$, and in this section we find completely explicit ones (but we will see that they are not the same ones abstractly constructed in Theorem~\ref{thm:Wkdetailed}); one for which the fixed points move in the preimage of the vertical wall with the parameter $t$, and one for which this is not the case and the images of two fixed points collide for $t=1/2$.

\subsection{Preliminaries on \texorpdfstring{$\Hirzscaled{1}$}{W1}}
\label{subsect:notation_W1}

As before, we see $\Hirzscaled{1}$ as the symplectic reduction at $(0,0)$ of $\C^4$ by 
\[ N(u_1,u_2,u_3,u_4) = \frac{1}{2} \left( |u_1|^2 + |u_2|^2 + |u_3|^2, |u_3|^2 + |u_4|^2 \right) - \left( \alpha + \beta, \beta \right) \]
at level zero.
Consider
\begin{equation}\label{eqn:W1_J}
 J = \frac{1}{2} |u_2|^2, \qquad R = \frac{1}{2} |u_3|^2, 
\end{equation}
so that $(J,R)$ is the standard toric system on $\Hirzscaled{1}$. The fixed set of the $\S^1$-action generated by $J$ consists of the fixed sphere $u_2 = 0$ and the fixed points $C = [0,\sqrt{2\alpha},\sqrt{2\beta},0]$ and $D = [0,\sqrt{2(\alpha + \beta)},0,\sqrt{2\beta}]$.

For $j \notin \{\alpha,\alpha + \beta \}$, the reduced space $M_j^{\text{red}} = J^{-1}(j) \slash \S^1$ is a smooth symplectic 2-sphere (cf Lemma \ref{lm:sphere}), displayed in Figure \ref{fig:W1_reduced_space}. This figure was obtained as follows: consider the $J$ and $N$ invariant functions
\begin{equation} X = \Re(\bar{u}_1 u_3 \bar{u}_4), \quad Y = \Im(\bar{u}_1 u_3 \bar{u}_4). \label{eq:X_W1}\end{equation}
Using the relations $|u_1|^2 = 2(\alpha + \beta - J - R)$, $|u_2|^2 = 2J$, $|u_3|^2 = 2 R$ and $|u_4|^2 = 2(\beta - R)$, we obtain that
\begin{equation} X^2 + Y^2 = 8 R (\beta - R) (\alpha + \beta - J - R). \label{eq:W1_red}\end{equation}
The bounds for $R$ are
$ 0 \leq R \leq \min(\beta,\alpha + \beta - J) $
because $|u_3|^2 \leq 2 \beta$ and $|u_1|^2 \geq 0$.
Equation~\eqref{eq:W1_red} with $J=j$ defines $M_j^{\text{red}}$ implicitly for all $j\in [0,\alpha+\beta]$.

Observe that for $j\notin \{0,\alpha,\alpha+\beta\}$ the set
\[ M_j^{\text{red}} \setminus \{ [u] \ | \ u_1 = 0 \ \text{or}  \  u_3 = 0 \ \text{or} \  u_4 = 0 \} \]
is $M_j^{\text{red}}$ minus two points since when $[u]$ belongs to this set $u_2 \neq 0$,  and since when $0 < j < \alpha$, necessarily $u_1 \neq 0$, and when $\alpha < j < \alpha + \beta$, necessarily $u_4 \neq 0$. We obtain cylindrical coordinates $(\rho, \theta)$ on this set as follows. Using the actions generated by $J$ and $N$, we may choose a representative $(x_1,x_2,u_3,x_4)$ of $[u_1, u_2, u_3, u_4]$ such that $x_1, x_2$ and $x_4$ are real and nonnegative. Write $u_3 = \rho \exp(i \theta)$ with $\rho > 0$, $\rho^2 < \min\left(2\beta, 2(\alpha + \beta - j) \right)$ (these bounds come from $u_1, u_3, u_4 \neq 0$) and $\theta \in [0,2\pi)$; then
\[ x_1 = \sqrt{2(\alpha + \beta - j) - \rho^2}, \quad x_2 = \sqrt{2j}, \quad x_4 = \sqrt{2\beta - \rho^2}. \]

\begin{figure}
\begin{center}
\includegraphics[scale=0.40]{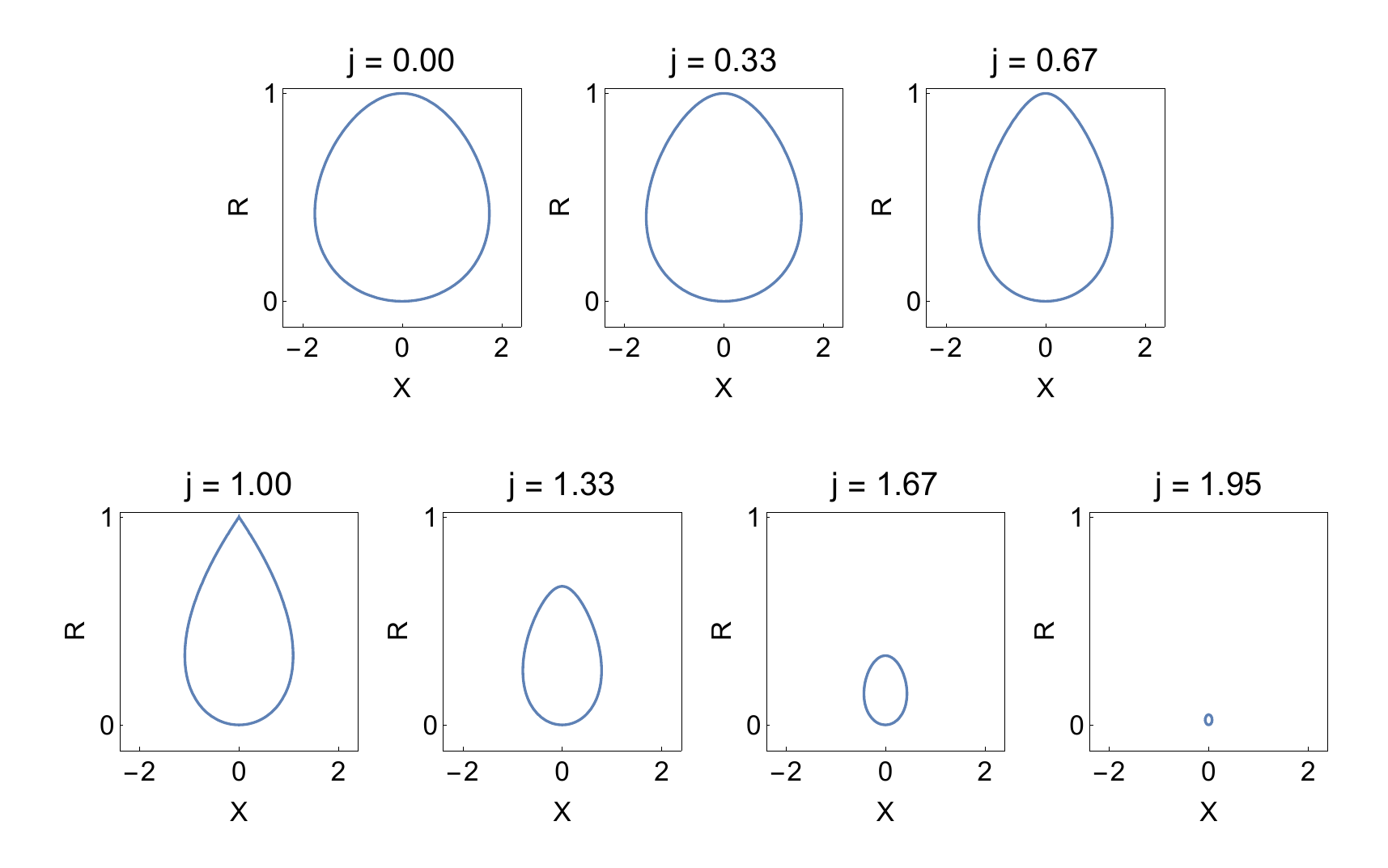}
\caption{The section $Y=0$ of the reduced space $M_j^{\text{red}}$ for $\alpha = 1 = \beta$ and several values of $j$, see Equation (\ref{eq:W1_red}). Note that $M_j^{\text{red}}$ is a surface of revolution, and that $M_j^{\text{red}}$ is a single point when $j=\alpha + \beta$. For $j=1$ it has a singular point.}
\label{fig:W1_reduced_space}
\end{center}
\end{figure}

\subsection{The system}

From Figure~\ref{fig:W1_reduced_space} and the general idea presented at the beginning of Section~\ref{sec:semitorictransfam} and Figure~\ref{fig:teardrops}, $(J,X)$ seems to be a good candidate for a semitoric system. Hence we consider, for $ \gamma \in \R$, 
\begin{equation}\label{system:W1_movingAB}
 F_t = (J,H_t),\,\,\textrm{ where }H_t = (1-t) R + t (-R + \gamma X) = (1 - 2t)R + t \gamma X,
\end{equation}
$J,R$ are as in Equation \eqref{eqn:W1_J} and $X$ is as in Equation \eqref{eq:X_W1}, so that $(J, H_0)$ is toric with image the standard Delzant polygon for $\Hirzscaled{1}$ (see Figure \ref{fig:poly_Wn}), $(J,H_1)$ is of toric type with image shown in Figure \ref{fig:moment_map_W1_noswitch} and $H_{1/2} = \frac{\gamma}{2} X$. 
\begin{rmk}
We will prove below that $(J,X)$ is indeed semitoric, and two representatives of the different classes of semitoric polygons of this system are the image of $(J,R)$ and $(J,-R)$ (note that the image of $(J,H_1)$ is similar to this semitoric polygon), but we must also include the term $\gamma X$ in the second part of the convex combination to avoid the image collapsing to a line for $t=1/2$ and produce a focus-focus singular point. Also note that the options for the function $X$ are very limited; the choice must be $J$ and $N$ invariant, low enough order to contribute to the quadratic part of $H_t$, and real-valued.
\end{rmk}

\begin{thm}\label{thm:W1_movingAB}
Under the assumption that
\begin{equation} 0 < \gamma < \frac{1}{2 \sqrt{2 \beta}}, \label{eq:hyp_gamma_W1ns}\end{equation}
$F_t$ from Equation \eqref{system:W1_movingAB} is a semitoric 1-transition family on $\Hirzscaled{1}$ with transition times $t^-, t^+$ satisfying
\[ 0 < t^- = \frac{1}{2 (1 + \gamma \sqrt{2\beta})} < t^+ = \frac{1}{2 (1 - \gamma \sqrt{2\beta})} < 1. \]
\end{thm}

The rest of this section is devoted to proving this theorem; we check that the rank zero points are non-degenerate of the correct type (except for the transition point which is degenerate for $t=t^{\pm}$) in Lemmas~\ref{lem:W1C} and~\ref{lem:W1ABD}, and we check that the rank one singular points are non-degenerate of elliptic-transverse type in Lemmas~\ref{lm:rankone_W1} and~\ref{lem:W1_fixedsphere}. In Figure \ref{fig:moment_map_W1_noswitch}, we display the image of the momentum map for several values of $t \in [0,1]$.

\begin{figure}
\begin{center}
\includegraphics[scale=0.45]{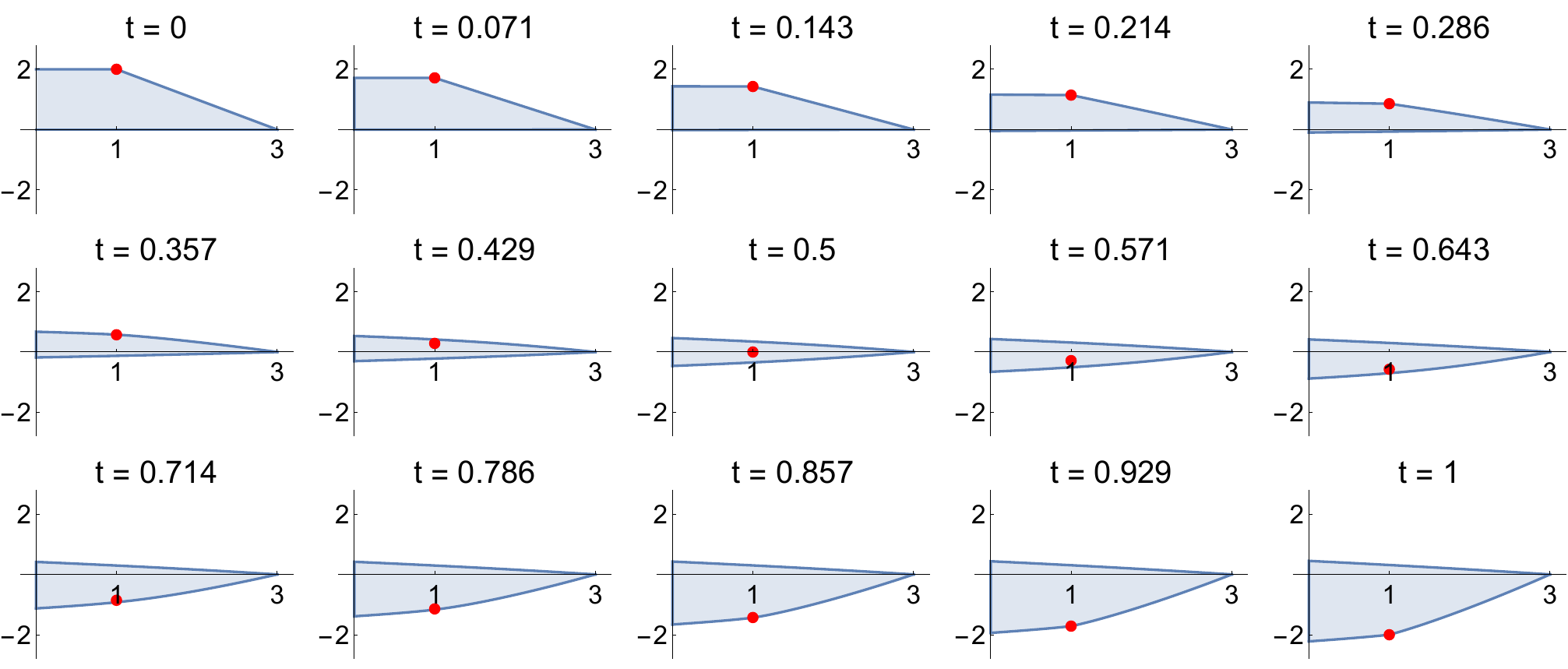}
\caption{The image $F_t(\Hirzscaled{1})$ for $F_t$ as in Equation~\eqref{system:W1_movingAB}, $\alpha = 1$, $\beta = 2$, and $\gamma = \frac{9}{20 \sqrt{2 \beta}}$. Note that in this case, $t^- = \frac{10}{29}$ and $t^+ = \frac{10}{11}$.}
\label{fig:moment_map_W1_noswitch}
\end{center}
\end{figure}

\subsection{Fixed points}

\begin{lm}
\label{lm:fixed_points_W1ns}
The fixed points of $F_t$ from Equation~\eqref{system:W1_movingAB} are the points $C = [0,\sqrt{2\alpha},\sqrt{2\beta},0]$, $D = [0,\sqrt{2(\alpha + \beta)},0,\sqrt{2\beta}]$ and
\[ \begin{cases} A_t = \left[ \sqrt{2(\alpha + \beta) - x_3^-(t)^2},0,x_3^-(t),\sqrt{2\beta - x_3^-(t)^2}\right],\\[3mm] B_t = \left[ \sqrt{2(\alpha + \beta) - x_3^+(t)^2},0,x_3^+(t),\sqrt{2\beta - x_3^+(t)^2}\right],\end{cases}\]
where $x_3^{\pm}(t)$ are the only two real solutions of the equation
\[ (1-2t) x_3 \sqrt{(2\beta - x_3^2)(2(\alpha + \beta) - x_3^2 )} + t \gamma \left( 3 x_3^4 - 4 (\alpha + 2 \beta) x_3^2 + 4 \beta (\alpha + \beta)  \right) = 0 \]
satisfying $-\sqrt{2\beta} < x_3^-(t) < x_3^+(t) < \sqrt{2\beta}$. Moreover, 
\[J(A_t) = 0, \quad J(B_t) = 0, \quad J(C) = \alpha, \quad J(D) = \alpha + \beta, \quad H_t(C) = (1-2t) \beta, \quad H_t(D) = 0.\]
\end{lm}

Note that the points $A_t$ and $B_t$ move on the fixed sphere of $J$ as $t$ varies. Note also that for $t = 1/2$, we can compute explicitly
\[ x_3^{\pm}(t) = \pm \sqrt{\frac{2}{3} \left( \alpha + 2 \beta - \sqrt{\alpha^2 + \alpha \beta + \beta^2} \right) }.  \]

\begin{proof}
The case $t = 0$ comes from the theory of toric systems, so we assume that $t \neq 0$. The fixed points for $F_t$ must lie in the fixed set of $J$, which means they can only be $C, D$ or satisfy $u_2 = 0$. We already know that the tangent maps of $R$ at $C$ and $D$ vanish because these are fixed points of the standard toric system. But since $C$ corresponds to $u_1 = 0 = u_4$ and $D$ corresponds to $u_1 = 0 = u_3$, a simple computation shows that the tangent maps of $X$ at these points also vanish. Thus the tangent maps of $H$ (a linear combination of $X$ and $R$) at $C$ and $D$ vanish. 

Since $t \neq 0$, the fixed points on $J^{-1}(0)$ belong to the open set $U_{1,4}$ defined in Section \ref{subsect:coords_W}. Using the local coordinates $x_2, y_2, x_3, y_3$ introduced in this section, and the relation
\[ x_1 = \sqrt{2(\alpha + \beta) - (x_3^2 + y_3^2) - (x_2^2 + y_2^2)}, \quad x_4 = \sqrt{2\beta - (x_3^2 + y_3^2)}, \]
we find that $J = \frac{x_2^2 + y_2^2}{2}$ and
\[ H_t = \frac{(1-2t)(x_3^2 + y_3^2)}{2} + \gamma t x_3 \sqrt{(2\beta - (x_3^2 + y_3^2))(2(\alpha + \beta) - (x_3^2 + y_3^2) - (x_2^2 + y_2^2))}. \]
Thus $J = 0$ implies $(x_2, y_2) = (0,0)$, and we look for $(x_3, y_3)$ such that $dH_t(0,0,x_3,y_3) = 0$. This amounts to
\[ \begin{cases} (1-2t) x_3 + t \gamma \sqrt{(2\beta - (x_3^2 + y_3^2))(2(\alpha + \beta) - (x_3^2 + y_3^2))} +  \frac{2 t \gamma x_3^2 (x_3^2 + y_3^2 - \alpha - 2 \beta)}{\sqrt{(2\beta - (x_3^2 + y_3^2))(2(\alpha + \beta) - (x_3^2 + y_3^2) )}}  = 0,  \\ (1-2t) y_3 + \frac{2 t \gamma x_3 y_3 (x_3^2 + y_3^2 - \alpha - 2 \beta)}{\sqrt{(2\beta - (x_3^2 + y_3^2))(2(\alpha + \beta) - (x_3^2 + y_3^2) )}}  = 0. \end{cases} \]
If $y_3 \neq 0$, the second equation yields
\[ \frac{2 t \gamma x_3 (x_3^2 + y_3^2 - \alpha - 2 \beta)}{\sqrt{(2\beta - (x_3^2 + y_3^2))(2(\alpha + \beta) - (x_3^2 + y_3^2) )}}  = 2t-1, \]
and substituting this in the first equation, we obtain that 
\[ t \gamma \sqrt{(2\beta - (x_3^2 + y_3^2))(2(\alpha + \beta) - (x_3^2 + y_3^2))} = 0, \]
which is impossible since $x_1 \neq 0$ and $x_4 \neq 0$. Therefore, $y_3 = 0$ and the first equation becomes
\begin{equation} (1-2t) x_3 + t \gamma \sqrt{(2\beta - x_3^2)(2(\alpha + \beta) - x_3^2)} +  \frac{2 t \gamma x_3^2 (x_3^2 - \alpha - 2 \beta)}{\sqrt{(2\beta - x_3^2)(2(\alpha + \beta) - x_3^2 )}}  = 0, \label{eq:first_eq_x3} \end{equation}
or equivalently
\[ (1-2t) x_3 \sqrt{(2\beta - x_3^2)(2(\alpha + \beta) - x_3^2 )} + t \gamma \left( 3 x_3^4 - 4 (\alpha + 2 \beta) x_3^2 + 4 \beta (\alpha + \beta)  \right) = 0. \]
We claim that this equation has exactly two real solutions $x_3^+(t), x_3^-(t)$, satisfying $-\sqrt{2\beta} < x_3^-(t) < 0 < x_3^+(t) < \sqrt{2\beta}$; these solutions are displayed in Figure \ref{fig:W1_x3plusminus}.
\begin{figure}
\begin{center}
\includegraphics[scale=0.4]{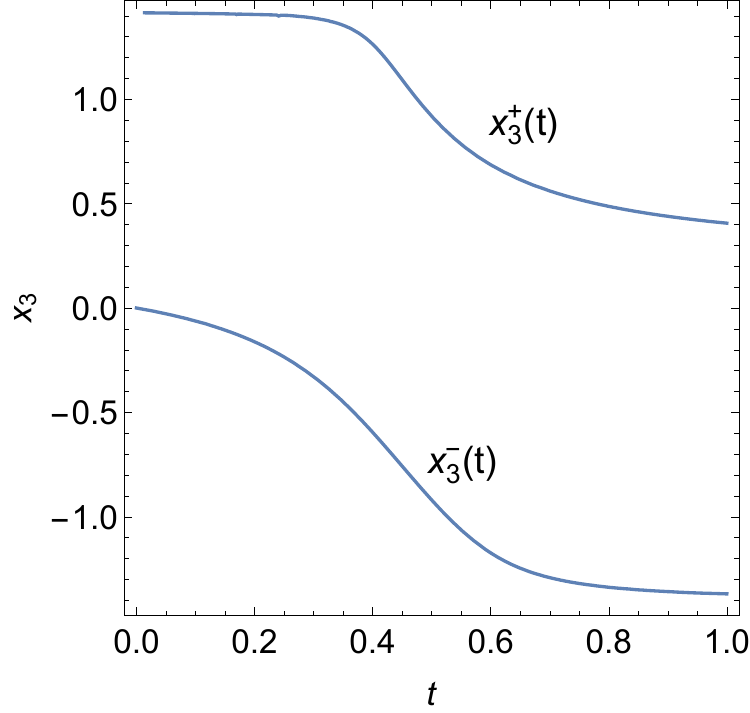}
\caption{The graphs of $x_3^-$ and $x_3^+$ as functions of $t$ for $\alpha = \beta = 1$ and $\gamma = \frac{1}{4 \sqrt{2\beta}}$.}
\label{fig:W1_x3plusminus}
\end{center}
\end{figure}
To prove the claim, let 
\[ f(x) = (1-2t) x \sqrt{(2\beta - x^2)(2(\alpha + \beta) - x^2 )} + t \gamma \left( 3 x^4 - 4 (\alpha + 2 \beta) x^2 + 4 \beta (\alpha + \beta)  \right); \]
we are looking for solutions of $f = 0$ in $(-\sqrt{2\beta},\sqrt{2\beta})$, whose squares are in particular roots of the polynomial
\[ P(X) = t^2 \gamma^2 \left( 3 X^2 - 4 (\alpha + 2 \beta) X + 4 \beta (\alpha + \beta)  \right)^2 - (1-2t)^2 X (2\beta - X)(2(\alpha + \beta) - X ) \]
in $(0,2\beta)$. We have
\[  P'(X)= (3 X^2 - 4 (\alpha +2\beta) X + 4 \beta(\alpha + \beta)) (12 \gamma^2 t^2 X - (1-2t)^2 - 8 \gamma^2 t^2(\alpha +2 \beta)) \]
which has three roots, only one of which depends on $t$. They are
\[ X^\pm = \frac{2}{3}\left(\alpha + 2\beta \pm \sqrt{\alpha^2+\alpha \beta +\beta^2}\right),\,\,X_t = \frac{(1-2t)^2+8(\alpha+2\beta)\gamma^2 t^2}{12\gamma^2 t^2}. \]
Note that $X^+ > 2 \beta$ since $\sqrt{\alpha^2+\alpha \beta +\beta^2} > \beta$, and that $0 < X^- < 2 \beta$ since we have $\alpha < \sqrt{\alpha^2+\alpha \beta +\beta^2} < \alpha + 2 \beta$. The third root $X_t$ may or may not lie in $(0,2\beta)$ depending on the parameters, but in any case $P'$ is negative on the intervals $(0,X^-)$ and $(X_t,2\beta)$ and positive on $(X^-,X_t)$; furthermore, one readily checks that $P(0) > 0, P(X^-) < 0$ and $P(2\beta) > 0$, so $P$ has exactly one zero $X_1$ in $(0,X^-)$ and one zero $X_2$ in $(X^-,2\beta)$. This gives four possible solutions $\pm \sqrt{X_1}, \pm \sqrt{X_2}$ to the equation $f(x) = 0$, but we can discard two of them since the form of $f$ implies that it can never have two zeros of the form $\pm x$, since the second term in $f$ takes the same value regardless of the choice of $x$ or $-x$ while its first term changes sign. So we find exactly two solutions $x_3^-(t)$ and $x_3^+(t)$ of $f = 0$, with the correct sign because $f(\pm \sqrt{2\beta}) = -4t\gamma \alpha \beta < 0$ and $f(0) = 4 t \gamma \beta (\alpha + \beta) > 0$.
\end{proof}

We see that in this system the fixed points on the vertical wall depend on $t$. In Section~\ref{subsect:W1_switch}, we will exhibit another system for which this is not the case.

\begin{lm}
\label{lem:W1C}
The point $C$ is elliptic-elliptic when $0 \leq t < t^-$ and $t^+ < t \leq 1$, and focus-focus when $t^- < t < t^+$, where  
\[ t^- = \frac{1}{2 (1 + \gamma \sqrt{2\beta})}, \qquad t^+ = \frac{1}{2 (1 - \gamma \sqrt{2\beta})}. \]
\end{lm}

Note that by Lemma \ref{lem:HPdegen}, it is degenerate for $t = t^-$ and $t = t^+$. Note also that because of assumption (\ref{eq:hyp_gamma_W1ns}), we have 
\[ 0 < t^- < \frac{1}{2} < t^+ < 1. \]

\begin{proof}
We do not treat the case $t=0$ since in this case the system is toric so we already know that $C$ is elliptic-elliptic. The point $C$ corresponds to $u_1 = 0 = u_4$, so we work with the local coordinates $x_1, y_1, x_4, y_4$ in the trivialization open set $U_{2,3}$, see Section \ref{subsect:coords_W}. In these coordinates, $C$ corresponds to $(0,0,0,0)$, and 
\[ H_t = \frac{(1-2t)(2\beta - (x_4^2 + y_4^2))}{2} + t \gamma (x_1 x_4 - y_1 y_4) \sqrt{2\beta - (x_4^2 + y_4^2)}. \]
The Taylor expansion of $H_t$ at $(0,0,0,0)$ reads
\[ H_t = (1-2t)\beta + \frac{2t-1}{2} (x_4^2 + y_4^2) +  t \gamma \sqrt{2 \beta} (x_1 x_4 - y_1 y_4) + O(3), \]
where $O(3)$ stands for $O(\|(x_1, y_1, x_4, y_4)\|^3)$.
Hence the Hessian of $H_t$ at $C$ in the basis corresponding to $x_1,y_1,x_4,y_4$ satisfies
\[ d^2 H_t(C) = \begin{pmatrix} 0 & 0 & t \gamma \sqrt{2 \beta} & 0 \\ 0 & 0 & 0 & -t \gamma \sqrt{2 \beta} \\ t \gamma \sqrt{2 \beta} & 0 & 2t-1 & 0 \\ 0 & -t \gamma \sqrt{2 \beta} & 0 & 2t-1  \end{pmatrix}. \]
In these local coordinates, the symplectic form at $(0,0,0,0)$ reads $\omega = dx_1 \wedge dy_1 + dx_4 \wedge dy_4$, hence the matrix $\Omega_C$ of this symplectic form in the aforementioned basis satisfies
\[ \Omega_C = \begin{pmatrix} 0 & 1 & 0 & 0 \\ -1 & 0 & 0 & 0 \\ 0 & 0 & 0 & 1 \\ 0 & 0 & -1 & 0 \end{pmatrix}, \quad \Omega_C^{-1} = \begin{pmatrix} 0 & -1 & 0 & 0 \\ 1 & 0 & 0 & 0 \\ 0 & 0 & 0 & -1 \\ 0 & 0 & 1 & 0 \end{pmatrix}, \]
and thus
\[ \Omega_C^{-1} d^2 H_t(C) = \begin{pmatrix} 0 & 0 & 0 & t\gamma\sqrt{2\beta} \\ 0 & 0 & t\gamma\sqrt{2\beta} & 0 \\ 0 & t\gamma\sqrt{2\beta} & 0 & 1-2t \\ t\gamma\sqrt{2\beta} & 0 & 2t-1 & 0 \end{pmatrix}. \]
First, we assume that $t \neq 1/2$. The reduced characteristic polynomial of $\Omega^{-1} d^2 H_t(C)$ (see Section \ref{subsect:sing_dim4} for its definition) is
\[ P = X^2 + \left( 4( 1 - \beta \gamma^2) t^2 - 4 t + 1 \right) X + 4 t^4 \gamma^4 \beta^2, \]
with discriminant $\Delta = -(2t-1)^2 \left( 4 (2 \beta \gamma^2 - 1) t^2 + 4 t - 1  \right)$. Let $c = \gamma \sqrt{2 \beta}$, so that 
\[ \Delta = -(2t-1)^2 \left( 4 (c^2 - 1) t^2 + 4 t - 1 \right) = - 4 (c^2 - 1) (2t-1)^2 \left(t - \frac{1}{2(1-c)} \right) \left(t - \frac{1}{2(1+c)} \right) \]
(note that by Equation (\ref{eq:hyp_gamma_W1ns}), $0 < c < 1/2$ so in particular $c^2 - 1 < 0$); this reads
\[ \Delta =4 (1 - c^2) (2t-1)^2 \left(t - t^+ \right) \left(t - t^- \right). \]
If $t^- < t < t^+$, then $\Delta < 0$ and $P$ has two complex roots with nonzero imaginary part, so $C$ is non-degenerate of focus-focus type. If $0 < t < t^-$ or $t^+ < t \leq 1$, $\Delta > 0$ and $P$ has two real roots
\[ \lambda^{\pm} = \frac{-\left( 4( 1 - \beta \gamma^2) t^2 - 4 t + 1 \right) \pm \sqrt{\Delta}}{2} = \frac{-\left( 4( 1 - \frac{c^2}{2}) t^2 - 4 t + 1 \right) \pm \sqrt{\Delta}}{2}. \]
Now, let
\[ b = 4\left( 1 - \frac{c^2}{2}\right) t^2 - 4 t + 1 = 4\left( 1 - \frac{c^2}{2}\right) \left( t - \frac{1}{2-c\sqrt{2}} \right) \left( t - \frac{1}{2+c\sqrt{2}} \right), \]
and note that 
\[ t^-  < \frac{1}{2+c\sqrt{2}} <  \frac{1}{2-c\sqrt{2}} < t^+, \] 
which means that $b > 0$ since $0 < t < t^-$ or $t^+ < t \leq 1$. Hence $\lambda^- < 0$; moreover, one readily checks that $b^2 - \Delta = 16 \beta^2 \gamma^4 t^4 > 0$. Consequently, $b > \sqrt{\Delta}$ and $\lambda^+ < 0$ as well. Hence $C$ is non-degenerate of elliptic-elliptic type.

If $t = 1/2$, then 
\[ \Omega^{-1} d^2 H_t(C) =  \frac{1}{2} \begin{pmatrix} 0 & 0 & 0 & \gamma\sqrt{2\beta} \\ 0 & 0 & \gamma\sqrt{2\beta} & 0 \\ 0 & \gamma\sqrt{2\beta} & 0 & 0 \\ \gamma\sqrt{2\beta} & 0 & 0 & 0 \end{pmatrix} \]
and the eigenvalues of this matrix are not distinct, so we cannot assert that $C$ is non-degenerate and infer its type right away. Since $J = \alpha - \frac{1}{2}(x_1^2 + y_1^2) + \frac{1}{2}(x_4^2 + y_4^2)$, we find 
\[ \Omega^{-1} d^2 J(C) = \begin{pmatrix} 0 & 1 & 0 & 0 \\ -1 & 0 & 0 & 0 \\ 0 & 0 & 0 & -1 \\ 0 & 0 & 1 & 0 \end{pmatrix}. \]
The linear combination
\[ \Omega^{-1} d^2 J(C) + \frac{2}{\gamma \sqrt{2\beta}} \Omega^{-1} d^2 H_t(C) = \begin{pmatrix} 0 & 1 & 0 & 1 \\ -1 & 0 & 1 & 0 \\ 0 & 1 & 0 & -1 \\ 1 & 0 & 1 & 0 \end{pmatrix} \]
has eigenvalues $\pm 1 \pm i$, so $C$ is non-degenerate of focus-focus type.
\end{proof}

The point $D$ can be treated in a similar fashion; we only give a few details of the proof. However, the treatment of the points $A_t$ and $B_t$ is different so we are more precise with them.

\begin{lm}
\label{lem:W1ABD}
The points $A_t, B_t$ and $D$ are always elliptic-elliptic.
\end{lm}

\begin{proof}
{\bf The point $D$.} We work on $U_{2,4} = \{ [u_1,u_2,u_3,u_4] \ | \ u_2 \neq 0, u_4 \neq 0 \} \subset  \Hirzscaled{1}$. Using the action of $N$, we may assume that $u_2 = x_2 \in \R^+$ and $u_4 = x_4 \in \R^+$, and we can write $u_1 = x_1 + i y_1, u_3 = x_3 + i y_3$ and use $x_1, y_1, x_3, y_3$ as local coordinates, since then
\[ x_2 = \sqrt{2(\alpha + \beta) - (x_3^2 + y_3^2) - (x_1^2 + y_1^2)}, \quad x_4 = \sqrt{2\beta - (x_3^2 + y_3^2)}. \]
In these coordinates, $D$ corresponds to $(0,0,0,0)$,
\[ H_t = \frac{(1-2t)(x_3^2 + y_3^2)}{2} + t \gamma (x_1 x_3 + y_1 y_3) \sqrt{2\beta - (x_3^2 + y_3^2)}, \]
and thus
\[ \Omega_D^{-1} d^2 H_t(D) = \begin{pmatrix} 0 & 0 & 0 & -t\gamma\sqrt{2\beta} \\ 0 & 0 & t\gamma\sqrt{2\beta} & 0 \\ 0 & -t\gamma\sqrt{2\beta} & 0 & 2t-1 \\ t\gamma\sqrt{2\beta} & 0 & 1-2t & 0 \end{pmatrix}. \]
First, we assume that $t \neq 1/2$. The reduced characteristic polynomial of $\Omega^{-1} d^2 H_t(D)$ is
\[ P = X^2 + \left( 4( 1 + \beta \gamma^2) t^2 - 4 t + 1 \right) X + 4 t^4 \gamma^4 \beta^2. \]
Its discriminant is given by 
\[ \Delta = (2t-1)^2 \left( 4( 1 + 2\beta \gamma^2) t^2 - 4 t + 1  \right) = (2t-1)^2 \left( 8\beta \gamma^2 t^2 + (2t-1)^2  \right) > 0. \]
Hence $P$ has two real roots
\[ \lambda^{\pm} = \frac{-\left( 4( 1 + \beta \gamma^2) t^2 - 4 t + 1 \right) \pm \sqrt{\Delta}}{2} < 0. \]
Consequently, $D$ is non-degenerate of elliptic-elliptic type.

If $t = 1/2$ one readily checks that 
\[ 2\Omega_D^{-1} d^2 J(D) + \frac{2}{\gamma \sqrt{2\beta}} \Omega^{-1} d^2 H_t(D) = \begin{pmatrix} 0 & 2 & 0 & -1 \\ -2 & 0 & 1 & 0 \\ 0 & -1 & 0 & 2 \\ 1 & 0 & -2 & 0 \end{pmatrix} \]
has eigenvalues $\pm i, \pm 3 i$, so $D$ is non-degenerate of elliptic-elliptic type.

\ \\

{\bf The points $A_t$ and $B_t$.} We use the same chart and notation as in the proof of Lemma \ref{lm:fixed_points_W1ns}. One readily checks that 
\[ \Omega_{(0,0,x_3,0)}^{-1} d^2 H_t(0,0,x_3,0) = \begin{pmatrix} 0 & -a(x_3) & 0 & 0 \\ a(x_3) & 0 & 0 & 0 \\ 0 & 0 & 0 & -c(x_3) \\ 0 & 0 & b(x_3) & 0 \end{pmatrix} \]
with
\[ a(x_3) =  \frac{ -t \gamma x_3 (2\beta - x_3^2)}{\sqrt{(2\beta - x_3^2)(2(\alpha + \beta) - x_3^2)}}, \quad c(x_3) = 1-2t + \frac{2t\gamma x_3 (x_3^2 - \alpha - 2\beta)}{\sqrt{(2\beta - x_3^2)(2(\alpha + \beta) - x_3^2)}}, \]
and
\[ b(x_3) = 1-2t + \frac{2t\gamma x_3 \left( 3 x_3^6 - 9(\alpha + 2\beta) x_3^4 + 4 (\alpha^2 + 9 \alpha \beta + 9 \beta^2) x_3^2 + 12 \beta (\alpha^2 + 3 \alpha \beta + 2 \beta^2) \right)}{\left( (2\beta - x_3^2)(2(\alpha + \beta) - x_3^2) \right)^{3/2}}.\]
Because of Equation (\ref{eq:first_eq_x3}), we have that 
\[ c(x_3^{\pm}(t)) = \frac{-t\gamma}{x_3^{\pm}(t)} \sqrt{(2\beta - x_3^{\pm}(t)^2)(2(\alpha + \beta) - x_3^{\pm}(t)^2)}, \]
so $c(x_3^-(t)) > 0$ and $c(x_3^+(t)) < 0$. Furthermore, using Equation (\ref{eq:first_eq_x3}) again, a straightforward computation yields
\[ b(x_3^{\pm}(t)) = \frac{t\gamma P(x_3^{\pm}(t)^2) }{x_3^{\pm}(t) \left( (2\beta - x_3^{\pm}(t)^2)(2(\alpha + \beta) - x_3^{\pm}(t)^2) \right)^{3/2}}  \]
where $P(X) = X^2 \left( 3 X^2 - 8 (\alpha + 2 \beta) X + 24 \beta (\alpha + \beta) \right) - 16 \beta^2 (\alpha + \beta)^2$. Since
\[ P'(X) = 12 X (2\beta - X) (2(\alpha + \beta) - X) > 0 \]
for $X \in (0,2\beta)$, we have that $P(X) \leq P(2\beta) = -16\alpha^2 \beta^2 < 0$. Thus, the sign of $b(x_3^{\pm}(t))$ is the opposite of the sign of $x_3^{\pm}(t)$. Therefore, $b(x_3^{\pm}(t)) c(x_3^{\pm}(t)) > 0$ for every $t$; thus, the eigenvalues of $\Omega_{(0,0,x_3^{\pm}(t),0)}^{-1} d^2 H_t(0,0,x_3^{\pm}(t),0)$ are $\pm i a(x_3^{\pm}(t))$, $\pm i \sqrt{b(x_3^{\pm}(t)) c(x_3^{\pm}(t))}$, so $A_t$ and $B_t$ are non-degenerate of elliptic-elliptic type.
\end{proof}

\subsection{Rank one points}
\label{sec:W1rankone}

We now prove that the rank one points are always non-degenerate of elliptic-transverse type; in order to do so, we will use some convenient local coordinates, but we first need to show that certain points cannot be singular of rank one.

\begin{lm}
\label{lm:discard_poles_W1}
$F_t$ as in Equation~\eqref{system:W1_movingAB} has no rank one points with $u_1 = 0$ or $u_3 = 0$ or $u_4 = 0$.
\end{lm}

\begin{proof}
Again, we may assume that $t\neq 0$.
Let $J, N, R, X$ be as in Section \ref{subsect:notation_W1}, and write $N = (N_1, N_2)$.
Since these are real-valued functions, we may use Lagrange multipliers, treating the variables
$u_j,\bar{u}_j$ as independent variables (see for instance~\cite{Kre}).
A rank one singular point is a solution of $\nabla H_t = \lambda \nabla J + \mu_1 \nabla N_1 + \mu_2 \nabla N_2$, which reads
\[ t \gamma \bar{u}_3 u_4 = \mu_1 \bar{u}_1, \quad 0 = (\lambda + \mu_1) \bar{u}_2, \quad (1 - 2t) \bar{u}_3 + t \gamma \bar{u}_1 \bar{u}_4 = (\mu_1 + \mu_2) \bar{u}_3, \quad t \gamma u_1 \bar{u}_3 = \mu_2 \bar{u}_4,  \]
plus the same equations but with every $u_j$ conjugate. If $u_1 = 0$, necessarily $u_3 = 0$ or $u_4 = 0$; we get the fixed points $C$ and $D$. If $u_4 = 0$, necessarily $u_1 = 0$ or $u_3 = 0$; the first case yields $C$, the second case is impossible since $|u_3|^2 + |u_4|^2 = 2 \beta$. Finally, if $u_3 = 0$, necessarily $u_1 = 0$ or $u_4 = 0$; the first case yields $D$, the second one is impossible.
\end{proof}

Thanks to this lemma, we can use the local coordinates $(\rho, \theta)$ on $M_j^{\text{red}} \setminus \{ [u] \ | \ u_1 = 0 \ \text{or}  \  u_3 = 0 \ \text{or} \  u_4 = 0 \}$ defined in Section \ref{subsect:notation_W1}. In these coordinates,
\[ H_t^{\text{red},j} = (1 - 2t) \frac{\rho^2}{2} + \gamma t \rho \cos \theta \sqrt{(2\beta - \rho^2)(2(\alpha + \beta - j) - \rho^2)}. \]

\begin{lm}
\label{lm:rankone_W1}
All rank one points of $F_t$ as in Equation~\eqref{system:W1_movingAB} on $\Hirzscaled{1} \setminus J^{-1}(0)$ are non-degenerate of elliptic-transverse type for all $t$. 
\end{lm}

\begin{proof}
We use Lemma \ref{lm:nondeg_rankone_red} (which, we recall, works away from the singular part of the reduced space for $j \in \{\alpha, \alpha + \beta\}$) and work with the reduced Hamiltonian $H_t^{\text{red},j}$ on $M_j^{\text{red}}$. We write
\[ H_t^{\text{red},j} = (1 - 2t) \frac{\rho^2}{2} + \gamma t \rho \cos \theta \sqrt{g(\rho)}. \]
where $g(\rho) = (2\beta - \rho^2)(2(\alpha + \beta - j) - \rho^2)$. We may assume that $\rho^2 \notin \{0, 2\beta, 2(\alpha + \beta - j)\}$ by Lemma~\ref{lm:discard_poles_W1}, since these values respectively correspond to $u_3 = 0, u_4 = 0$ and $u_1 = 0$. Singular points of $H_t^{\text{red},j}$ satisfy 
\begin{equation} 0 = \dpar{H_t^{\text{red},j}}{\theta} = - \gamma t \rho \sin \theta \sqrt{g(\rho)}, \qquad 0 = \dpar{H_t^{\text{red},j}}{\rho} = (1-2t) \rho + \gamma t h(\rho) \cos \theta. \label{eq:dHdrho=0_W1ns} \end{equation}
where 
\begin{equation} h(\rho) = \sqrt{g(\rho)} + \frac{\rho g'(\rho)}{2 \sqrt{g(\rho)}} = \frac{2 g(\rho) + \rho g'(\rho)}{2 \sqrt{g(\rho)}}. \label{eq:h_general}\end{equation}
The first equation boils down to $\sin \theta = 0$ because $\rho^2 \notin \{0, 2\beta, 2(\alpha + \beta - j)\}$. 
The derivative 
\[ \frac{\partial^2 H_t^{\text{red},j}}{\partial \theta \partial \rho} = -\gamma t h(\rho) \sin \theta  \]
therefore vanishes at a singular point. Moreover,
\[ \dpar{^2 H_t^{\text{red},j}}{\theta^2} = -\gamma t \rho \cos \theta \sqrt{g(\rho)}, \qquad \dpar{^2 H_t^{\text{red},j}}{\rho^2} = 1 - 2t + \gamma t h'(\rho) \cos \theta.  \]
We already see that $\dpar{^2 H_t^{\text{red},j}}{\theta^2} < 0$ if $\theta = 0$, and $\dpar{^2 H_t^{\text{red},j}}{\theta^2} > 0$ if $\theta = \pi$. Furthermore,
\[ h'(\rho) = \frac{2 g(\rho) ( \rho g''(\rho) + 2 g'(\rho) ) - \rho g'(\rho)^2}{4 g(\rho)^{3/2}}. \]
If $(\rho_c, \theta_c)$ is a singular point, the second equality in Equation (\ref{eq:dHdrho=0_W1ns}) yields
\[ 1 - 2 t = \frac{- \gamma t \cos \theta_c \left( 2 g(\rho_c) + \rho_c g'(\rho_c) \right)}{2 \rho_c \sqrt{g(\rho_c)}}. \]
Therefore, 
\[ \dpar{^2 H_t^{\text{red},j}}{\rho^2}(\rho_c,\theta_c) = \frac{\gamma t \cos \theta_c \left( 2 \rho_c^2 g(\rho_c) g''(\rho_c) + 2 \rho_c g(\rho_c) g'(\rho_c)  - \rho_c^2 g'(\rho_c)^2 - 4 g(\rho_c)^2 \right)}{4 \rho_c g(\rho_c)^{3/2}}. \]
We claim that $f(\rho) = 2 \rho^2 g(\rho) g''(\rho) + 2 \rho g(\rho) g'(\rho)  - \rho^2 g'(\rho)^2 - 4 g(\rho)^2$
is always negative. In order to prove this claim, we see $f$ as a second order polynomial in $j$; namely, $f = 4 P$ with
\[ P = -16\beta^2 j^2 + 8 (\rho^6 - 3\beta \rho^4 + 4 \beta^2 (\alpha + \beta)) j + 3 \rho^8 - 8 (\alpha + 2 \beta) \rho^6 + 24 \beta (\alpha + \beta) \rho^4 - 16 \beta^2 (\alpha + \beta)^2.  \] 
The discriminant $\Delta$ of $P$ reads $\Delta = -64 \rho^6 (2 \beta - \rho^2)^3 < 0$,
hence $P$ is of the sign of $ -16\beta^2 < 0$, and the claim is proved. In particular, the quantity $\dpar{^2 H_t^{\text{red},j}}{\rho^2}(\rho_c,\theta_c)$ has the sign of $-\cos \theta_c$, and the determinant of $d^2 H_t^{\text{red},j}(\rho_c, \theta_c)$ is positive. Consequently, $(\rho_c,\theta_c)$ corresponds to non-degenerate rank one points of elliptic-transverse type.
\end{proof}

\begin{lm}
\label{lem:W1_fixedsphere}
All rank one points of $F_t$ as in Equation~\eqref{system:W1_movingAB} on $J^{-1}(0)$ are non-degenerate of elliptic-transverse type for every $t \in [0,1]$.
\end{lm}

\begin{proof}
Using Definition \ref{dfn:nondeg_rankone}, a rank one point $p$ on $J^{-1}(0)$ (hence, such that $dJ(p) = 0$) is non-degenerate if and only if the restriction of $d^2 J(p)$ to $L^{\perp} \slash L$ is invertible, where $L = \mathrm{Span}\{ X_{H_t}(p) \}$. Furthermore, if this is the case, the singular point is of elliptic-transverse type if and only if the eigenvalues of $\Omega^{-1} d^2 J(p)_{|L^{\perp} \slash L}$ are of the form $\pm i \alpha$ with $\alpha \neq 0$ real, where $\Omega$ is the matrix of the symplectic form on $L^{\perp} \slash L$. As in the second part of the proof of Lemma \ref{lm:fixed_points_W1ns}, we work on the trivialization open set $U_{1,4}$ defined in Section \ref{subsect:coords_W}, with local coordinates $x_2, y_2, x_3, y_3$; in these coordinates $J = \frac{1}{2}(x_2^2 + y_2^2)$. We know from the aforementioned proof that at a singular point $p = (0,0,x_3,y_3)$ of rank one in $J^{-1}(0)$, the Hamiltonian vector field of $H_t$ reads
\[ X_{H_t}(p) = \begin{pmatrix} 0 \\ 0 \\ f_1(t,x_3,y_3) \\ f_2(t,x_3,y_3) \end{pmatrix} \]
for some functions $f_1, f_2$. Therefore, 
\[ L^{\perp} = \mathrm{Span}\left\{\begin{pmatrix} 1 \\ 0 \\ 0 \\ 0 \end{pmatrix},\begin{pmatrix} 0 \\ 1 \\ 0 \\ 0 \end{pmatrix}, X_{H_t}(p) \right\}, \quad L^{\perp} \slash L \simeq \mathrm{Span}\left\{\begin{pmatrix} 1 \\ 0 \\ 0 \\ 0 \end{pmatrix},\begin{pmatrix} 0 \\ 1 \\ 0 \\ 0 \end{pmatrix}\right\}; \]
so $d^2 J(p)$ restricts to the identity on $L^{\perp} \slash L$, and $\Omega^{-1} d^2 J(p)_{|L^{\perp} \slash L}$ has eigenvalues $\pm i$.
\end{proof}

\subsection{An explicit example with the same fixed points for all \texorpdfstring{$t$}{t}}
\label{subsect:W1_switch}

In this section, we describe another system on $\Hirzscaled{1}$, given in Equation \eqref{system:W1_switch}, which differs from the system 
from Equation \eqref{system:W1_movingAB} described above in several
important ways, but still has the same semitoric polygons when it is semitoric with one focus-focus point. Firstly, in the system described above the fixed points of $(J,H_t)$ vary in the manifold with $t$, while in the
system described below this is not the case. On the other hand, the \emph{images} of the fixed points of the system 
described in this section which are in the
set $J^{-1}(0)$ move vertically as $t$ changes and pass through each other, such that for $t<1/2$ we have
$H_t(A)<H_t(B)$, for $t>1/2$ we have $H_t(A)>H_t(B)$, and for $t=1/2$ we have $F_t(A)=F_t(B)$. Thus, at
$t=1/2$ the system is degenerate, with the entire sphere $J^{-1}(0)$ being sent to the same point by $F_{1/2}$;
this type of degeneracy, called the collapse of the fixed sphere, is described in item~\ref{item:fixedcollapse} in Section~\ref{sec:semitoricfam-degen}.
Figure \ref{fig:moment_map_W1_switch} shows the image of the momentum map of the system from Equation \eqref{system:W1_switch}
for varying values of $t$.
Because of the degeneracy at $t=1/2$, this is not a semitoric 1-transition family and thus 
this is not the system guaranteed by Theorem~\ref{thm:Wkdetailed}.

The idea is to replace $X$ by $J X$ in the system from Equation~\eqref{system:W1_movingAB} (and we also add a constant to shift the momentum map image), with $X$ as in Equation \eqref{eq:X_W1} and $J$ as in Equation \eqref{eqn:W1_J}, so that this term is zero on the fixed sphere. 
Thus, we consider the system
\begin{equation}
 F_t = (J,H_t),\,\,\textrm{ where }\,\,H_t = (1-t) R + t \left(-R + \gamma J X +\beta\right) = (1-2t) R + t(\gamma J X + \beta). 
\label{system:W1_switch}\end{equation}

\begin{thm}\label{thm:W1_switch}
For $0 < \gamma < \frac{1}{2 \alpha \sqrt{2\beta}}$ the system $F_t = (J,H_t)$ given in Equation \eqref{system:W1_switch} is a semitoric family on $\Hirzscaled{1}$ with degenerate times $t^-$, $\frac{1}{2}$, and $t^+$, where
\[ t^- = \frac{1}{2(1 + \alpha \gamma \sqrt{2 \beta})}, \qquad t^+ =  \frac{1}{2(1 - \alpha \gamma \sqrt{2 \beta})}. \]
\end{thm}

Note that since $0 < \gamma < 1$, we indeed have $0 < t^- < t^+ < 1$. Since the proof of this theorem is very similar to the proof of Theorem~\ref{thm:W1_movingAB}, we leave it to the reader. Let us simply mention that for all $t$ the fixed points include $A$, $B$, $C$, and $D$ with 
\[
\left\{\begin{aligned}
A &= [\sqrt{2(\alpha + \beta)},0,0,\sqrt{2\beta}],  & J(A) &= 0,             & H_t(A) &= t\beta, \\ 
B &= [\sqrt{2\alpha},0,\sqrt{2\beta},0],            & J(B) &=0,              & H_t(B) &= (1-t)\beta,\\ 
C &= [0,\sqrt{2\alpha},\sqrt{2\beta},0],            & J(C) &= \alpha,        & H_t(C) &= (1-t)\beta,\\ 
D &= [0,\sqrt{2(\alpha + \beta)},0,\sqrt{2\beta}],  & J(D) &= \alpha +\beta, & H_t(D) &= t\beta,
\end{aligned}\right.
\]
and that the point which transitions is again $C$. 
If $t\neq 1/2$ the only fixed points are $A,B,C,D$ and if $t=1/2$ the fixed points are $C,D$ and the entire set $J^{-1}(0)$.
Note that by Lemma \ref{lem:HPdegen}, $C$ is degenerate for $t = t^-$ and $t = t^+$. The image of $F_t$ is displayed in Figure \ref{fig:moment_map_W1_switch}.

\begin{figure}[h]
\begin{center}
\includegraphics[scale=0.45]{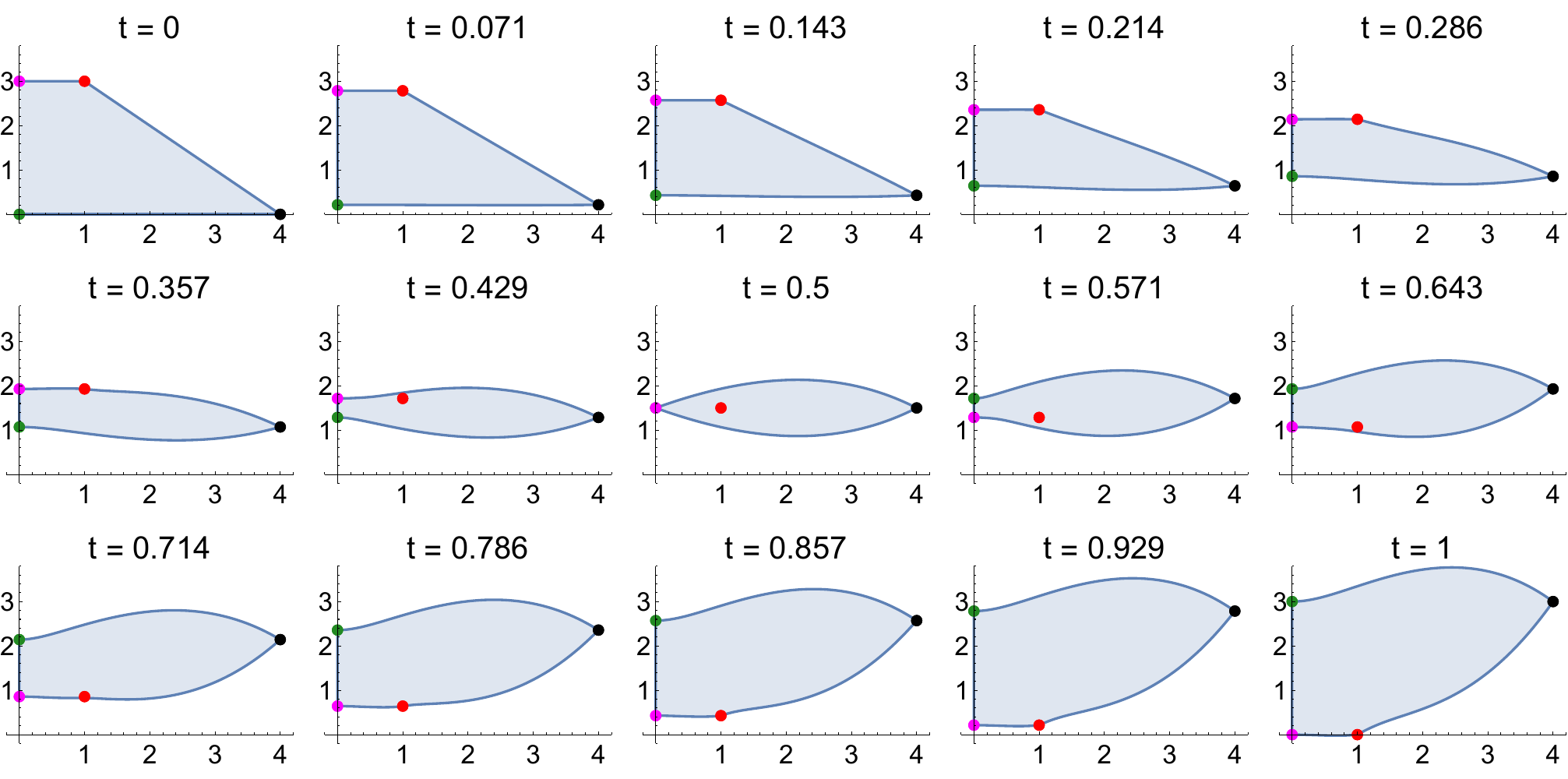}
\caption{The image $F_t(\Hirzscaled{1})$ of the system from Equation~\eqref{system:W1_switch} for $\alpha = 1$, $\beta = 3$ and $\gamma = \frac{3}{8 \alpha \sqrt{2 \beta}}$. Note that in this case, $t^- = \frac{4}{11}$ and $t^+ = \frac{4}{5}$. The images of $A$, $B$, $C$, and $D$ are colored blue, purple, red, and black respectively; note that at $t=1/2$
the images of $A$ and $B$ pass through each other. }
\label{fig:moment_map_W1_switch}
\end{center}
\end{figure}

\subsection{A \texorpdfstring{\family}{fixed-S1 family} on \texorpdfstring{$W_1(1,1)$}{W1} with hyperbolic singularities}
\label{subsect:triangle}

To conclude this section, we explain how to modify the system in Equation \eqref{system:W1_movingAB} in order to obtain a \family displaying hyperbolic singularities; more precisely, some elements in this family are so-called hyperbolic semitoric systems, as in \cite[Definition 3.2]{dullin-pelayo}. Following the techniques used in that paper, the idea is to add a higher order term to the non-periodic Hamiltonian; depending on this term, we may obtain that for some values of the parameter, the transition point is elliptic-elliptic and has image in the interior of the image of the momentum map, with curves of images of rank one points emanating from it. We achieve this by considering, on $W_1(1,1)$, $J$ as in Equation \eqref{eqn:W1_J} and 
\[ H_t = (1-2t) R + t X + 2 t |u_1|^2 |u_3|^2 \]
for $t \in [0,1]$, where $R$ is defined in Equation~\eqref{eqn:W1_J} and $X$ is as in Equation~\eqref{eq:X_W1}. Here, for simplicity, we have assumed that $\alpha = \beta = 1$ and we have taken $\gamma = 1$ in the original system \eqref{system:W1_movingAB}, but in principle the same recipe could be applied for other values of the parameters. The corresponding image of the momentum map is shown in Figure \ref{fig:triangle_W1}.
The interior lines in the image of the momentum map denote the image of non-degenerate rank one points (the ones in the segments adjacent to the red dot are elliptic-transverse and the rest are hyperbolic-transverse), except for the red point which is the image of an elliptic-elliptic point and the other two corners of the triangle, which are the images of degenerate rank one points; see Figure 4 in~\cite{dullin-pelayo}. Systems of this type and their relationship to Hamiltonian $\S^1$-spaces are discussed in \cite{HohPal21}. One of the reasons that the family of systems in Figure~\ref{fig:triangle_W1} is interesting is because it displays different
behavior for various values of $t$; it is toric for $t=0$, displays a so-called \emph{flap} for small values of $t$ (such as $t=0.33$),
and displays a so-called \emph{swallowtail} or \emph{pleat} for large values of $t$ (such as $t=1$).
For a discussion of topological types of the base of the Lagrangian fibration in integrable
systems, such as flaps and swallowtails, including physical examples of systems with these features, 
see~\cite{Efs-book,EfsSug10,EfsGia}.

\begin{figure}[h]
\begin{center}
\includegraphics[scale=0.45]{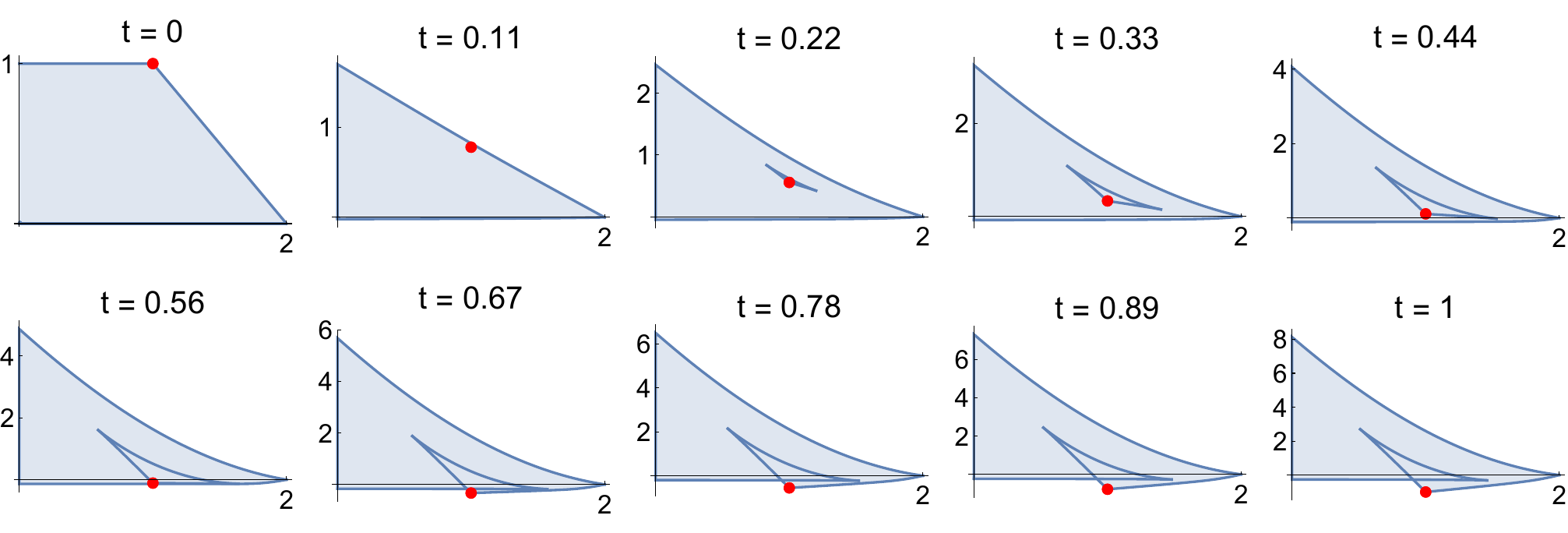}
\caption{Image of the momentum map for the system described in Section \ref{subsect:triangle}.}
\label{fig:triangle_W1}
\end{center}
\end{figure}

\section{Explicit semitoric families on \texorpdfstring{$\Hirzscaled{2}$}{W2}}
\label{sec:W2examples}

In this section we present two explicit examples of semitoric 1-transition families on $\Hirzscaled{2}$ (in
Sections~\ref{sec:W2_firstexample} and~\ref{sec:W2_secondexample}),
and in Section \ref{sec:W2_s1s2} we present a two-parameter family of semitoric systems on $\Hirzscaled{2}$
which transitions between having zero, one, and two focus-focus points depending on the values of
the parameters.

\subsection{Preliminaries on \texorpdfstring{$\Hirzscaled{2}$}{W2}}
\label{subsect:notation_W2}

As before, we see $\Hirzscaled{2}$ as the symplectic reduction of $\C^4$ by 
\[ N(u_1,u_2,u_3,u_4) = \frac{1}{2} \left( |u_1|^2 + |u_2|^2 + 2 |u_3|^2, |u_3|^2 + |u_4|^2 \right) - \left( \alpha + 2 \beta, \beta \right) \]
at level zero.
We consider the periodic Hamiltonians
\begin{equation}
\label{eqn:W2JR}
 J = \frac{1}{2}\left( |u_2|^2 + |u_3|^2 \right), \qquad R = \frac{1}{2} \left( |u_3|^2 - |u_4|^2 \right);
\end{equation}
note that $(J,R)$ is not toric (since the action generated by $R$ is not effective) but $(J,\frac{R}{2})$ is. Nevertheless, using $R$ is more convenient for our purpose. The image of the momentum map $(J,R)$ is displayed in Figure \ref{fig:W2_JR}; note that one feature of this system is that there is no vertical wall, so we can hope to write a semitoric 1-transition family for which the point with image $(\beta,\beta)$ transitions into the interior of the momentum map image and becomes a focus-focus point.

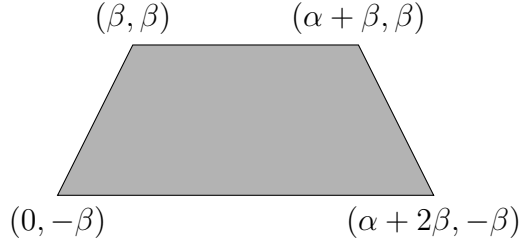
\begin{figure}
\begin{center}
\begin{tikzpicture}
\filldraw[draw=black,fill=gray!60] (0,-1) node[anchor=north,color=black]{$(0,-\beta)$}
  -- (1,1) node[anchor=south,color=black]{$(\beta,\beta)$}
  -- (4,1) node[anchor=south,color=black]{$(\alpha+\beta,\beta)$}
  -- (5,-1) node[anchor=north,color=black]{$(\alpha + 2 \beta,-\beta)$}
  -- cycle;
\end{tikzpicture}
\caption{The image of the system $(J,R)$ from Equation~\eqref{eqn:W2JR} on $\Hirzscaled{2}$.}
\label{fig:W2_JR}
\end{center}
\end{figure}

We will need the following simple lemma, which we state without proof.
\begin{lm}
The fixed points of $J$ on $\Hirzscaled{2}$ are
\[\begin{cases}
A = \left[\sqrt{2(\alpha + 2 \beta)},0,0,\sqrt{2\beta}\right], &B = \left[\sqrt{2\alpha},0,\sqrt{2\beta},0\right]\\[2mm	]
C = \left[0,\sqrt{2\alpha},\sqrt{2\beta},0\right], &D = \left[0,\sqrt{2(\alpha + 2 \beta)},0,\sqrt{2\beta}\right]
\end{cases}\]
with values $J(A) = 0$, $J(B) = \beta$, $J(C) = \alpha + \beta$, and $J(D) = \alpha + 2 \beta$.
\end{lm}

For $j \notin \{0,\alpha,\alpha + \beta, \alpha + 2\beta \}$, the reduced space $M_j^{\text{red}} = J^{-1}(j) \slash \S^1$ is a smooth symplectic 2-sphere, displayed in Figure \ref{fig:reduced_space_W2}. We obtain this figure by considering the functions $X = \Re\left( \bar{u}_1 \bar{u}_2 u_3 \bar{u}_4 \right)$ and $Y = \Im\left( \bar{u}_1 \bar{u}_2 u_3 \bar{u}_4 \right)$, which are invariant under the actions generated by $J$ and $N$.
Since $|u_1|^2 = 2\alpha + 3\beta - 2J - R$, $|u_2|^2 = 2J - \beta - R$, $|u_3|^2 = \beta + R$ and $|u_4|^2 = \beta - R$, we have that $M_j^{\text{red}}$ is described by the equation
\[ X^2 + Y^2 = \left( 2\alpha + 3\beta - 2j - R \right) \left( 2j - \beta - R \right) \left( \beta^2 - R^2 \right). \] 
As before, the set
\[ M_j^{\text{red}} \setminus \{ [u] \ | \ u_1 = 0  \ \text{or}   \  u_2 = 0 \ \text{or}   \  u_3 = 0 \ \text{or} \  u_4 = 0 \} \]
is $M_j^{\text{red}}$ minus two points. We obtain cylindrical coordinates $(\rho, \theta)$ on this set as follows. Using the actions of $J$ and $N$, we may choose a representative $(x_1,x_2,u_3,x_4)$ of $[u_1, u_2, u_3, u_4]$ such that $x_1, x_2$ and $x_4$ are real and nonnegative. We obtain by writing $u_3 = \rho \exp(i \theta)$ with $\rho > 0$, $\rho^2 < \min\left( 2 \beta, 2j, 2(\alpha + 2 \beta - j) \right)$, and $\theta \in [0,2\pi)$ that
\[ x_1 = \sqrt{2(\alpha + 2\beta - j) - \rho^2}, \quad x_2 = \sqrt{2j - \rho^2}, \quad x_4 = \sqrt{2\beta - \rho^2}. \]
In these coordinates, $R = \rho^2 - \beta$ and $X = \rho \cos \theta \sqrt{(2(\alpha + 2\beta - j) - \rho^2)(2j - \rho^2)(2\beta - \rho^2)}$.

\begin{figure}
\begin{center}
\includegraphics[scale=0.4]{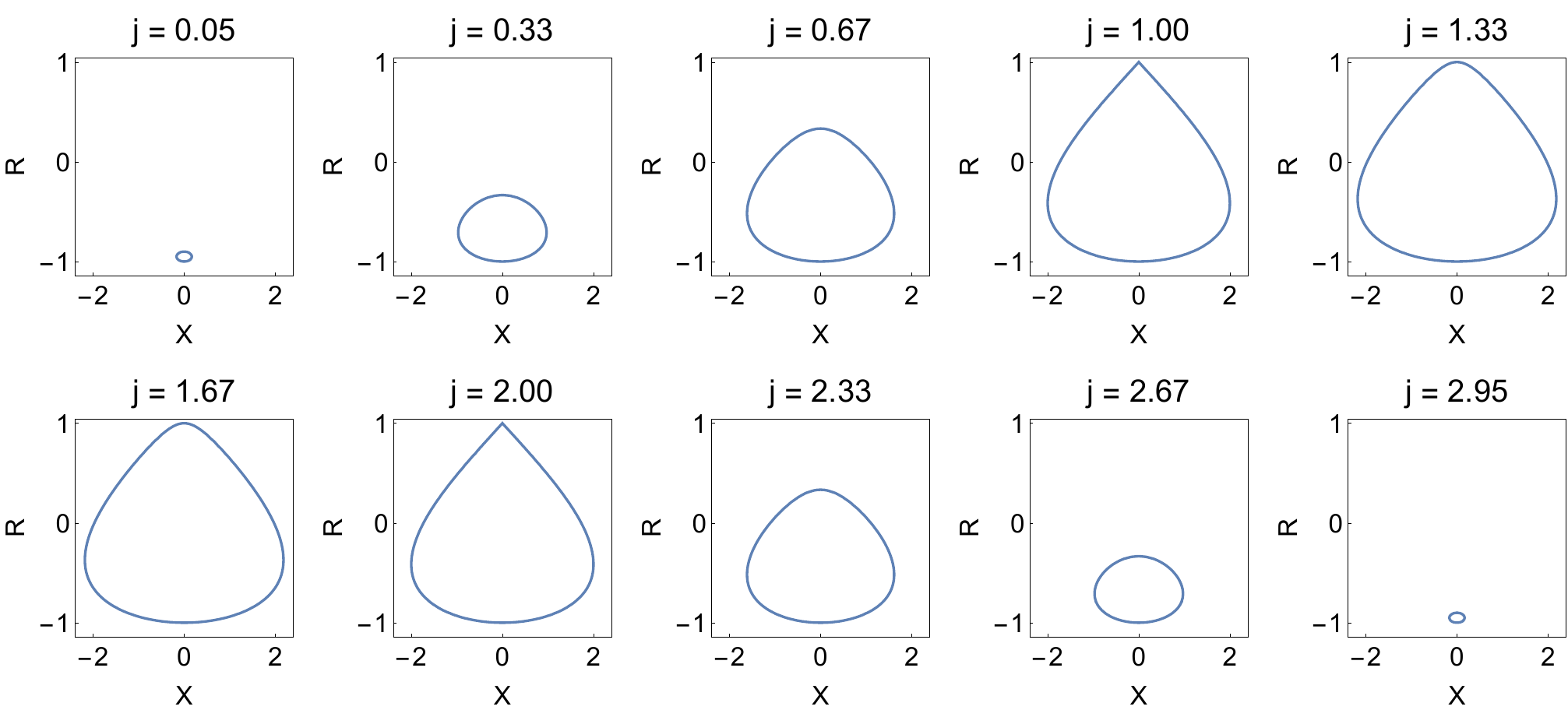}
\caption{The section $Y=0$ of the reduced space $M_j^{\text{red}}$ (which is a surface of revolution around the $R$ axis) for $\alpha = \beta = 1$ and different values of $j$. As expected, the reduced space is singular for $j = \beta$ and $j = \alpha + \beta$,
since the rank zero points $B$ and $C$ are in these levels of $J$.}
\label{fig:reduced_space_W2}
\end{center}
\end{figure}

\subsection{A semitoric 1-transition family on \texorpdfstring{$\Hirzscaled{2}$}{W2}}
\label{sec:W2_firstexample}

We consider
\[ H_t = (1-t) R + \frac{\beta t}{\alpha (\alpha + 2 \beta)} \left( \gamma X + (2J - \alpha - 2\beta)(R + \alpha + \beta) \right), \qquad \text{with} \ 0 < \gamma < \frac{1}{2} \sqrt{\frac{\alpha}{\beta}};  \]
the upper bound for $\gamma$ is required to get transition times in $(0,1)$ in the semitoric transition family that we will build. The constants are chosen so that 
$H_t(A) = -\beta$, $H_t(B) = (1-2t)\beta$, $H_t(C) = \beta$ and $H_t(D) = (2t - 1) \beta$, and thus the images of the fixed points by $H_0$ and $H_1$ are always $\pm\beta$.
We claim that $(J,H_t)$ is a transition family with transition point $B$. The momentum map image of this family for varying values of $t$ is shown in the top row of Figure~\ref{fig:moment_map_W2_array}.

\begin{thm}
\label{thm:W2_trans_B}
The system given by
\begin{equation}\label{system:W2_transB}
  J = \frac{1}{2}\left( |u_2|^2 + |u_3|^2 \right), \quad 
  H_t = (1-t) R + \frac{\beta t \left( \gamma X + (2J - \alpha - 2\beta)(R + \alpha + \beta) \right) }{\alpha (\alpha + 2 \beta) } 
\end{equation} 
on $\Hirzscaled{2}$ is a semitoric 1-transition family with transition point $B = [\sqrt{2\alpha}, 0, \sqrt{2\beta},0]$ and transition times
\[ t^- = \frac{1 + 2 \nu}{1 + (3 + c) \nu}\,\,\textrm{ and }\,\, t^+ = \frac{1 + 2 \nu}{1 + (3 - c) \nu}, \]
where $\nu = \frac{\beta}{\alpha}$ and $c = 2\gamma \sqrt{\nu}$, so that $0 < c < 1$.
\end{thm}

The rest of this section is devoted to the proof of this theorem;
indeed, in Lemmas~\ref{lem:W2B} and~\ref{lem:W2ACD} we show that the rank zero points are non-degenerate (except for the transition point at the transition times) and of the desired type and the fact that the rank one points are non-degenerate of elliptic-transverse type follows from Lemma~\ref{lm:rankone_W2}.

\begin{lm}
\label{lem:W2B}
The point $B$ is non-degenerate of focus-focus type for $t^- < t < t^+$ and of elliptic-elliptic type for $0 \leq t < t^-$ and $t^+ < t \leq 1$.
\end{lm}

\begin{proof}
The point $B$ corresponds to $u_2 = 0 = u_4$; we use the local coordinates $x_2, y_2, x_4, y_4$ on  $U_{1,3}$ defined in Section \ref{subsect:coords_W}, and write 
\[ x_1 = \sqrt{2\alpha + 2(x_4^2 + y_4^2) - (x_2^2 + y_2^2)}, \quad x_3 = \sqrt{2\beta - (x_4^2 + y_4^2)}. \]
In these coordinates, $B$ corresponds to $(0,0,0,0)$, $J = \beta + \frac{1}{2}(x_2^2 + y_2^2) - \frac{1}{2}(x_4^2 + y_4^2)$, $R = \beta - (x_4^2 + y_4^2)$, and
\[ X = (x_2 x_4 - y_2 y_4) \sqrt{\left( 2 \alpha + 2 (x_4^2 + y_4^2) - (x_2^2 + y_2^2) \right)\left( 2\beta - (x_4^2 + y_4^2) \right)}. \]
The Taylor expansion of $H_t$ at $(0,0,0,0)$ reads
\[ H_t = (1-2t)\beta + \frac{\beta t (x_2^2 + y_2^2)}{\alpha} + \left(t - 1 - \frac{2 \beta^2 t}{\alpha (\alpha + 2 \beta)} \right) (x_4^2 + y_4^2) + \frac{2 \beta \gamma \sqrt{\alpha \beta} t}{\alpha (\alpha + 2 \beta)} (x_2 x_4 - y_2 y_4) + O(3). \]
Hence we obtain that 
\[ \Omega_B^{-1} d^2 H_t(B) = 2 \begin{pmatrix} 0 & -\frac{\beta t}{\alpha} & 0 & \frac{\beta \gamma \sqrt{\alpha \beta} t}{\alpha (\alpha + 2 \beta)} \\ \frac{\beta t}{\alpha} & 0 & \frac{\beta \gamma\sqrt{\alpha \beta} t}{\alpha (\alpha + 2 \beta)} & 0 \\ 0 & \frac{\beta \gamma \sqrt{\alpha \beta} t}{\alpha (\alpha + 2 \beta)} & 0 & -\left(t - 1 - \frac{2 \beta^2 t}{\alpha (\alpha + 2 \beta)} \right) \\ \frac{\beta \gamma \sqrt{\alpha \beta} t}{\alpha (\alpha + 2 \beta)} & 0 & t - 1 - \frac{2 \beta^2 t}{\alpha (\alpha + 2 \beta)} & 0 \end{pmatrix}. \] 
The discriminant of the reduced characteristic polynomial $P$ of $\frac{1}{2}\Omega_B^{-1} d^2 H_t(B)$ reads
\[ \Delta = \frac{((1 + \nu - 4 \nu^2)t - 1 - 2 \nu)^2 (1 + 2 \nu - (1 + 3 \nu)(1-\lambda)t) (1 + 2 \nu - (1 + 3 \nu)(1+\lambda)t)}{(1 + 2 \nu)^4}, \]
where 
\[ \lambda = \frac{2 \gamma \nu}{1 + 3 \nu} \sqrt{\nu}. \]
One readily checks that the quantity $(1 + \nu - 4 \nu^2)t - 1 - 2 \nu$ never vanishes for $t \in [0,1]$ whenever $\nu > 0$; hence $\Delta > 0$ except when $t^- \leq t \leq t^+$ where 
\[ t^- = \frac{1 + 2 \nu}{(1 + 3 \nu)(1+\lambda)} = \frac{1 + 2 \nu}{1 + (3 + c) \nu}, \quad  t^+ = \frac{1 + 2 \nu}{(1 + 3 \nu)(1-\lambda)} = \frac{1 + 2 \nu}{1 + (3 - c) \nu}.  \] 
Thus when $t^- < t < t^+$, $\Delta < 0$, hence $P$ has two complex roots with nonzero imaginary part, therefore $B$ is non-degenerate of focus-focus type. When $0 \leq t < t^-$ or $t^- \leq t \leq t^+$, $\Delta > 0$ and $P$ has two distinct real roots, which are both negative. Indeed, these roots are $\frac{-b \pm \sqrt{\Delta}}{2}$
where $P = X^2 + b X + c$. A straightforward computation shows that $b$ is a second order polynomial in $t$ whose discriminant equals
\[ \frac{2\nu^2 (c^2 - 2 - 8 \nu (1 + \nu))}{(1+2\nu)^2} < 0  \]
since $0 < c < 1$ and $\nu > 0$, so $b$ never vanishes. Moreover, $b = 1$ for $t=0$, so $b > 0$. Furthermore, one readily checks that 
\[ b^2 - \Delta = \frac{\left( ( (1 + 3 \nu)^2 \lambda^2 + 4 \nu (1 + 2 \nu) (2 \nu^2 - 2 \nu - 1) )t + 4 \nu (1 + 2 \nu)^2 \right)^2 t^2}{4 (1 + 2 \nu)^4} > 0. \]
This analysis proves that $-b + \sqrt{\Delta} < 0$, hence both roots of $P$ are negative. Consequently, $B$ is non-degenerate of elliptic-elliptic type. 
\end{proof}

\begin{lm}
\label{lem:W2ACD}
The points $A, C, D$ are non-degenerate of elliptic-elliptic type for every $t \in [0,1]$.
\end{lm}

\begin{proof}
As usual, we may assume that $t \neq 0$ since the statement is clear in the case $t=0$ for which the system $(J,\frac{R}{2})$ is toric.

{\bf The point $A$.} The point $A$ corresponds to $u_2 = 0 = u_3$, so we use the local coordinates $x_2,y_2,x_3,y_3$ on $U_{1,4}$ as in Section \ref{subsect:coords_W}, in which $A$ corresponds to $(0,0,0,0)$, $R = (x_3^2 + y_3^2) - \beta$, $J = \frac{1}{2}(x_2^2 + y_2^2) + \frac{1}{2}(x_3^2 + y_3^2)$ and 
\[ X = (x_2 x_3 + y_2 y_3) \sqrt{\left( 2 (\alpha + 2 \beta) - 2 (x_3^2 + y_3^2) - (x_2^2 + y_2^2) \right)\left( 2\beta - (x_3^2 + y_3^2) \right)}, \]
The Taylor expansion of $H_t$ at $(0,0,0,0)$ reads 
\[ H_t = -\beta + \frac{\beta t (x_2^2 + y_2^2)}{\alpha + 2 \beta} + \left(1 - t - \frac{2 \beta^2 t}{\alpha (\alpha + 2 \beta)} \right) (x_3^2 + y_3^2) + \frac{2 \beta \gamma t}{\alpha} \sqrt{\frac{\beta}{\alpha + 2 \beta}} (x_2 x_3 + y_2 y_3) + O(3). \]
Hence the reduced characteristic polynomial $P$ of
\[ \frac{1}{2} \Omega_A^{-1} d^2 H_t(A) = \begin{pmatrix} 0 & -\frac{\beta t}{\alpha + 2 \beta}  & 0 & -\frac{\beta \gamma t}{\alpha} \sqrt{\frac{\beta}{\alpha + 2 \beta}} \\  \frac{\beta t}{\alpha + 2 \beta}  & 0 & \frac{\beta \gamma t}{\alpha} \sqrt{\frac{\beta}{\alpha + 2 \beta}} & 0 \\ 0 &-\frac{\beta \gamma t}{\alpha} \sqrt{\frac{\beta}{\alpha + 2 \beta}} & 0 & -\left(1 - t - \frac{2 \beta^2 t}{\alpha (\alpha + 2 \beta)} \right) \\ \frac{\beta \gamma t}{\alpha} \sqrt{\frac{\beta}{\alpha + 2 \beta}} & 0 & 1 - t - \frac{2 \beta^2 t}{\alpha (\alpha + 2 \beta)} & 0 \end{pmatrix}. \] 
has discriminant 
\[ \Delta = \frac{((1 + \nu + 2 \nu^2)t - 1 - 2 \nu)^2 Q(t)}{(1 + 2 \nu)^3}, \]
where 
\[ Q(t) = (1 + 4 \nu + (c^2 + 5 ) \nu^2 + 2 \nu^3) t^2 - 2 (1 + 3 \nu + 2 \nu^2) t + 1 + 2 \nu \]
is itself a quadratic polynomial in $t$ whose discriminant equals $-4(1+2\nu)c^2 \nu^2 < 0$. Hence $Q > 0$, so $\Delta > 0$ as well (except when $(1 + \nu + 2 \nu^2)t - 1 - 2 \nu = 0$, but in this case one readily checks that the eigenvalues of the linear combination $\Omega_A^{-1} d^2 H(A) + \frac{2 \nu}{1 + \nu + 2 \nu^2} \Omega_A^{-1} d^2 J(A) $ have the desired type). Now, let $P = X^2 + b X + c$, so that the two real roots of $P$ are $ \frac{-b \pm \sqrt{\Delta}}{2}$. Again, $b$ is a second order polynomial in $t$ whose discriminant equals 
\[ - \frac{2 \nu^2 (c^2 + 2 + 2 c^2 \nu) }{(1 + 2 \nu)^2} < 0 \]
and such that $b(0) = 1$, so $b > 0$. Furthermore
\[ b^2 - \Delta = \frac{ \nu^2 t^2  \left( ( 2(c^2 + 4)\nu^2 + c^2 \nu + 4 (1 + 2 \nu) )t - 4 \nu (1 + 2 \nu) \right)^2}{4 (1 + 2 \nu)^4} > 0 \]
(except in the case $( 2(c^2 + 4)\nu^2 + c^2 \nu + 4 (1 + 2 \nu) )t - 4 \nu (1 + 2 \nu) = 0$, left to the reader), hence both roots of $P$ are negative, and $A$ is non-degenerate of elliptic-elliptic type.\\ 

{\bf The point $C$.} The point $C$ satisfies $u_1 = 0 = u_4$, so we use the local coordinates $x_1,y_1,x_4,y_4$ on $U_{2,3}$ as in Section \ref{subsect:coords_W}, in which $C$ corresponds to $(0,0,0,0)$, $J = \alpha + \beta - \frac{1}{2}(x_1^2 + y_1^2) + \frac{1}{2}(x_4^2 + y_4^2)$, $R = \beta - (x_4^2 + y_4^2)$, and 
\[ X = (x_1 x_4 - y_1 y_4) \sqrt{\left( 2 \alpha + 2 (x_4^2 + y_4^2) - (x_1^2 + y_1^2) \right)\left( 2\beta - (x_4^2 + y_4^2) \right)}. \]
The discriminant of the reduced characteristic polynomial $P$ of $\frac{1}{2}\Omega_C^{-1} d^2 H_t(C)$ is
\[ \Delta = \frac{((1 + 3 \nu + 4 \nu^2)t - 1 - 2 \nu)^2 (1 + 2 \nu - (1 + (1+c)\nu)t) (1 + 2 \nu - (1 + (1-c)\nu)t)}{(1 + 2 \nu)^4}. \]
Since $0 < c < 1$, we have that $1 + 2 \nu > 1 + (1+c)\nu > 1 + (1-c) \nu $,
so $\Delta > 0$ (except in the case $(1 + 3 \nu + 4 \nu^2)t - 1 - 2 \nu = 0$, left to the reader). Writing $P = X^2 + b X + c$, $b$ is a second order polynomial in $t$ whose discriminant equals 
\[ \frac{2 \nu^2 (c^2 - 2 - 8\nu (1+\nu)) }{(1 + 2 \nu)^2} < 0 \]
and such that $b(0) = 1$, so $b > 0$. Furthermore
\[ b^2 - \Delta = \frac{ \nu^2 t^2  \left( ( 4 + (c^2 + 16) \nu  + 24 \nu^2 + 16 \nu^3 )t - 4 \nu (1 + 2 \nu)^2 \right)^2}{4 (1 + 2 \nu)^4} > 0\]
(except when $( 4 + (c^2 + 16) \nu  + 24 \nu^2 + 16 \nu^3 )t - 4 \nu (1 + 2 \nu)^2 = 0$, but then one can again find a suitable linear combination of $\Omega_C^{-1} d^2 J(C)$ and $\Omega_C^{-1} d^2 H_t(C)$), hence both roots of $P$ are negative, and $C$ is non-degenerate of elliptic-elliptic type.\\

{\bf The point $D$.} The point $D$ corresponds to $u_1 = 0 = u_3$, thus we work with the local coordinates $x_1,y_1,x_3,y_3$ on $U_{2,4}$ as in Section \ref{subsect:coords_W}. In these coordinates, $D$ corresponds to $(0,0,0,0)$, $R = (x_3^2 + y_3^2) - \beta$, $J = \alpha + 2 \beta - \frac{1}{2}(x_1^2 + y_1^2) - \frac{1}{2}(x_3^2 + y_3^2)$, and 
\[ X = (x_1 x_3 + y_1 y_3) \sqrt{\left( 2 (\alpha + 2 \beta) - 2 (x_3^2 + y_3^2) - (x_1^2 + y_1^2) \right)\left( 2\beta - (x_3^2 + y_3^2) \right)}. \]
The discriminant of the reduced characteristic polynomial $P$ of $\frac{1}{2} \Omega_D^{-1} d^2 H_t(D)$ is
\[ \Delta = \frac{(1 + 2 \nu - (1 + 3\nu - 2 \nu^2)t )^2 Q(t)}{(1 + 2 \nu)^3}, \]
where 
$ Q(t) = (1 + (c^2-3) \nu^2 + 2 \nu^3) t^2 - 2 (1 + \nu + 2 \nu^2) t + 1 + 2 \nu $
is a quadratic polynomial in $t$ with discriminant $-4(1+2\nu)c^2 \nu^2 < 0$. Hence $Q > 0$ (again, unless $(1 + 2 \nu - (1 + 3\nu - 2 \nu^2)t  = 0$, left to the reader), so $\Delta > 0$ as well. Now, let $P = X^2 + b X + c$, so that the two real roots of $P$ are $ \frac{-b \pm \sqrt{\Delta}}{2}$. Again, $b$ is a second order polynomial in $t$ whose discriminant equals 
\[ - \frac{2 \nu^2 (c^2 + 2 + 2 c^2 \nu) }{(1 + 2 \nu)^2} < 0 \]
and such that $b(0) = 1$, so $b > 0$. Furthermore
\[ b^2 - \Delta = \frac{ \nu^2 t^2  \left( ( 2(c^2 + 4)\nu^2 + (c^2 - 8) \nu - 4 )t + 4 \nu (1 + 2 \nu) \right)^2}{4 (1 + 2 \nu)^4} > 0 \]
(unless $( 2(c^2 + 4)\nu^2 + (c^2 - 8) \nu - 4 )t + 4 \nu (1 + 2 \nu)$, left to the reader), hence both roots of $P$ are negative, and $D$ is non-degenerate of elliptic-elliptic type.
\end{proof}

\subsection{Another semitoric 1-transition family on \texorpdfstring{$\Hirzscaled{2}$}{W2}}
\label{sec:W2_secondexample}

By starting with the same system $(J,R)$ at $t=0$, but defining a new function $H_t$, we can produce a
semitoric family in which the point $C$ is the transition point instead of $B$.
The proof of the following is similar to the proof of Theorem~\ref{thm:W2_trans_B}, so we omit it.

\begin{thm}
\label{thm:W2_trans_C}
The system given by
\begin{equation}\label{system:W2_trans_C} 
  J = \frac{1}{2}\left( |u_2|^2 + |u_3|^2 \right), \quad H_t = (1-t) R + \frac{(\alpha + \beta) t}{\alpha (\alpha + 2 \beta)} \left(\gamma X - (2J-\alpha - 2 \beta)\left(R + \frac{\beta^2}{\alpha + \beta} \right) \right) 
\end{equation}
on $\Hirzscaled{2}$ is a semitoric 1-transition family with transition point $C$ and transition times
\[ t^- = \frac{1 + 2\nu}{2 + c + (3+c) \nu}, \qquad t^+ = \frac{1 + 2\nu}{2 - c + (3 - c) \nu}.\]
\end{thm}
The image of the momentum map of this system for varying values of $t$ is shown in the leftmost
column of Figure~\ref{fig:moment_map_W2_array} ($t=0$ is the upper left image and $t$ increases
to $1$ moving down the column; in other words, $s_1 = 0$ and $t$ corresponds to $1-s_2$).

\subsection{A two-parameter family on \texorpdfstring{$\Hirzscaled{2}$}{W2}}
\label{sec:W2_s1s2}

In fact, both systems described above (in Sections~\ref{sec:W2_firstexample} and~\ref{sec:W2_secondexample}) are part of a
two-parameter family of integrable systems, $(J,H_{s_1,s_2})$, $s_1, s_2 \in [0,1]$, which are semitoric for almost all values of $(s_1,s_2)$. The idea is very similar to the one given in Equation~\eqref{eqn:HPsystem}, originally introduced in \cite{HohPal}, where the authors exhibit a two-parameter family of integrable systems on $\S^2 \times \S^2$ semitoric for almost all values of the parameters. Taking the parameters $(s_1,s_2) = (1/2,1/2)$ for both systems, the system introduced in the present paper has the same unmarked semitoric polygon as the system from~\eqref{eqn:HPsystem} whenever $\alpha = 2(R_2 - R_1)$ and $\beta = 2 R_1$.
However, there is no reason to believe these systems are isomorphic as semitoric systems, since their Taylor series, height, or twisting index invariants may differ. In Appendix \ref{sect:appendix}, we investigate this by computing the height invariants of both systems.

The system that we consider is $F_{s_1,s_2} = (J,H_{s_1,s_2})$ with
\begin{equation}\label{eqn:W2system2param}
 J = \frac{1}{2}\left( |u_2|^2 + |u_3|^2 \right),\,\,
 H_{s_1,s_2} = (1-s_1)(1-s_2) H_{00} + s_2 (1-s_1) H_{01} + s_1 (1-s_2) H_{10} + s_1 s_2 H_{11}
\end{equation}
where $H_{01} = R$, $H_{10} = -R$ and 
\[ H_{00} =  \frac{(\alpha + \beta) \left(\gamma X - (2J-\alpha - 2 \beta)(R + \frac{\beta^2}{\alpha + \beta} ) \right)}{\alpha (\alpha + 2 \beta)} , \ H_{11} = \frac{\beta \left( \gamma X + (2J - \alpha - 2\beta)(R + \alpha + \beta) \right) }{\alpha (\alpha + 2 \beta)}. \]
Here, we need to choose $\gamma$ so that
\begin{equation} \frac{1}{2 (1 + 2\nu) \sqrt{\nu}} < \gamma < \frac{1}{2 \sqrt{\nu}}, \qquad \textrm{ with }\nu =\frac{\beta}{\alpha} \label{eq:range_gamma}\end{equation}
as can be seen from the proof below.
Note that the system in Theorem \ref{thm:W2_trans_B} corresponds to $H_{t,1}$ (first row in Figure \ref{fig:moment_map_W2_array}), while the system in Theorem \ref{thm:W2_trans_C} corresponds to $H_{0,1-t}$ (first column in Figure~\ref{fig:moment_map_W2_array}). In this section, we will not fully prove that the system is a semitoric system for almost all $(s_1,s_2)$; indeed, we will only prove that the rank one points are non-degenerate of elliptic-transverse type for every choice of $(s_1,s_2)$ (see Lemma~\ref{lm:rankone_W2}), and that the fixed points are non-degenerate of the desired type when $(s_1,s_2) = (1/2,1/2)$, see Lemma~\ref{lm:W2_halfhalf} (which implies that this is also the case in an open neighborhood of $(1/2,1/2)$ in the parameter space). However, we will give numerical evidence for this fact (see Figure \ref{fig:moment_map_W2_array}). Since, as explained above, the case $H_{t,1}$ has been rigorously treated (and the case $H_{0,1-t}$ is extremely similar), this yields strong evidence that the system is indeed semitoric with the desired number of focus-focus points for every choice of $(s_1,s_2)$. We also include a picture showing the different regions of $[0,1] \times [0,1]$ in which $B$ and $C$ are either both elliptic-elliptic, or elliptic-elliptic and focus-focus, or focus-focus and elliptic-elliptic, or both focus-focus, see Figure \ref{fig:regions_W2}. In order to obtain this picture, we numerically computed the signs of the discriminants of $\Omega_B^{-1} d^2 H_{s_1,s_2}(B)$ and $\Omega_C^{-1} d^2 H_{s_1,s_2}(C)$.

\begin{figure}
\begin{center}
\includegraphics[scale=.5]{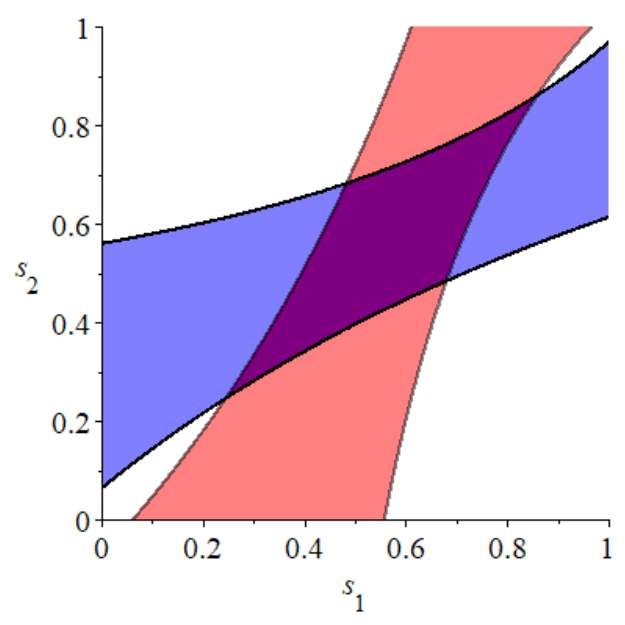}
\caption{A plot showing the types of the points $B$ and $C$ depending on the parameters $s_1$ and $s_2$ for the system given in Equation~\eqref{eqn:W2system2param},
where $\alpha = 1$, $\beta = 1$ and $\gamma = \frac{9}{20 \sqrt{\nu}}$.
In the white region both are elliptic-elliptic, in the red region $B$ is focus-focus while $C$ is elliptic-elliptic,
in the blue region $B$ is elliptic-elliptic while $C$ is focus-focus, and in the central purple region both are focus-focus. The solid lines separating the different regions correspond to parameters for which at least one of these two points is degenerate. This can be compared with Figure \ref{fig:moment_map_W2_array}.}
\label{fig:regions_W2}
\end{center}
\end{figure}

\begin{figure}
\begin{center}
\includegraphics[scale=0.5]{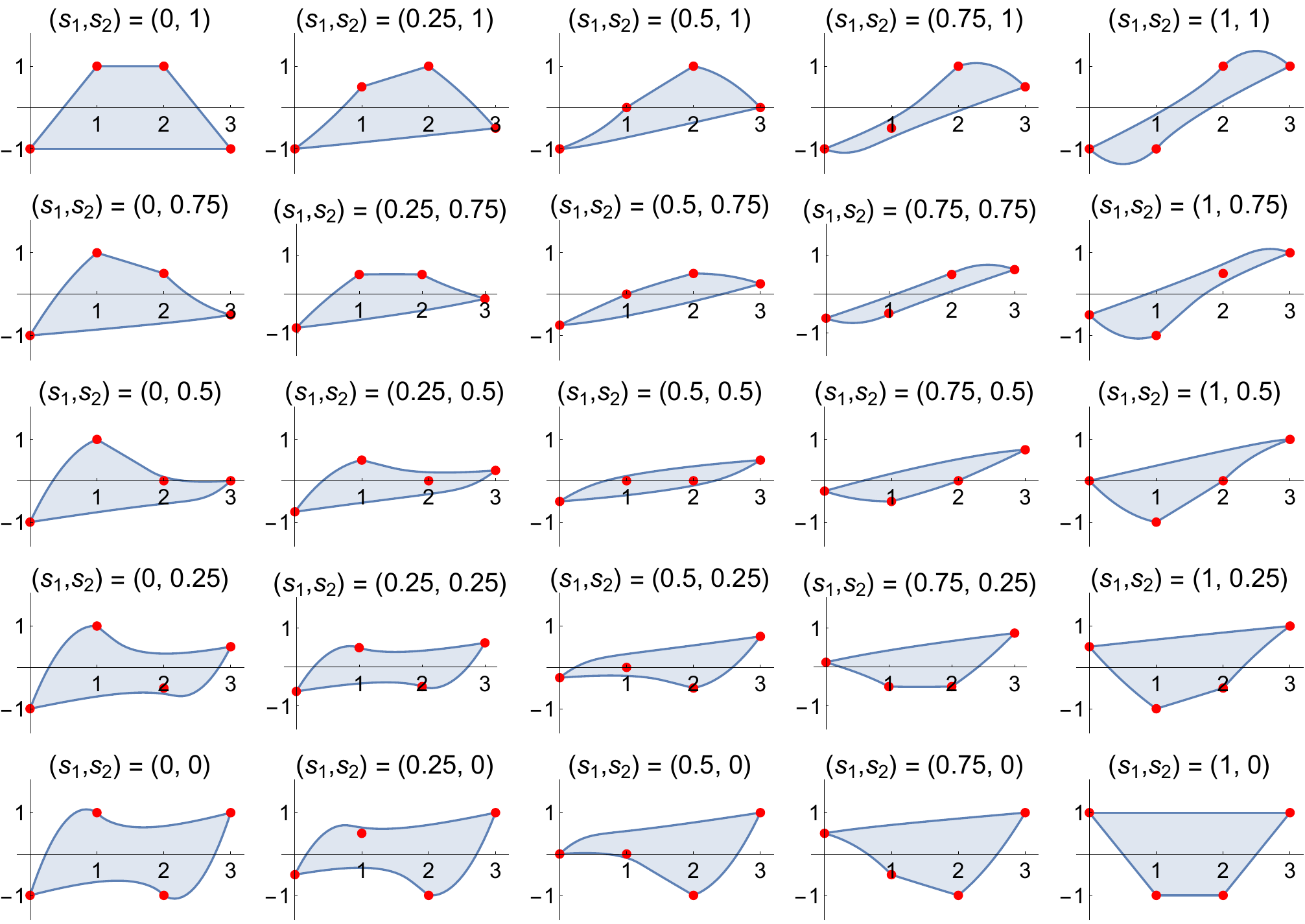}
\caption{The image $F_{s_1,s_2}(\Hirzscaled{2})$ with $F_{s_1,s_2}$ as in Equation~\eqref{eqn:W2system2param}, $\alpha = 1$, $\beta = 1$, and $\gamma = \frac{9}{20 \sqrt{\nu}}$. This should be compared to Figure \ref{fig:HPsystem}, displaying a similar system on $\S^2 \times \S^2$ constructed in \cite{HohPal} and described in Section \ref{sec:knownex}.}
\label{fig:moment_map_W2_array}
\end{center}
\end{figure}

\begin{lm}
\label{lm:rankone_W2}
For every choice of $(s_1,s_2) \in [0,1] \times [0,1]$, the rank one singular points of $(J,H_{s_1,s_2})$ are non-degenerate of elliptic-transverse type.
\end{lm}

\begin{proof}
We may assume that $(s_1,s_2) \notin \{(0,1),(1,0)\}$ because we already know that the corresponding systems are toric up to vertical scaling. As in Section~\ref{sec:W1rankone}, we start by proving that there is no rank one point with $u_{\ell} = 0$ for some $\ell \in \llbracket 1, 4 \rrbracket$. Since the idea is the same as in Lemma \ref{lm:discard_poles_W1}, we only give partial details. We write $H_{s_1,s_2} = c_1 X + c_2 R + c_3 J + c_4 J R + c_5$ for some constants $c_1,c_2,c_3,c_4,c_5$ that we will not write explicitly, such that $c_1 \neq 0$.
A rank one singular point is a solution of $\nabla H_{s_1,s_2} = \lambda \nabla J + \mu_1 \nabla N_1 + \mu_2 \nabla N_2$, which reads
\[ \begin{cases} c_1 u_2 \bar{u}_3 u_4 = \mu_1 \bar{u}_1, \\
c_1 u_1 \bar{u}_3 u_4 + c_3 \bar{u}_2 + \frac{c_4}{2} (|u_3|^2 - |u_4|^2) \bar{u}_2 = (\lambda + \mu_1) \bar{u}_2 \\
c_1 \bar{u}_1 \bar{u}_2 \bar{u}_4 + (c_2 + c_3) \bar{u}_3 + \frac{c_4}{2}(|u_2|^2 + 2 |u_3|^2 - |u_4|^2) \bar{u}_3 = (\lambda + 2 \mu_1 + \mu_2) \bar{u}_3, \\
c_1  u_1 u_2 \bar{u}_3 - c_2 \bar{u}_4 - \frac{c_4}{2} (|u_2|^2 + |u_3|^2) \bar{u}_4 = \mu_2 \bar{u}_4, \end{cases} \]
plus the same equations with conjugate $u_j$. If $u_1 = 0$, the first equation implies that $u_2 = 0$ or $u_3 = 0$ or $u_4 = 0$ since $c_1 \neq 0$; the first case is impossible since $|u_1|^2 + |u_2|^2 + 2 |u_3|^2 = 2(\alpha + 2 \beta)$ and $|u_3|^2 \leq 2 \beta$, and the other ones give the fixed points $D$ and $C$. If $u_2 = 0$, the second line yields $u_1 = 0$ (impossible for the same reason), $u_3 = 0$ (fixed point $A$), or $u_4 = 0$ (fixed point $B$). If $u_3 = 0$, the third equation implies that $u_1 = 0$ (fixed point $D$) or $u_2 = 0$ (fixed point $A$) or $u_4 = 0$ (impossible since $|u_3| + |u_4|^2 = 2 \beta$). Finally, if $u_4 = 0$, the last line gives $u_1 = 0$ (fixed point $C$) or $u_2 = 0$ (fixed point $B$) or $u_3 = 0$ (impossible for the same reason).

So we use the coordinates $(\rho,\theta)$ on $M_j^{\text{red}}$ introduced in Section \ref{subsect:notation_W2}, and thus we obtain that $H_{s_1,s_2}^{\text{red},j} = a \rho^2 + b \rho  \cos \theta \sqrt{g(\rho)} + c $ where 
\begin{equation}\label{eqn:foureqns}
\begin{cases}
a = s_2 - s_1 + \frac{(2j-\alpha-2\beta) \left( s_1 s_2 \beta - (1-s_1)(1-s_2)(\alpha + \beta) \right) }{\alpha(\alpha + 2 \beta)},\\[.7em] 
b = \frac{(1-s_1)(1-s_2)(\alpha + \beta) + s_1 s_2 \beta}{\alpha(\alpha + 2 \beta)} \gamma,\\[.7em] 
c = (1 - s_1 - s_2 + 2 s_1 s_2) \frac{(2j - \alpha - 2 \beta) \beta}{\alpha + 2 \beta} + (s_1 - s_2) \beta,\\[.7em]
g(\rho) = (2(\alpha + 2\beta - j) - \rho^2)(2j - \rho^2)(2\beta - \rho^2).
\end{cases}
\end{equation}
Singular points of $ H_{s_1,s_2}^{\text{red},j}$ satisfy 
\begin{equation} 0 = \dpar{H_{s_1,s_2}^{\text{red},j}}{\theta} =  -b \rho  \sin \theta \sqrt{g(\rho)}, \qquad 0 = \dpar{H_{s_1,s_2}^{\text{red},j}}{\rho} = 2 a \rho + b  h(\rho) \cos \theta \label{eq:dHdrho=0_W2} \end{equation}
with $h$ as in Equation \eqref{eq:h_general} (with the function $g$ now being as in Equation~\eqref{eqn:foureqns}). The first equality implies that $\theta \in \{0,\pi\}$, and as before, the partial derivative 
\[ \frac{\partial^2 H_{s_1,s_2}^{\text{red},j}}{\partial \theta \partial \rho} = -b h(\rho) \sin \theta\]
vanishes at a singular point. Moreover,
\[ \dpar{^2 H_{s_1,s_2}^{\text{red},j}}{\theta^2} = -b \rho  \cos \theta \sqrt{g(\rho)}, \qquad \dpar{^2 H_{s_1,s_2}^{\text{red},j}}{\rho^2} = 2 a + b h'(\rho) \cos \theta. \]
The first of these derivatives has the sign of $-\cos \theta$. Using Equation (\ref{eq:dHdrho=0_W2}), and proceeding as in the proof of Lemma \ref{lm:rankone_W1}, we obtain that at a singular point $(\rho_c,\theta_c)$,
\[ \dpar{^2 H_{s_1,s_2}^{\text{red},j}}{\rho^2}(\rho_c,\theta_c) =  \left(  \frac{2 \rho_c^2 g(\rho_c) g''(\rho_c) + 2 \rho_c g(\rho_c) g'(\rho_c)  - \rho_c^2 g'(\rho_c)^2 - 4 g(\rho_c)^2 }{4 \rho_c g(\rho_c)^{3/2}} \right) b \cos \theta_c. \]
Again, one can check that the numerator $f(\rho) = 2 \rho^2 g(\rho) g''(\rho) + 2 \rho g(\rho) g'(\rho)  - \rho^2 g'(\rho)^2 - 4 g(\rho)^2$ is always negative for $j,\rho$ in the range described earlier, which concludes the proof. The proof of this last claim follows the same lines as the end of the proof of Lemma \ref{lm:rankone_W1}, except that this time $f$ is a degree 4 polynomial in $j$, which makes it a bit more technical; see Figure \ref{fig:rankone_W2} for a plot of $f$ as a function of $\rho$ for several values of $j$.
\begin{figure}
\begin{center}
\includegraphics[scale=.45]{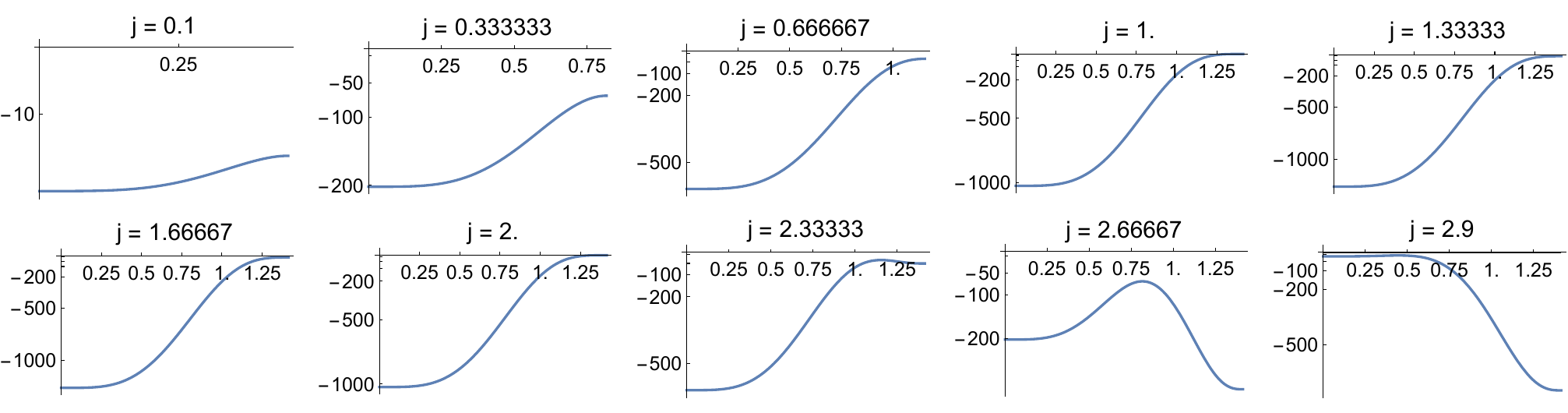}
\caption{Graph of $f: \rho \mapsto 2 \rho^2 g(\rho) g''(\rho) + 2 \rho g(\rho) g'(\rho)  - \rho^2 g'(\rho)^2 - 4 g(\rho)^2$ for different values of $j$ (recall $g$, and hence $f$, depends on the parameter $j$); here $\alpha = \beta =1$. Note that for $j=1$ and $j=2$, $f$ vanishes when $\rho = \sqrt{2}$; this corresponds to $u_4 = 0$, which yields fixed points and not rank one points.}
\label{fig:rankone_W2}
\end{center}
\end{figure}
\end{proof}

\begin{lm}\label{lm:W2_halfhalf}
The points $B$ and $C$ are both non-degenerate of focus-focus type for the system $(J,H_{1/2,1/2})$ in Equation~\eqref{eqn:W2system2param}, and the points $A$ and $B$ are non-degenerate of elliptic-elliptic type. 
\end{lm}

\begin{proof}
Note that 
\[ H_{1/2,1/2} = \frac{1}{4} \left( \frac{\gamma X }{\alpha} + \frac{(2J - \alpha - 2 \beta)( \beta - R)}{\alpha + 2 \beta}   \right). \]
A straightforward computation yields
\[ \Omega_A^{-1} d^2 H_{1/2,1/2}(A) = \frac{1}{2} \begin{pmatrix} 0 & -\frac{2\beta}{\alpha + 2 \beta}  & 0 & -\frac{\gamma \sqrt{\beta(\alpha + 2 \beta)}}{\alpha} \\  \frac{2\beta}{\alpha + 2 \beta}  & 0 &\frac{\gamma \sqrt{\beta(\alpha + 2 \beta)}}{\alpha}  & 0 \\ 0 &-\frac{\gamma \sqrt{\beta(\alpha + 2 \beta)}}{\alpha} & 0 & -\frac{\alpha}{\alpha + 2 \beta} \\ \frac{\gamma \sqrt{\beta(\alpha + 2 \beta)}}{\alpha} & 0 & \frac{\alpha}{\alpha + 2 \beta} & 0 \end{pmatrix}. \]
The reduced characteristic polynomial $P = X^2 + bX + c$ of the latter has discriminant 
\[ \Delta = \frac{ 4 \nu (1 + 2\nu)^3  \gamma^2 + (1 - 2 \nu)^2}{(1 + 2 \nu)^2} > 0 \]
and has two negative roots since 
\[ b = \frac{1 + 4 \nu^2 + 2 \nu (1 + 2 \nu)^3}{(1 + 2 \nu)^2} > 0, \quad b^2 - \Delta = \frac{ \nu^2 \left( (1+2\nu)^3  \gamma^2 - 2 \right)^2 }{(1 + 2 \nu)^4} > 0 \]
except when $\gamma = \sqrt{\frac{2}{(1+2\nu)^3}}$ (but this case can be checked by finding a suitable linear combination of $\Omega_A^{-1} d^2 H_{1/2,1/2}(A)$ and $\Omega_A^{-1} d^2 J(A)$). Hence $A$ is non-degenerate of elliptic-elliptic type. Similarly, the reduced characteristic polynomial of 
\[ 2 \Omega_B^{-1} d^2 H_{1/2,1/2}(B) = \begin{pmatrix} 0 & 0 & 0 & \gamma \sqrt{\nu} \\  0 & 0 & \gamma \sqrt{\nu} & 0 \\ 0 & \gamma \sqrt{\nu} & 0 & \frac{1}{1 + 2 \nu} \\ \gamma \sqrt{\nu} & 0 & -\frac{1}{1 + 2 \nu} & 0 \end{pmatrix} \]
is $Q = X^2 + \frac{ 1 - 2\gamma^2 \nu(1 + 2\nu)^2}{(1 + 2 \nu)^2} X + \gamma^4 \nu^2$
and its discriminant equals 
\[ \frac{1 - 4 \gamma^2 \nu (1 + 2\nu)^2}{(1 + 2 \nu)^4}, \]
and is negative since $\gamma > \frac{1}{2(1+2\nu)\sqrt{\nu}}$. Hence $Q$ has two complex roots with nonzero imaginary part, so $B$ is non-degenerate of focus-focus type. Moreover,
\[ 2 \Omega_C^{-1} d^2 H_{1/2,1/2}(C) = \begin{pmatrix} 0 & 0 & 0 & \gamma \sqrt{\nu} \\  0 & 0 & \gamma \sqrt{\nu} & 0 \\ 0 & \gamma \sqrt{\nu} & 0 & -\frac{1}{1 + 2 \nu} \\ \gamma \sqrt{\nu} & 0 & \frac{1}{1 + 2 \nu} & 0 \end{pmatrix} \]
has the same eigenvalues as $2 \Omega_B^{-1} d^2 H_{1/2,1/2}(B)$, so $C$ is also non-degenerate of focus-focus type. Finally, the discriminant of the reduced characteristic polynomial $R = X^2 + bX + c$ of
\[ \Omega_D^{-1} d^2 H_{1/2,1/2}(D) = \frac{1}{2} \begin{pmatrix} 0 & \frac{2\beta}{\alpha + 2 \beta}  & 0 & -\frac{\gamma \sqrt{\beta(\alpha + 2 \beta)}}{\alpha} \\  -\frac{2\beta}{\alpha + 2 \beta}  & 0 &\frac{\gamma \sqrt{\beta(\alpha + 2 \beta)}}{\alpha}  & 0 \\ 0 &-\frac{\gamma \sqrt{\beta(\alpha + 2 \beta)}}{\alpha} & 0 & \frac{\alpha+4\beta}{\alpha + 2 \beta} \\ \frac{\gamma \sqrt{\beta(\alpha + 2 \beta)}}{\alpha} & 0 & -\frac{\alpha+4\beta}{\alpha + 2 \beta} & 0 \end{pmatrix} \]
reads
\[ \Delta =  \frac{(1 + 6 \nu)^2 (1 + 4 \gamma^2 \nu + 8 \gamma^2 \nu^2 )}{(1 + 2\nu)^2} > 0. \]
Both roots of $R$ are negative since
\[ b = \frac{1 + 8\nu + 20\nu^2 + 2 \nu (1 + 2\nu)^3 \gamma^2 }{(1 + 2\nu)^2} > 0, \quad b^2 - \Delta = \frac{4 \nu^2 \left( (1 + 2\nu)^3 \gamma^2 - 2 (1 + 4 \nu) \right)^2 }{(1 + 2\nu)^4} > 0. \]
(except when $ (1 + 2\nu)^3 \gamma^2 = 2 (1 + 4 \nu)$, left to the reader). Hence $D$ is non-degenerate of elliptic-elliptic type.
\end{proof}

\begin{rmk}
One could try to find a similar two-parameter family of systems on $\Hirzscaled{n}$ with $n \geq 3$ starting from the analogous $(J,R)$, but it seems unlikely that such a family would exist. This is because it would seem the system with $s_1=s_2=1/2$ would have two focus-focus points and a semitoric polygon with no vertical wall. By~\cite[Theorem 1.4]{KPP_min}, the only such semitoric polygon, up to rescaling the lengths of the edges while preserving the slopes, is the one associated to the system in Section~\ref{sec:W2_s1s2} (shown in Figure~
\ref{fig:stpolygon-2ff}), but this cannot be the semitoric polygon for a system with two focus-focus points on $\Hirzscaled{n}$ because, for instance, it does not give the correct weights of the $\mathbb{S}^1$-action associated to $J$ at the maximum and minimum values of $J$. Of course, by Remark~\ref{rmk:weightsnot1}, we know that none of the representatives of the semitoric polygons of a semitoric 2-transition family on $\Hirzscaled{n}$ can be the usual toric polygon for $\Hirzscaled{n}$ because the weights of the $J$-action for the latter are not $\pm 1$ at both transition points.

Still, we could try to write down such a system in the same fashion. The most reasonable $J$ and $N$-invariant choice for $X$ would be $X = \Re\left( \bar{u}_1^{n-1} \bar{u}_2 u_3 \bar{u}_4 \right)$; and indeed, this does not work because rank one points that are not of elliptic-transverse type show up. 
In particular, the points with $u_1 = 0$  cannot be discarded as in Lemma \ref{lm:rankone_W2}, because 
the first equation in the system obtained by writing $\nabla H_{s_1,s_2} = \lambda \nabla J + \mu_1 \nabla N_1 + \mu_2 \nabla N_2$ reads
$ c_1 (n-1) u_1^{n-2} u_2 \bar{u}_3 u_4 = \mu_1 \bar{u}_1, $
and for $n \geq 3$, this equation does not imply that $u_2 \bar{u}_3 u_4 = 0$ if $u_1 = 0$. We omit
the details, but it turns out that these points are indeed rank one points which are not
of elliptic-transverse type.
\end{rmk}

\appendix

\section{Comparison \texorpdfstring{of the systems in~\eqref{eqn:W2system2param} and~\eqref{eqn:HPsystem}}{} using the height invariant}
\label{sect:appendix}

In this appendix we compare the systems
$(\Hirzscaled{2}, \omega_{\Hirzscaled{2}},(J,H_{1/2,1/2}))$ from Equation~\eqref{eqn:W2system2param} in Section \ref{sec:W2_s1s2} and $(\S^2 \times \S^2, R_1 \omega_{\S^2} \oplus R_2 \omega_{\S^2}, (J,H_{1/2, 1/2}))$ in Equation~\eqref{eqn:HPsystem}, originally introduced in \cite{HohPal}.
When $(\alpha, \beta)$ and $(R_1, R_2)$  are such that both systems have the same unmarked semitoric polygon, the ambient 
manifolds $(\Hirzscaled{2},\omega_{\Hirzscaled{2}})$ 
and $(\S^2 \times \S^2, R_1 \omega_{\S^2} \oplus R_2 \omega_{\S^2})$ are symplectomorphic;
this is well-known and proved for instance in \cite[Lemma 3]{Kar02}. In this case the systems could
in principle be isomorphic as
semitoric systems. We will prove that this is not the case, except possibly for one special choice of 
the extra parameter $\gamma$ in the first system, by computing their height invariants and seeing that they differ.
Note that we only give integral formulas for these invariants, since it is sufficient for our purpose; however, one can obtain closed forms for the integrals involved.

\subsection{The height invariant}

The additional information contained in the marked semitoric polygon compared to the unmarked semitoric polygon (see Section \ref{sec:semitoric}) is the so-called \emph{height invariant}. Keeping the notation from the aforementioned section, let $\left( \Delta_{\vec{\epsilon}}, \vec{c} = ( (u_1,v_1), \ldots, (u_{m_f},v_{m_f})), \vec{\epsilon} \right)$ be any representative of the marked semitoric polygon of a semitoric system $(M, \omega, (J,H))$, and let  
\[ \vec{h} = (h_1, \ldots, h_{m_f}) = \left( v_1 - \min_{\pi_1^{-1}(u_1) \cap \Delta_{\vec{\epsilon}}} \pi_2, \ldots, v_{m_f} - \min_{\pi_1^{-1}(u_{m_f}) \cap \Delta_{\vec{\epsilon}}} \pi_2 \right). \]
This quantity does not depend on the choice of representative of the marked semitoric polygon, and is called the height invariant of the system. In practice, this invariant can be computed as follows. Let $\ell \in \{1, \ldots, m_f \}$ and let $(x_{\ell}, y_{\ell})$ be the corresponding focus-focus value. Let $(M_j^{\text{red}}, \omega^{\text{red}}_j)$ be the symplectic reduction of $M$ with respect to the $\S^1$-action generated by $J$ at level $j$, and let $H^{\text{red},j}$ be the Hamiltonian induced by $H$ on $M_j^{\text{red}}$. Then $2 \pi h_{\ell}$ is the symplectic area (with respect to $\omega^{\text{red}}_j$) of $\{ m \in M_{x_{\ell}}^{\text{red}} \ | \ H^{\text{red},x_{\ell}}(m) < y_{\ell} \}$. 

\subsection{The system from Equation~(\ref{eqn:W2system2param})}

We start by computing the height invariant $\vec{h} = (h_1, h_2)$ of the semitoric system\\ 
$(\Hirzscaled{2}, \omega_{\Hirzscaled{2}},(J,H_{1/2,1/2}))$ defined in Equation~\eqref{eqn:W2system2param}. The two focus-focus values of this system are $(\beta,0)$ and $(\alpha + \beta,0)$.

First, we compute $h_1$. Using the local cylindrical coordinates from Section~\ref{subsect:notation_W2}, the symplectic form on $M_{\beta}^{\text{red}}$ is $\omega^{\text{red}}_\beta = \rho \ d\rho \wedge d\theta$, and a straightforward computation shows that the reduced Hamiltonian reads
\[ H^{\text{red}, \beta}_{1/2,1/2} = \frac{(2 \beta - \rho^2)}{4} \left( \frac{\gamma \rho \cos \theta \sqrt{2(\alpha + \beta) - \rho^2}}{\alpha} - \frac{\alpha}{\alpha + 2 \beta} \right).  \]
Hence $H^{\text{red}, \beta}_{1/2,1/2} < 0$ if and only if $\cos(\theta) < f(\rho)$ where
\[  f(\rho) = \frac{\alpha^2}{(\alpha + 2 \beta) \gamma \rho \sqrt{2(\alpha + \beta) - \rho^2}}. \]
One readily checks that $f(\rho)$ belongs to $[0,1]$ if and only if $\rho \geq \rho^-$ where 
\[ \rho^- = \sqrt{ \alpha + \beta -\frac{ \sqrt{(\alpha + 2 \beta)^2 (\alpha + \beta)^2 \gamma^2 - \alpha^4}}{(\alpha + 2\beta)\gamma}}. \]
Therefore, $H^{\text{red}, \beta}_{1/2,1/2} < 0$ if and only if
\begin{itemize}
\item either $\rho^- \leq \rho \leq \sqrt{2\beta}$ and $\arccos(f(\rho)) < \theta < 2\pi - \arccos(f(\rho))$, 
\item or $\rho < \rho^-$ (and there is no constraint on $\theta$).
\end{itemize}
Consequently, 
\[ h_1 = \frac{1}{2\pi} \int_{\rho^-}^{\sqrt{2\beta}} \left( 2 \pi - 2 \arccos(f(\rho))  \right) \rho \ d\rho + \int_0^{\rho^-} \rho \ d\rho. \]
Since $\int_0^{\sqrt{2\beta}} \rho \ d\rho = \beta$, we finally obtain that
\begin{equation} h_1 = \beta - \frac{I_{W_2}}{\pi}, \qquad I_{W_2} = \int_{\rho^-}^{\sqrt{2\beta}} \rho \arccos(f(\rho)) \ d\rho . \label{eq:h1_W2} \end{equation}
We obtain in a similar fashion that
$ h_2 = \frac{I_{W_2}}{\pi} = \beta - h_1. $

\subsection{The system from Equation~(\ref{eqn:HPsystem})}

Now, we compute the height invariant of the semitoric system $(\S^2 \times \S^2, R_1 \omega_{\S^2} \oplus R_2 \omega_{\S^2}, (J,H_{1/2, 1/2}))$ defined in \cite{HohPal} (see Equation \eqref{eqn:HPsystem}) as
\[ J = R_1 z_1 + R_2 z_2, \quad H_{1/2,1/2} = \frac{1}{4} \left( z_1 + z_2 + 2 (x_1 x_2 + y_1 y_2) \right) \]
where $(x_i, y_i, z_i)$ are the usual Cartesian coordinates on the $i$-th copy of $\S^2$. This system also has two focus-focus values, namely $(R_1 - R_2, 0)$ and $(R_2 - R_1, 0)$.

In order to compute $h_1$, we use the coordinates $(\rho, \alpha)$ on $M_{R_1 - R_2}^{\text{red}}$ with
\[
  \rho = \sqrt{\frac{1-z_1}{1+z_1}}, \qquad \alpha = \theta_2-\theta_1
\]
defined in \cite[Section 3.5]{LFP} (here $(z_i, \theta_i)$ are the usual cylindrical coordinates on each copy of $\mathbb{S}^2$). In these coordinates, the reduced symplectic form and reduced Hamiltonian read
\[ \omega^{\text{red}}_{R_1-R_2} = \frac{4 R_1 \rho}{(1 + \rho^2)^2} d\rho \wedge d\alpha, \qquad H^{\text{red},R_1-R_2}_{\frac{1}{2},\frac{1}{2}} = \frac{4 \rho^2 \cos \alpha \sqrt{( \Theta + (\Theta - 1) \rho^2 )} + (1 - \Theta)(1 + \rho^2) \rho^2}{2 \Theta (1 + \rho^2)^2} \]
where $\Theta = \frac{R_2}{R_1}$. Therefore, $H^{\text{red},R_1-R_2}_{\frac{1}{2},\frac{1}{2}}(\rho,\alpha) < 0$  if and only if 
\[ \cos \alpha < \frac{(\Theta-1)(1+\rho^2)}{4 \sqrt{\Theta + (\Theta - 1) \rho^2}}, \]
which amounts to 
\begin{itemize}
\item either $0 < \rho < \sqrt{\frac{9-\Theta + 4 \sqrt{5}}{\Theta - 1}}$ and $\arccos\left( \frac{(\Theta - 1)(1 + \rho^2)}{4 \sqrt{\Theta + (\Theta - 1) \rho^2}} \right) < \alpha < 2 \pi - \arccos\left( \frac{(\Theta - 1)(1 + \rho^2)}{4 \sqrt{\Theta + (\Theta - 1) \rho^2}} \right)$,
\item or $\rho \geq \sqrt{\frac{9-\Theta + 4 \sqrt{5}}{\Theta - 1}}$ (in which case there is no constraint on $\alpha$).
\end{itemize}

Using this, one readily checks that
\begin{equation} h_1 = 2R_1 \left( 1 - \frac{2I_{\S^2 \times \S^2}}{\pi} \right), \quad I_{\S^2 \times \S^2} = \int_{0}^{\sqrt{\frac{9-\Theta + 4 \sqrt{5}}{\Theta - 1}}} \frac{\rho}{(1+\rho^2)^2} \arccos\left( \frac{(\Theta - 1)(1 + \rho^2)}{4 \sqrt{\Theta + (\Theta - 1) \rho^2}} \right) \ d\rho.  \label{eq:h1_S2}\end{equation}
A similar computation yields $h_2 = 2 R_1 - h_1$.

\subsection{Comparison between the two systems}

A representative of the marked semitoric polygon of each of these two systems is shown in Figure \ref{fig:poly_comparison}; we see that the corresponding representatives of the unmarked semitoric polygons coincide if and only if 
\begin{equation} \alpha = 2(R_2 - R_1) \qquad \textrm{and} \qquad \beta = 2R_1, \label{eq:R_alpha_beta} \end{equation} in which case all representatives coincide.
Hence we now assume that this equation is satisfied.

\begin{figure}[h]
\begin{center}

 \begin{subfigure}[b]{.45\linewidth}
  \begin{center}
   \begin{tikzpicture}[scale=.9]
    \filldraw[draw=black,fill=gray!60] (0,0) node[anchor=north,color=black]{\footnotesize{$(0,0)$}}
      -- (2,2) node[anchor=south,color=black]{\footnotesize{$(\beta,\beta)$}}
      -- (4,2) node[anchor=south,color=black]{\footnotesize{$(\alpha+\beta,\beta)$}}
      -- (6,0) node[anchor=north,color=black]{\footnotesize{$(\alpha + 2 \beta,0)$}}
      -- cycle;
     \draw [dashed] (2,2) -- (2,1.2);
     \draw (2,1.2) node {$\times$};	
		 \draw [dashed] (4,2) -- (4,0.8);
 
     \draw (4,0.8) node {$\times$};
   \end{tikzpicture}
  \end{center}
 \caption{The system from Equation~\eqref{eqn:W2system2param}.}
 \end{subfigure}\,\,\,
 \begin{subfigure}[b]{.45\linewidth}
  \begin{center}
   \begin{tikzpicture}[scale=.9]
    \filldraw[draw=black,fill=gray!60] (0,0) node[anchor=north,color=black]{\footnotesize{$(-(R_1 + R_2), 0)$}}
      -- (2,2) node[anchor=south east,color=black]{\footnotesize{$(R_1 - R_2, 2 R_1)$}}
      -- (4,2) node[anchor=south west,color=black]{\footnotesize{$(R_2 - R_1, 2 R_1)$}}
      -- (6,0) node[anchor=north,color=black]{\footnotesize{$(R_1 + R_2, 0)$}}
      -- cycle;
     \draw [dashed] (2,2) -- (2,0.6);
     
     \draw (2,0.6) node {$\times$};
	\draw [dashed] (4,2) -- (4,1.4);

	\draw (4,1.4) node {$\times$};
   \end{tikzpicture}
  \end{center}
 \caption{The system from Equation~\eqref{eqn:HPsystem}.}
 \end{subfigure}
\end{center}
\caption{A representative of the marked semitoric polygon for each system at $s_1=s_2=1/2$.}
\label{fig:poly_comparison}
\end{figure}
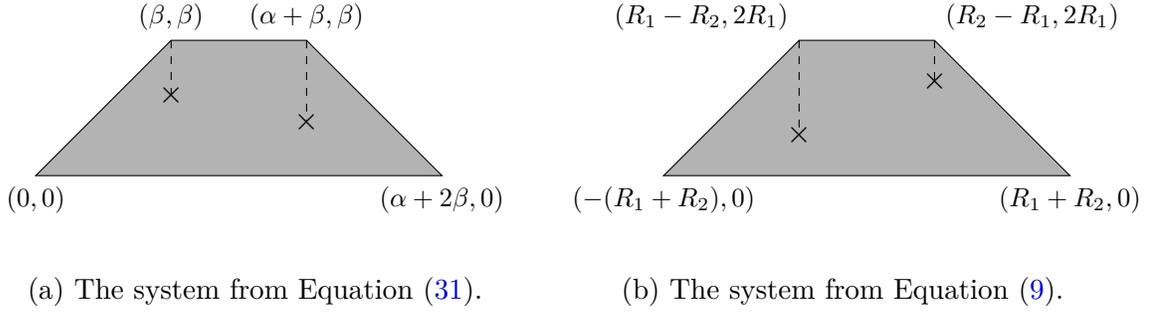

One can check from Equation \eqref{eq:h1_W2} and the expression of $\rho^-$ that the first component $h_1$ of the height invariant of the system on $\Hirzscaled{2}$ is a strictly decreasing function of $\gamma$. Thus, for all but possibly one value of $\gamma$ in the range given in Equation \eqref{eq:range_gamma}, the height invariants are distinct, and the systems $(\Hirzscaled{2}, \omega_{\Hirzscaled{2}},(J,H_{\frac{1}{2},\frac{1}{2}}))$ and $(\S^2 \times \S^2, R_1 \omega_{\S^2} \oplus R_2 \omega_{\S^2}, (J,H_{\frac{1}{2},\frac{1}{2}}))$ are not isomorphic. We display the values of $h_1$ for both systems in Figure \ref{fig:comparison_heights}.

\begin{figure}[H]
\begin{center}
\begin{subfigure}[b]{0.45\textwidth}
\begin{center}
\includegraphics[scale=0.22]{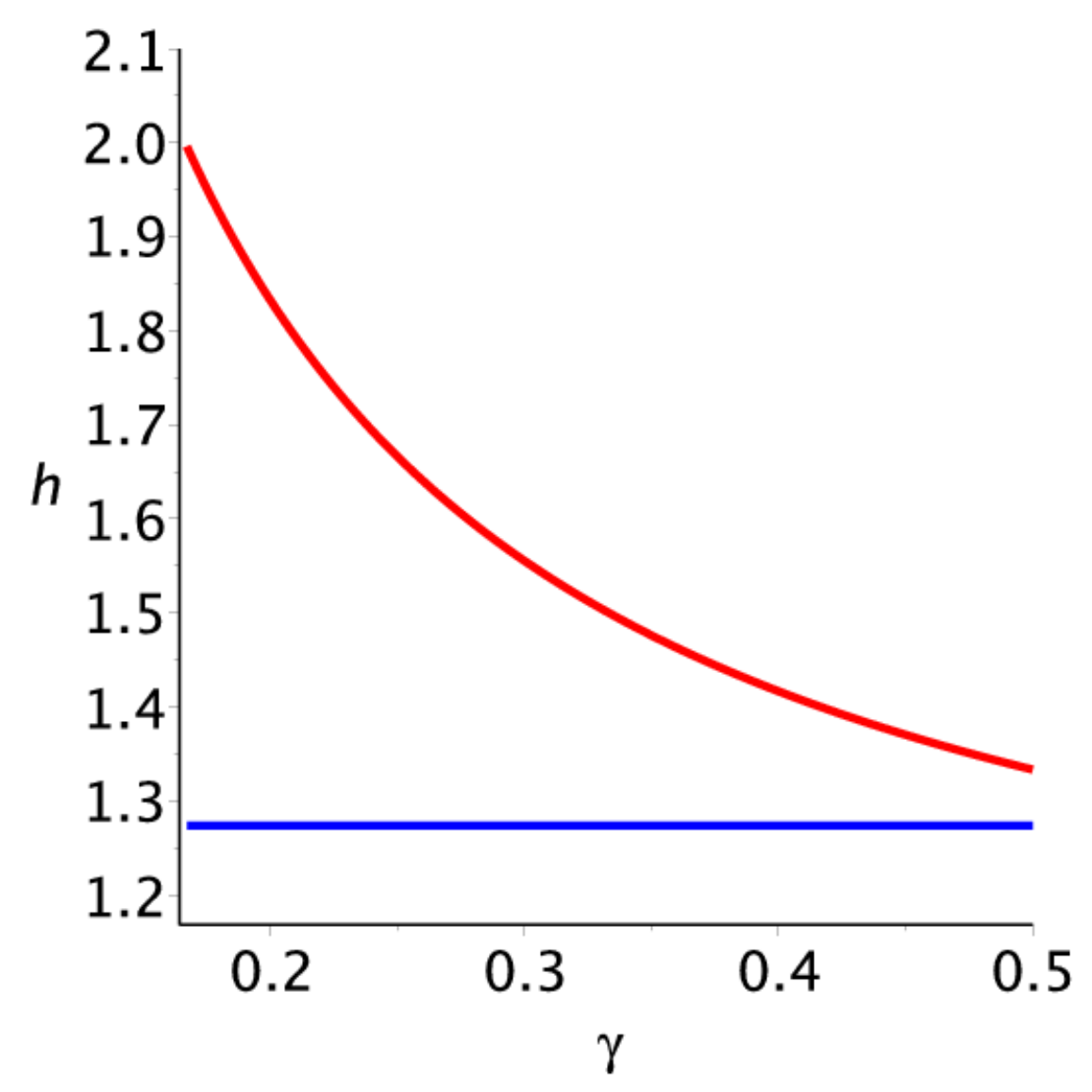}
\end{center}
\caption{$R_1 = 1$, $R_2 = 2$}
\end{subfigure} 
\begin{subfigure}[b]{0.45\textwidth}
\begin{center}
\includegraphics[scale=0.22]{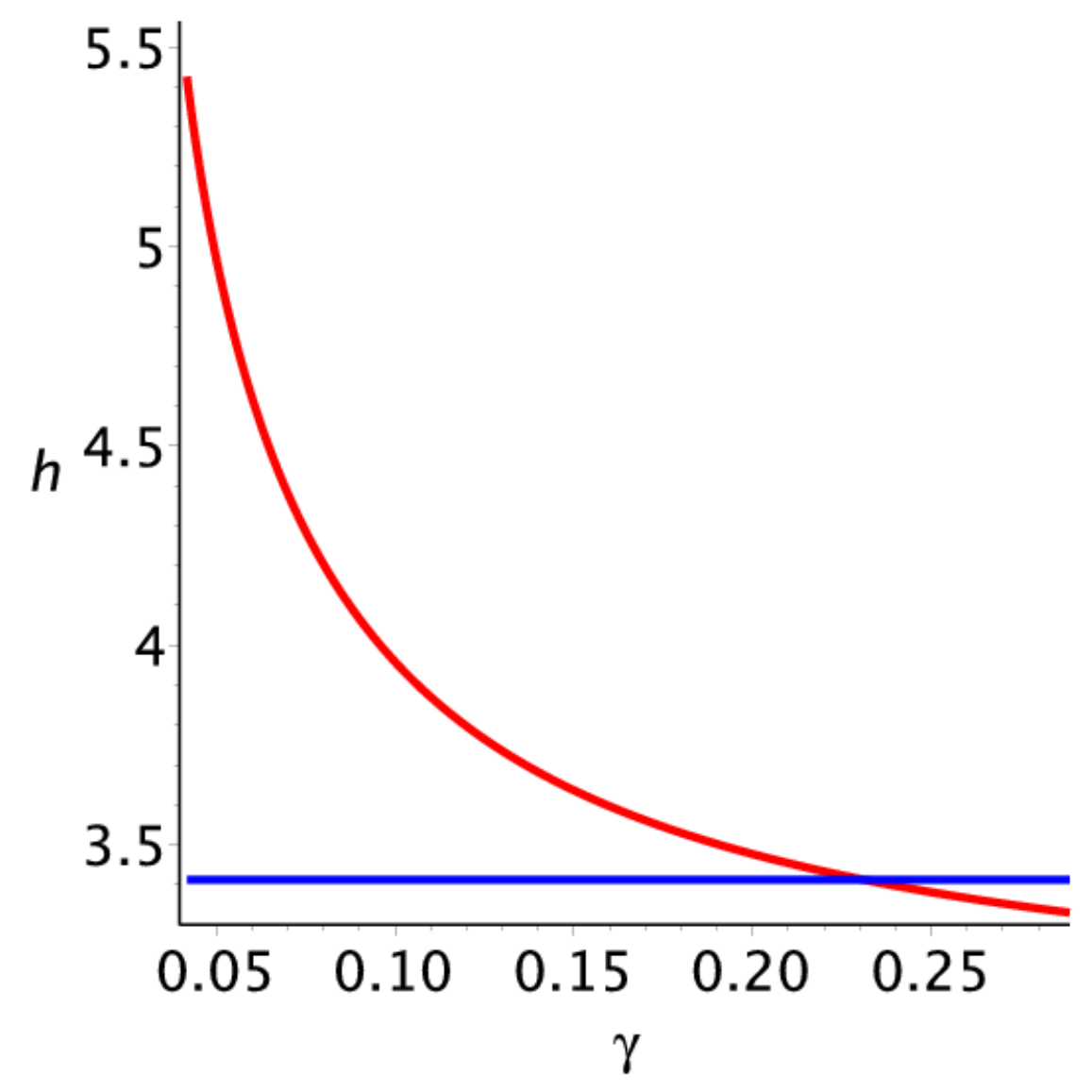}
\end{center}
\caption{$R_1 = 3$, $R_2 = 4$}
\end{subfigure}
\caption{The first component $h_1$ of the height invariant for the systems $(\S^2 \times \S^2, R_1 \omega_{\S^2} \oplus R_2 \omega_{\S^2}, (J,H_{\frac{1}{2},\frac{1}{2}}))$ (blue) and $(\Hirzscaled{2}, \omega_{\Hirzscaled{2}},(J,H_{\frac{1}{2},\frac{1}{2}}))$ for $\gamma$ varying in the range given in Equation \eqref{eq:range_gamma} (red). Here we have fixed $R_1, R_2$ and chosen $\alpha, \beta$ according to Equation \eqref{eq:R_alpha_beta}.}
\label{fig:comparison_heights}
\end{center}
\end{figure}

We see that for some values of $(R_1, R_2)$, there exists one value of $\gamma$ for which the heights invariants of the systems coincide (since $h_2 = 2 R_1 - h_1$ in both cases). It would be interesting to check whether or not the systems are isomorphic for this special value of $\gamma$; this would require to compute the Taylor series or twisting index invariant of these systems, which calls for more work.


\bibliographystyle{abbrv}
\bibliography{st_families}

\end{document}